\documentclass[intlim,righttag,10pt]{amsart}
\usepackage{amscd}
\usepackage{amssymb}
\usepackage[all]{xy}
\RequirePackage{mathrsfs}

\def\invlim{\mathop{\vtop{\ialign{##\crcr$\hfill{\lim}\hfil$\crcr  
\noalign{\kern1pt\nointerlineskip}\leftarrowfill\crcr\noalign
{\kern -3pt}}}}\limits}
\def\dirlim{\mathop{\vtop{\ialign{##\crcr$\hfill{\lim}\hfil$\crcr
\noalign{\kern1pt\nointerlineskip}\rightarrowfill\crcr\noalign
{\kern -3pt}}}}\limits}
\def\lomapr#1{\smash{\mathop{\relbar\joinrel\longrightarrow}\limits^{#1}}}
 \def\verylomapr#1{\smash{\mathop{\relbar\joinrel\relbar\joinrel\relbar\joinrel\longrightarrow}\limits^{#1}}}
\def\veryverylomapr#1{\smash{\mathop{\relbar\joinrel\relbar\joinrel\relbar
\joinrel\relbar\joinrel\relbar\joinrel\longrightarrow}\limits^{#1}}}
\def\epsilon{\varepsilon}
\let\mathcal\mathscr

\newtheorem{theorem}[equation]{Theorem} 
 \newtheorem{lemma}[equation]{Lemma}
 \newtheorem{proposition}[equation]{Proposition}
 \newtheorem{corollary}[equation]{Corollary}

\theoremstyle{definition}
\newtheorem{definition}[equation]{Definition}
\theoremstyle{remark}
\newtheorem{remark}[equation]{Remark}

\newtheorem{example}[equation]{Example}
\newtheorem*{acknowledgments}{Acknowledgments}

\def\cdot{{\scriptscriptstyle\bullet}}

\newcommand{\Qp}{\mathbf{Q}_p}

\renewcommand{\phi}{\varphi}

\newcommand{\R}{\mathrm {R} }

\newcommand{\pri}{^{\prime}}

\newcommand{\ovk}{\overline{K} }

 \newcommand{\coker}{\operatorname{coker} }
 \newcommand{\coim}{\operatorname{coim} }

 \newcommand{\holim}{\operatorname{holim} }
  \newcommand{\rig}{\operatorname{rig} }
 \newcommand{\hocolim}{\operatorname{hocolim} }

  \newcommand{\proeet}{\operatorname{pro\acute{e}t}  }

 \newcommand{\eet}{\operatorname{\acute{e}t} }

 \newcommand{\an}{\operatorname{an} }

 \newcommand{\nr}{\operatorname{nr} }
  
   \newcommand{\conv}{\operatorname{conv} }
 \newcommand{\Spec}{\operatorname{Spec} } 
  \newcommand{\Sp}{\operatorname{Sp} }  
   \newcommand{\Spwf}{\operatorname{Spwf} }

 \newcommand{\Spf}{\operatorname{Spf} }
 
 \newcommand{\Hom}{\operatorname{Hom} }

 \newcommand{\Ext}{\operatorname{Ext} }

 \newcommand{\Gal}{\operatorname{Gal} }
 \newcommand{\tr}{ \operatorname{tr} }
 \newcommand{\can}{ \operatorname{can} }

 \newcommand{\id}{ \operatorname{Id} }
\newcommand{\synt}{ \operatorname{syn} }

\newcommand{\sem}{\operatorname{ss} }
\newcommand{\hk}{\operatorname{HK} } 
\newcommand{\dr}{\operatorname{dR} }

 \newcommand{\crr}{\operatorname{cr} }

 \newcommand{\im}{\operatorname{im} }

 \newcommand{\kr}{^{\scriptscriptstyle\bullet}}

  \newcommand{\Ind}{\operatorname{Ind}} 
  \newcommand{\Pro}{\operatorname{Pro}}

 \newcommand{\sff}{{\mathcal{F}}}
 
 \newcommand{\sy}{{\mathcal{Y}}}
 \newcommand{\sh}{{\mathcal{H}}}
 \newcommand{\sg}{{\mathcal{G}}}
 
 \newcommand{\sv}{{\mathcal{V}}} 
 \newcommand{\srr}{{\mathcal{R}}}
  \newcommand{\su}{{\mathcal{U}}}
 \newcommand{\sbb}{{\mathcal{B}}}

 \newcommand{\so}{{\mathcal O}}
 \newcommand{\sj}{{\mathcal J}}
 
 \newcommand{\sa}{{\mathcal{A}}}
 
 \newcommand{\sx}{{\mathcal{X}}}
 \newcommand{\sss}{{\mathcal{S}}}
\newcommand{\sd}{{\mathcal{D}}}
\newcommand{\sm}{{\mathcal{M}}}

 \newcommand{\wt}{\widetilde}
 \newcommand{\wh}{\widehat}

 \newcommand{\Z}{ {\mathbf Z} }
    
   \newcommand{\Q}{ {\mathbf Q}}
   \newcommand{\N}{{\mathbf N}}
         \newcommand{\rg}{\R\Gamma}

\numberwithin{equation}{section}

\begin{document}
 \title[On  $p$-adic comparison theorems for rigid analytic varieties, I]{On $p$-adic comparison theorems for rigid analytic varieties, I}
 \author{Pierre Colmez}
\address{CNRS, IMJ-PRG, Sorbonne Universit\'e, 4 place Jussieu,
75005 Paris, France}
\email{pierre.colmez@imj-prg.fr}
\author{Wies{\l}awa Nizio{\l}}
  \address{CNRS, UMPA, \'Ecole Normale Sup\'erieure de Lyon, 46 all\'ee d'Italie, 69007 Lyon, France}
 \email{wieslawa.niziol@ens-lyon.fr}
 \thanks{This research was partially supported by the  project  ANR-14-CE25 and the NSF grant No. DMS-1440140.}
\begin{abstract} We compute, in a stable range,  the arithmetic $p$-adic \'etale cohomology of smooth rigid analytic and dagger varieties (without any assumption on the existence of a nice integral model) in terms of differential forms using syntomic methods. The main technical input is a construction of a Hyodo-Kato cohomology and a Hyodo-Kato isomorphism with de Rham cohomology. 

\end{abstract}
 \setcounter{tocdepth}{2}

 \maketitle
 \tableofcontents
 \section{Introduction}Let $p$ be a prime. 
 Let  $\so_K$ be a complete discrete valuation ring of mixed characteristic $(0,p)$ with perfect residue field $k$ and fraction field $K$. Let $F$ be the fraction field of the ring of Witt vectors
 $\so_F=W(k)$ of $k$. Let $\ovk$ be an algebraic closure of $K$ and let $C=\wh{\ovk}$ be its $p$-adic completion; let  $\sg_K=\Gal(\ovk/K)$. Let $F^{\nr}$ be the maximal unramified extension of $F$ in $\ovk$. 
 
   In a joint work with Gabriel Dospinescu \cite{CDN1}, \cite{CDN3} we have computed the $p$-adic (pro-)\'etale cohomology of certain $p$-adic symmetric spaces. A key ingredient of these computations was a one-way (de Rham to \'etale) comparison theorem for rigid analytic varieties over $K$ with a semistable formal model over $\so_K$ that allowed us to pass from (pro-)\'etale cohomology to syntomic cohomology and then to a filtered Frobenius eigenspace associated to  de Rham cohomology. 
   
   The main goal of this paper is to define all the cohomologies that will be necessary for extending such comparison quasi-isomorphisms to all smooth rigid analytic varieties over $K$ or $C$ (without any assumption on the existence of a nice integral model). We will focus on the arithmetic case and leave the geometric case for the sequel of this paper \cite{CN4}. 
   \subsection{Main results}
   We are mainly interested in partially proper  rigid analytic varieties. Since these varieties have a canonical overconvergent (or dagger) structure we are led to study dagger varieties\footnote{Recall that a dagger variety is a rigid analytic variety equipped with an overconvergent structure sheaf. See \cite{GK0} for the basic definitions and properties.}. This is advantageous: for example, a dagger affinoid has de Rham cohomology that is a finite rank vector space with its natural Hausdorff topology while the de Rham cohomology of rigid analytic affinoids is, in general, infinite dimensional and not Hausdorff.
   
   Our first main result is the following theorem: 
 \begin{theorem}
 \label{main00}
 To any smooth dagger  variety $X$ over $L=K,C$ there are  naturally  associated\footnote{All cohomology complexes live in the bounded below derived $\infty$-category of locally convex topological vector spaces over $\Q_p$. Quasi-isomorphisms in this category we call {\em strict quasi-isomorphisms}. }:
 \begin{enumerate}
 \item A pro-\'etale cohomology $\rg_{\proeet}(X,\Q_p(r))$, $r\in \Z$. If $X$ is partially proper this agrees with the pro-\'etale cohomology of $X$ considered  as a rigid analytic variety.
 \item For $L=C$, a $\ovk$-valued rigid cohomology $\rg_{\rig,\ovk}(X)$ and a natural strict quasi-isomorphism\footnote{See Proposition \ref{pierre1} for the definition of the tensor product.}
 $$
 \rg_{\rig,\ovk}(X)\wh{\otimes}^R_{\ovk}C\simeq \rg_{\dr}(X).
 $$
 This defines a natural $\ovk$-structure on the de Rham cohomology\footnote{By the same procedure one can define a $F^{\nr}$-valued rigid cohomology $\rg_{\rig,F^{\nr}}(X)$ and a natural strict quasi-isomorphism
  $ \rg_{\rig,F^{\nr}}(X)\wh{\otimes}^R_{F^{\nr}}C\simeq \rg_{\dr}(X).
 $}.
\item A Hyodo-Kato cohomology 
$\rg_{\hk}(X)$. This is a dg $F$-algebra if $L=K$,
and   a dg $F^{\nr}$-algebra if $L$=$C$,
equipped with a Frobenius $\phi$ and a monodromy operator $N$. For $L=C$, we have  natural Hyodo-Kato strict quasi-isomorphisms
$$
 \iota_{\hk}: \rg_{\hk}(X)\wh{\otimes}_{F^{\nr}}\ovk\stackrel{\sim}{\to} \rg_{\rig,\ovk}(X),\quad  \iota_{\hk}: \rg_{\hk}(X)\wh{\otimes}^R_{F^{\nr}}C\stackrel{\sim}{\to} \rg_{\dr}(X).
 $$
 \item For $L=K$, a syntomic cohomology $\rg_{\synt}(X,\Q_p(r))$, $r\in \N$, that fits into a distinguished triangle
  \begin{equation}
  \label{triangle}
 \rg_{\synt}(X,\Q_p(r))\lomapr{}  [\rg_{\hk}(X)]^{N=0,\phi=p^r}\lomapr{\iota_{\hk}} \rg_{\dr}(X)/F^r,
 \end{equation}
and  a natural period morphism
 $$
 \alpha_r: \rg_{\synt}(X,\Q_p(r))\to \rg_{\proeet}(X,\Q_p(r))
 $$
 that is a strict quasi-isomorphism after truncation $\tau_{\leq r}$. 
 \end{enumerate}
 \end{theorem}
 We also prove an analogous theorem for smooth rigid analytic varieties. 
 
 The second main result of this paper is the following corollary of Theorem \ref{main00}. 
 \begin{theorem}\label{first3.1}
  Let $X$ be a smooth dagger variety over $K$ and let $r\geq 0$. 
    \begin{enumerate}
    \item 
    For $1\leq i\leq r-1$, the boundary map induced by the distinguished triangle (\ref{triangle})
    $$
    \partial_{r}: \wt{H}^{i-1}_{\dr}(X)\to \wt{H}^{i}_{\proeet}(X,\Q_p(r))
    $$
    is an isomorphism. In particular, the cohomology $ \wt{H}^{i}_{\proeet}(X,\Q_p(r))$ is classical and it has a natural $K$-structure. 
    \item We have  long exact sequences
    \begin{align*}
   &  0\to \wt{H}^{r-1}(\rg_{\dr}(X)/F^r)\lomapr{\partial_r} \wt{H}^r_{\proeet}(X,\Q_p(r))\to \wt{H}^{r}([\rg_{\hk}(X)]^{N=0,\phi=p^r})\lomapr{\iota_{\hk}}\wt{H}^r(\rg_{\dr}(X)/F^r)\\
   & 
0\to  \wt{H}^{r-1}_{\hk}(X)^{\phi=p^{r-1}}  \to \wt{H}^{r}([\rg_{\hk}(X)]^{N=0,\phi=p^r})\to \wt{H}^{r}_{\hk}(X)^{N=0,\phi=p^r}\to 0
    \end{align*}
    Moreover, the cohomology $ \wt{H}^{i}_{\hk}(X)$ is classical. 
    \end{enumerate}
    \end{theorem}
    Here $\wt{H}$ refers to cohomology taken in the derived category of locally convex topological vector spaces over $\Q_p$ and ``classical" means that the cohomology $\wt{H}$ is isomorphic to the algebraic cohomology $H$ equipped with its natural quotient topology (very often this is equivalent to the natural topology on $H$ being separated).  If $X$ is proper,  we have the   isomorphisms
  $$
  H^{r-1}_{\dr}(X)\stackrel{\sim}{\to} \wt{H}^{r-1}(\rg_{\dr}(X)/F^r),\quad  H^{r}_{\dr}(X)/\Omega^r(X)\stackrel{\sim}{\to} \wt{H}^{r}(\rg_{\dr}(X)/F^r).
  $$
If  $X$ is Stein,  we get the  isomorphisms 
  \begin{align*}
  \wt{H}^{r-1}(\rg_{\dr}(X)/F^r)\simeq  \Omega^{r-1}(X)/\im d_{r-1},\quad 
   \wt{H}^i(\rg_{\dr}(X)/F^r)\simeq 0,\, i\geq r.
  \end{align*}
  Hence the cohomology $\wt{H}^{r-1}(\rg_{\dr}(X)/F^r)$  is classical.

    We prove an analogous result in the case of smooth rigid analytic varieties over $K$ and this generalizes the computations \cite[Cor. 3.16]{CN} done for smooth affinoids with semistable reduction.
\begin{remark}
For a smooth proper scheme $X$ over $K$, the analog of the  map $
    \partial_{r}: \wt{H}^{i-1}_{\dr}(X)\to \wt{H}^{i}_{\proeet}(X,\Q_p(r))$ is a geometric incarnation of the Bloch-Kato exponential. See \cite[Remark 2.14]{NN}, \cite[Prop. 3.8]{DN}, \cite[Th. 3.1]{Ni} for a detailed discussion. 
\end{remark}
 \subsection{Proof of Theorem \ref{main00}}
 We will now sketch  how Theorem \ref{main00} is proved. The pro-\'etale cohomology in (1) is defined in the most naive way: if $X$ is a smooth dagger affinoid with a presentation $\{X_h\}_{h\in\N}$ by a pro-affinoid rigid analytic variety\footnote{See Section \ref{pierre11} for the definition of presentations.} we set 
 $\rg_{\proeet}(X,\Q_p(r)):=\hocolim_h\rg_{\proeet}(X_h,\Q_p(r))$; then we globalize. From this description it is clear that we have a natural map $\rg_{\proeet}(X,\Q_p(r))\to
 \rg_{\proeet}(\wh{X},\Q_p(r))$, where $\wh{X}$ is the completion of $X$ (a rigid analytic variety). 
 
 For the rest of Theorem \ref{main00}, first we show that, using the rigid analytic \'etale local alterations of Hartl and Temkin \cite{Urs}, \cite{Tem},  the \'etale topology on $X_{L}$ has a base consisting of semistable weak formal schemes (always assumed to be of finite type) over finite extensions of $\so_K$. This allows us to define sheaves by specifying them on such integral models and then sheafifying for the $\eta$-\'etale topology\footnote{This construction mimics that of Beilinson in \cite{BE2} done for algebraic varieties; here $\eta$-\'etale means topology induced from the \'etale topology of the generic fiber.}. For example, for (2), we define  $\rg_{\rig,\ovk}(X):=\rg_{\eet}(X,\sa_{\rig,\ovk})$, for a sheaf $\sa_{\rig,\ovk}$ induced from a 
 presheaf  assigning to a semistable model $\sy$ over $\so_C$ coming by base change from a semistable model $\sy_{\so_E}$ over $\so_E$, $[E:K]<\infty$,  the complex\footnote{We give here a rough definition; see Section \ref{lebras} for a precise definition.} $\hocolim\rg_{\rig}(\sy_{\so_E,0})$,  $\sy_{\so_E,0}$ is the special fiber of $\sy_{\so_E}$, where the homotopy colimit is taken over   such models $\sy_{\so_E}$. In an analogous way we define, for (3), 
  the Hyodo-Kato cohomology using the overconvergent  Hyodo-Kato cohomology of Grosse-Kl\"onne that for a semistable model $\sy$ over $\so_K$ is defined as $\rg_{\hk}(\sy_0):=\rg_{\rig}(\sy_0/\so_F^0)$; the Hyodo-Kato quasi-isomorphism is induced from the one defined by Grosse-Kl\"onne $\iota_{\hk}: \rg_{\rig}(\sy_0/\so_F^0)\stackrel{\sim}{\to}\rg_{\rig}(\sy_0/\so_F^{\times})$. Here $\so_K^{\times}$, $\so_K^0$ denote the (weak formal) scheme associated to $\so_K$ with the canonical and the induced by $\N\to \so_K$, $1\mapsto 0$, log-structure, respectively. 
  
  We define the syntomic cohomology in (4)  in two different, but (non obviously) equivalent, ways. One definition is just as a homotopy fiber that yields the distinguished triangle (\ref{triangle}). The other, for dagger affinoids with a presentation $\{X_h\}_{h\in\N}$, sets $\rg_{\synt}(X,\Q_p(r)):=\hocolim_h\rg_{\synt}(X_h,\Q_p(r))$. Here the syntomic cohomology $\rg_{\synt}(X_h,\Q_p(r))$ of a rigid analytic variety $X_h$ is defined by  $\eta$-\'etale descent, using the fact that semistable formal models form a base for the \'etale topology of $X$,  from the crystalline syntomic cohomology of Fontaine-Messing. Recall that the latter is defined as   the homotopy fiber 
  $\rg_{\synt}(\sx,\Q_p(r):=[F^r\rg_{\crr}(\sx)\lomapr{\phi-p^r}\rg_{\crr}(\sx)]$, where the crystalline cohomology is absolute (i.e., over $\Z_p$). The second definition works also for smooth dagger varieties over $C$.

   It is quite nontrivial to show that these two definitions agree. Along the way, we prove the main technical result of this paper: 
   \begin{theorem}
   \label{first4}Let $r\geq 0$. Let $X$ be 
a smooth dagger variety over $K$. There is a natural morphism $$\rg_{\synt}(X,\Q_p(r))\to \rg_{\synt}(\wh{X},\Q_p(r)).$$ 
It is a strict quasi-isomorphism if $X$ is partially proper. 
   \end{theorem}This theorem is proved  by representing both sides of the morphism by means of  the crystalline and the overconvergent Hyodo-Kato cohomology, respectively, then passing via Galois descent to $X_C$, and finally passing  through the crystalline and overconvergent Hyodo-Kato quasi-isomorphisms (that need to be shown to be compatible) to the de Rham cohomology, where the result is known.

   To define the period map in (4), for $L=K,C$,  we first define it for rigid analytic varieties by the $\eta$-\'etale descent of the Fontaine-Messing period map $\alpha_r: \rg_{\synt}(\sx,\Q_p(r))\to \rg_{\eet}(\sx_L,\Q_p(r))$, for a semistable formal scheme $\sx$ over $\so_L$. Then we use the second  definition of syntomic cohomology and   the period maps $\alpha_r: \rg_{\synt}(X_h,\Q_p(r))\to \rg_{\eet}(X_h,\Q_p(r))$ to get the period map $\alpha_r$ in Theorem \ref{main00}. The fact that it is a strict quasi-isomorphism in a stable range  follows from the computations of $p$-adic nearby cycles via syntomic complexes done in \cite{Ts} in the geometric case and in  \cite{CN} in the arithmetic case.
   \begin{remark}For
   an algebraic variety $X$ over $L=K,C$,  a well behaved syntomic cohomology $\rg_{\synt}(X,\Q_p(r))$, $r\geq 0$, was defined in \cite{NN}. A more conceptual definition was given in \cite{DN} but the approach in \cite{NN} is more concrete and this is the one we mimic  in this paper.
   For $L=K$ and smooth $X$, there exists a natural map $\rg_{\synt}(X,\Q_p(r))\to \rg_{\synt}(X^{\an},\Q_p(r))$, where $X^{\an}$ denotes the analytification of $X$. This should be a strict quasi-isomorphism if $X$ is proper although we do not prove this in this paper. 
   \end{remark}
   \begin{remark}
   Let $\sx$ be a proper  semistable scheme over $\so_K$ (we allow a horizontal divisor at infinity). Ertl-Yamada \cite{EY}  have extended  Grosse-Kl\"onne's  definition of the Hyodo-Kato morphism to this setting and defined the corresponding rigid syntomic cohomology by the defining property (\ref{triangle}).  See \cite{Yam} for a more conceptual definition in the case when there is no horizontal divisor at infinity. 
   \end{remark}
  \begin{acknowledgments}W.N. would like to thank MSRI, Berkeley, for hospitality during Spring 2019 semester when parts of this paper were written. We would like to thank 
 Benjamin Antieau, Antoine Chambert-Loir, Antoine Ducros, Veronika Ertl, and  Luc Illusie  for helpful discussions concerning the content of this paper. We thank the referees for a careful reading of the paper and helpful comments. 
 \end{acknowledgments}

 \subsubsection{Notation and conventions.} All formal schemes are $p$-adic. For a (weak formal or formal) scheme $\sx$ over $\so_K$, we will denote by $\sx_n$ its reduction modulo $p^n$, $n\geq 1$, and  by $\sx_0$ its special fiber. 
 
 We will denote by $\so_K$,
$\so_K^{\times}$, and $\so_K^0$, depending on the context,  the scheme $\Spec ({\so_K})$ or the formal scheme $\Spf (\so_K)$ with the trivial, the canonical (i.e., associated to the closed point), and the induced by $\N\to \so_K, 1\mapsto 0$,
log-structure, respectively.

 \begin{definition}
Let $N\in {\mathbf N}$. For a morphism $f: M\to M^{\prime}$ of ${\mathbf Z}_p$-modules, we say that $f$ is 
{\it $p^N$-injective} (resp. {\it $p^N$-surjective}) if its kernel (resp. its cokernel) is annihilated by $p^N$ 
and we say that $f$ is a  {\it $p^N$-isomorphism} if it is $p^N$-injective and $p^N$-surjective. 
We define in the same way the notion of {\it $p^N$-distinguished triangle} or {\it $p^N$-acyclic complex} 
(a complex whose cohomology groups are annihilated by $p^N$) as well as the notion of {\it $p^N$-quasi-isomorphism}
 (map in the derived category that induces a $p^N$-isomorphism on cohomology). 
\end{definition}

  Unless otherwise stated, we work in the derived (stable) $\infty$-category $\sd(A)$ of left-bounded complexes of a quasi-abelian category $A$ (the latter will be clear from the context).
  Many of our constructions will involve (pre)sheaves of objects from $\sd(A)$. The reader may consult the notes of Illusie \cite{IL} and Zheng \cite{Zhe} for a brief introduction to how to work with such (pre)sheaves  and \cite{Lu1}, \cite{Lu2} for a thorough treatment.
  
    We will use a shorthand for certain homotopy limits. Namely,  if $f:C\to C'$ is a map  in the  derived $\infty$-category of a quasi-abelian category, we set
$$[\xymatrix{C\ar[r]^f&C'}]:=\holim(C\to C^{\prime}\leftarrow 0).$$ 
And we set
$$
\left[\begin{aligned}
\xymatrix@R=.5cm{C_1\ar[d]\ar[r]^f & C_2\ar[d]\\
C_3\ar[r]^g & C_4
}\end{aligned}\right]
:=[[C_1\stackrel{f}{\to} C_2]\to [C_3\stackrel{g}{\to} C_4]],
$$ 
for 
 a commutative diagram (the one inside the large bracket) in the derived $\infty$-category of a quasi-abelian category.
 \section{An equivalence of topoi}
 Let $X$ be a smooth rigid analytic variety over $K$, resp. $C$. In this section, we will show that the \'etale site of $X$ has a base (in the sense of Verdier, see \cite{V2}) built from semistable formal schemes over  finite extensions of $\so_K$, resp. over $\so_C$. We will show the same for smooth dagger spaces over $K$ and $C$.
 \subsection{A general criterium}
 \label{topoi}In \cite[2.1]{Be1}  Beilinson generalized a well-known criterium of Verdier \cite[4.1]{V2} stating  conditions under which one  can change sites while preserving their topoi. While Verdier assumed the functor $F$ below to be fully faithful, Beilinson allows it to be just faithful.

 We will briefly summarize  \cite[2.1]{Be1}.   Let $\sv$ be an essentially small site and let  ${\rm Sh}({\sv})$ be the corresponding topos. 
 {\em A base for} $\sv$ is a pair $(\sbb,F)$, where $\sbb$ is an essentially small category and  $F: \sbb\to \sv$ is a faithful functor, which satisfies the following property: \vspace{3mm}
 \\
($\star$) For  $V\in\sv$ and a finite family of pairs $(B_{\alpha},f_{\alpha}), B_{\alpha}\in \sbb, f_{\alpha}: V\to F(B_{\alpha}),$ there exists a set of objects $B^{\prime}_{\beta}\in \sbb$ and a covering family $\{F(B^{\prime}_{\beta})\to V\}$ such that each composition $F(B^{\prime}_{\beta})\to V\to F(B_{\alpha})$ lies in $\Hom(B_{\beta}^{\prime},B_{\alpha})\subset \Hom(F(B^{\prime}_{\beta}),F(B_{\alpha}))$. 
 \begin{remark}
 \label{kolo2}
 \begin{enumerate}
 \item  For the empty set of $(B_{\alpha},f_{\alpha})$'s the above means that every  $V\in\sv$ has a covering by objects $F(B),B\in\sbb$. If $F$ is fully faithful, then ($\star$) is equivalent  to this assertion. 
 \item  If $\sbb$ admits finite products and $F$ commutes with finite products, then it suffices to check ($\star$) for families $(B_{\alpha},f_{\alpha})$ having $\leq 1$ elements.
 \item In the general case, it suffices to check ($\star$) for families $(B_{\alpha},f_{\alpha})$ having $\leq 2$ elements.
 \end{enumerate}
 \end{remark}
 
  Let  $(\sbb,F)$ be a base for $\sv$. Define a covering sieve in $\sbb$ as a sieve whose $F$-image is a covering sieve in $\sv$. The following proposition is proved by Beilinson \cite[2.1]{Be1}.
   \begin{proposition}
   \begin{enumerate}
   \item  Covering sieves in $\sbb$ form a Grothendieck topology on $\sbb$.
   \item The functor $F:\sbb\to\sv$ is continuous.
   \item $F$ induces an equivalence of  topoi ${\rm Sh}(\sbb)\stackrel{\sim}{\to}{\rm Sh}(\sv)$.
   \end{enumerate}
   We call the above topology on $\sbb$ the $F$-induced topology.
   \end{proposition}
   \begin{remark}
   \label{kolo1}
   \begin{enumerate}
   \item If $F$ is fully faithful, the above proposition is \cite[4.1]{V2}.
   \item Let $\xymatrix@C=0.5cm{(F^s,F_s): {\rm Sh}(\sbb)\ar@<-0.5mm>[r]&{\rm Sh}(\sv)\ar@<-0.5mm>[l]}$ be the usual adjoint functors.  For a presheaf $\sff$ on $\sv$, we have $F_s(\sff^{ a})=F_p(\sff)^{a }$,
   where $F_p$ is the pushforward of presheaves and the subscript $a$ means ``associated sheaf".
   \item
If $(\sbb,F)$ is a base for $\sv$ and $(\sbb^{\prime},F^{\prime})$ is a base for the $F$-induced topology on $\sbb$ then $(\sbb^{\prime},FF^{\prime})$ is a base for $\sv$. 
   \end{enumerate}
   \end{remark}
   \subsection{Categories of formal models}
   \label{models}
   We will show now that  the \'etale site of smooth rigid analytic varieties over $K$, resp.  over  $C$, admits a base built from semistable formal schemes over finite extensions of $\so_K$, resp. over $\so_C$.
   \subsubsection{Models.}
   Let $L=K,C$. A morphism of $\so_L$-schemes $f:Y\to X$ is called $\eta$-\'etale, an $\eta$-isomorphism, etc., if  its generic fiber $f_L$ is \'etale, an isomorphism, etc..  An $\so_L$-scheme is {\em admissible} if it is flat and of finite type over $\so_L$. 
   A formal $\so_L$-scheme $\sx$ is {\em admissible} if it is flat and of finite type over $\Spf(\so_L)$. For an admissible formal $\so_L$-scheme $\sx$, we denote by $\sx_L$ (or $\sx_{\eta}$) its rigid analytic generic fiber.  We say that a morphism $\sy\to\sx$ between admissible formal $\so_L$-schemes is {\em $\eta$-\'etale} if its generic fiber $f_L$ (or $f_{\eta}$) is \'etale. Similarly, we define {\em $\eta$-smooth} morphisms\footnote{In a more traditional language we would call such morphisms ``$\rig$-\'etale", etc. However, since it is becoming standard to use $\eta$ to denote the rigid generic fiber, we have elected to use $\eta$-\'etale in  this paper.}.

    Let ${\rm Sm}_L$ be the category of smooth $L$-rigid varieties. We will consider categories $\sm$ formed by  semistable  formal models of such varieties.\\
   (a) {\em  $K$-setting}:  A {\em   model over} $K$ (a {\em  $K$-model}) is an admissible formal $\so_K$-scheme $\sx$. 
   A formal scheme over $\so_K$ is called {\em  semistable} if, locally for the Zariski topology,  it admits an \'etale morphism to a formal scheme of the form
    $$\Spf(\so_K\{X_1,\ldots,X_l\}/(X_1\cdots X_m-\varpi)),\quad 0\leq m\leq l,
    $$
    for a uniformizer $\varpi$ of $\so_K$ (we allow $m=0$ just to get formal affine space -- when the formal scheme is smooth). 
    A $K$-model $\sx$ is called {\em semistable}  if it is semistable over $\so_E$ for a finite field extension $E$ of $K$. In that case, assume that $\sx_K$ is connected (which is equivalent to $\sx$ being connected) and let $K_{\sx}$ be the algebraic closure of $K$ in $\Gamma(\sx_K,\so_{\sx_K})$ (note that $E\subset K_{\sx}$). Then $\so_{K_{\sx}}$ is the integral closure of $\so_K$ in $\Gamma(\sx,\so_{\sx})$ and $\sx$ is semistable over $\so_{K_{\sx}}$. 
We will say that $\sx$ is  {\em split} over $K_{\sx}$.

  Let $\sm_K$ denote the category of  $K$-models (morphisms are morphisms of formal schemes over $\so_K$) and  let $\sm^{\sem}_K$ be its full subcategory of semistable $K$-models. 
  
     (b) {\em  $C$-setting}: A {\em  model over $C$} (a {\em   $C$-model}) is an admissible formal $\so_C$-scheme $\sx$.   It is 
         called {\em  semistable}, if locally for the Zariski topology, 
      it admits an \'etale morphism to a formal scheme of the form
    $$\Spf(\so_C\{X_1,\ldots,X_l\}/(X_1\cdots X_m-\varpi)), \quad 0\leq m\leq l,
    $$
for $0\neq \varpi\in \so_C$. 
      It is called {\em basic semistable}
     if there exists a semistable model $\sx^{\prime}$ over $\so_E$, $E$ a finite extension of $K$,  and a $C$-point $\alpha: E\to C$ such that $\sx$ is isomorphic to the base change $\sx^{\prime}_{\so_C}$. 
     Let $\sm_C$ denote the category of  $C$-models and  let $\sm^{\sem}_C$, $\sm^{\sem,b}_C$ be its full subcategories  of semistable  and basic semistable $C$-models, respectively. 
     
 We note that, if we equip the formal schemes in $\sm^{\sem}_K$, $\sm^{\sem,b}_C$, and $\sm^{\sem}_C$ with the log-structure associated to the special fiber over the ring over which they split, every map in these categories is a map of log-schemes. Warning: the maps in the category $\sm^{\sem,b}_C$ {\em do not} have to come from finite levels.

  The $K$- and $C$-settings are connected by  the base change functors
\begin{equation}
\label{base-change}
\xymatrix@R=.5cm{
\sm^{\sem,b}_C\ar[r]& {\rm Sm}_C\\
\sm^{\sem}_K\ar[r]\ar[u]& {\rm Sm}_K,\ar[u]
}
\end{equation}
where the right vertical arrow is the base change $(-)\wh{\otimes}_KC$ and the left arrow assigns to a  $K$-model $\su$ semistable over $\so_L$, $L$ a finite extension of $K$, 
the  disjoint union of semistable models $\su\wh{\otimes}_{\so_L, \alpha}\so_C$ over $C$-points $\alpha: L\to C$.
\subsubsection{Semistable reduction.}
We say that an admissible formal $\so_L$-scheme $\sx$ is {\em algebraizable} if it is isomorphic to the $p$-adic completion of an admissible $\so_L$-scheme $X$. 
   The well-known algebraization theorem of Elkik \cite{Elk} yields the following theorem.
   \begin{theorem}{\rm (Temkin, \cite[Th. 3.1.3]{Tem})}\label{algebraization}
   Any affine $\eta$-smooth admissible formal $\so_L$-scheme $\sx$ is algebraizable. Moreover, we can find an affine $\eta$-smooth admissible $\so_L$-scheme $X$ such  that $\sx\simeq \wh{X}$. 
   \end{theorem}

 We quote  two results of Temkin which generalize  results of Hartl \cite[Th. 1.4]{Urs} (which works for complete discretely-valued fields) and Faltings \cite[III.2]{Fal} (see \cite[Th. 2.5.2]{Tem} for an algebraic analog and  \cite{BS} for a refined algebraic analog).
       \begin{theorem}{\rm (Temkin, \cite[Th. 3.3.1]{Tem})}\label{Tem1}
    Let $\sx$ be an $\eta$-smooth admissible formal scheme over $\so_L$. Then there exists a finite field extension  $E/L$ and a $\eta$-\'etale covering $\sx^{\prime}\to\sx\otimes_{\so_L}\so_E$ such that $\sx^{\prime}$ is  semistable over $\so_E$.
    \end{theorem}
    \begin{corollary}{\rm (Temkin, \cite[Cor. 3.3.2]{Tem})}\label{Tem11}
    Let  $X$ a smooth qcqs rigid space over $L$. Then there exists a finite extension $E/L$ and an \'etale covering $X^{\prime} \to X\otimes_LE$ such that $X^{\prime}$ is affinoid and has  a  semistable affine formal model.
    \end{corollary}
    \begin{proof}
    Take an admissible formal model $\sx$ of $X$ (such a model exists  by a theorem of Raynaud \cite[Th. 4.1]{Ray}). 
    Take $E/L$ and $\sx^{\prime} \to \sx\otimes_{\so_L}{\so_E}$ as in Theorem \ref{Tem1}. We can refine  $\sx^{\prime}$ to make it affine. Then its generic fiber $\sx^{\prime}_{E}$ is affinoid and has $\sx^{\prime}$ for a  semistable model.
    \end{proof}
\subsubsection{An equivalence of topoi.}
Let  $\sm$ be any category from Section \ref{models} and let $F_{\eta}$ be the forgetful functor $\sx\mapsto\sx_{\eta}$.  The main result of this section is the following
\begin{proposition}
\label{Bonn1}
If $\sm$ is  the category $\sm_K$ or $\sm^{\sem}_K$ then $(\sm,F_{\eta})$ is a base for ${\rm Sm}_{K,\eet}$. If $\sm$ is $\sm_C$, $\sm^{\sem,b}_C$,  or $\sm^{\sem}_C$ then $(\sm,F_{\eta})$ is a base for ${\rm Sm}_{C,\eet}$. 
\end{proposition}
\begin{proof}
Consider first the $K$-setting. We need to show that $\sm_K$ satisfies condition $(\star)$  from Section \ref{topoi}. For that, assume that  $X$ is a rigid analytic variety over $K$ and take a finite family\footnote{By Remark \ref{kolo2}, we may assume that this family consists of one element.} of  $K$-models $\su_{\alpha}$ together with  maps $f_{\alpha}: X\to \su_{\alpha,K}$. We need to find an \'etale covering $\pi: X^{\prime}\to X$ and a  $K$-model $\sx^{\prime}$  of $X^{\prime}$ such that every map $f_{\alpha}\pi$ extends to a map $\sx^{\prime}\to \su_{\alpha}$. 

 Replacing $X$ by an  affinoid admissible covering, we may assume that $X$ is a disjoint union of affinoids. By a theorem of Raynaud \cite[Th. 4.1]{Ray}, we can find a $K$-model of  $X$. By \cite[Lemma 5.6]{FRG}, 
 this model can be modified by  an admissible blow-up to a $K$-model $\sx$  of $X$ such that there exists a dotted arrow that makes the following diagram commute
$$
 \xymatrix{
 \sx\ar@{.>}[d]  & X\ar[d]^{\prod_{\alpha}f_{\alpha}}\ar@{_(-->}[l]\\
  \prod_{\alpha}\su_{\alpha} & \prod_{\alpha}\su_{\alpha,K}\ar@{_(-->}[l] 
   }
 $$ 
This is the model we wanted. 
 
  Now, to show that $(\sm^{\sem}_K,F_{\eta})$ is a base it suffices, by Remark \ref{kolo1},  to show that $(\sm^{\sem}_K, \iota)$, for the natural functor $\iota: \sm^{\sem}_K\hookrightarrow \sm_K$,  is a base of $\sm_K$. Since $\iota$ is fully faithful, by Remark \ref{kolo2},  it suffices to check that, for every $K$-model $\su\in \sm_K$, there exists a map of $K$-models $\su^{\prime}\to \su$ such that $\su^{\prime}_K\to\su_K$ is \'etale and $\su^{\prime}$ is semistable. But this follows from Theorem \ref{Tem1}.
  
   For the  $C$-setting the argument is analogous in the case of $\sm_C$ and $\sm^{\sem}_C$.  For $\sm^{\sem,b}_C$,  since $\sm^{\sem,b}_C\hookrightarrow \sm^{\sem}_C$ is fully faithful, by Remark \ref{kolo2},  it suffices to check that, for every $C$-model $\su\in \sm^{\sem}_C$, there exists a map of $C$-models $\su^{\prime}\to \su$ such that $\su^{\prime}_C\to\su_C$ is \'etale and $\su^{\prime}$ is basic semistable. But this can be achieved by taking for $\su^{\prime}$ a log-blow-up of $\su$ (see \cite[Lemma 1.11]{Sai}).
\end{proof}   
We call the topology induced by $F_{\eta}$ on the categories $\sm$ the {\em $\eta$-\'etale} topology. The functors in (\ref{base-change}) are continuous for the respective \'etale  topologies. 
By Section \ref{topoi} and Proposition \ref{Bonn1}, $F_{\eta}$ identifies \'etale sheaves on ${\rm Sm}_K$, resp. ${\rm Sm}_C$, with $\eta$-\'etale sheaves on $\sm_K$, $\sm^{\sem}_K$, resp. $\sm_C$, $\sm^{\sem,b}_C$, $\sm^{\sem}_C$. We obtain the {\em \'etale localization} functors
    $$
    {\rm Psh}(\sm^?_K)\to {\rm Sh}({\rm Sm}_{K,\eet}),\quad  {\rm Psh}(\sm^?_C)\to {\rm Sh}({\rm Sm}_{C,\eet}),
    $$
    which assign to any presheaf $\sff$ on models the corresponding \'etale sheaf $\sff^{\sim}$ viewed as an \'etale sheaf on varieties. 
    \begin{remark}
    For any presheaf on $\sm_K$ or $\sm_C$, its $\eta$-\'etale sheafification is the same as the $\eta$-\'etale sheafification of its restriction to resp. $\sm^{\sem}_K$ or $\sm^{\sem,b}_C$, $\sm^{\sem}_C$. 
    \end{remark}
    \begin{remark}
    In this paper we will use over and over again the following procedure to define an \'etale  sheaf $\sff$  on, say, ${\rm Sm}_K$. 
    \begin{enumerate}
    \item {\rm ({\em Local definition})}: We define a functorial  $\sff(Y)$, $Y\in\sm^{\sem}_K$.
    \item {\rm ({\em Globalization})}: We sheafify the so defined presheaf in $\eta$-\'etale topology. This yields an \'etale  sheaf $\sff$ on ${\rm Sm}_K$ (this notation is slightly abusive but hopefully will not cause problems in understanding).
    \item {\rm ({\em Local-global compatibility})}: We will often need to know that we have $\eta$-\'etale descent, i.e., that, for $Y\in\sm^{\sem}_K$,  the natural map $\sff(Y)\to \rg_{\eet}(Y_K,\sff)$ is a quasi-isomorphism. 
    \end{enumerate}
    \end{remark}
    \subsection{Categories of weak formal models}\label{dagger-topoi}
       In this section, we will show  that  the \'etale site of smooth dagger varieties\footnote{For basics on dagger (or overconvergent) varieties we refer the reader to \cite{GK0}.} over $K$, resp.  over  $C$, admits a base built from semistable weak formal schemes over finite extensions of $\so_K$, resp. over $\so_C$.
\subsubsection{Models}\label{dagger-models}
    Let $L=K,C$. 
   A weak formal $\so_L$-scheme $\sx$ is {\em admissible} if it is flat and of finite type over $\so_L$. For an admissible weak  formal $\so_L$-scheme $\sx$, we denote by $\sx_L$ (or $\sx_{\eta}$) its dagger generic fiber.  We say that a morphism $f:\sy\to\sx$ between admissible weak  formal $\so_L$-schemes is {\em $\eta$-\'etale} if its generic fiber $f_L$ (or $f_{\eta}$) is \'etale. Similarly, we define {\em $\eta$-smooth} morphisms. 
   
    Let ${\rm Sm}^{\dagger}_L$ be the category of smooth $L$-dagger varieties. We define the  categories $\sm^{\dagger}_L, \sm^{\dagger,\sem,b}_C,$ and $ \sm^{\dagger,\sem}_L$ formed by  weak  formal models, basic semistable, and semistable weak formal models\footnote{Semistable weak formal schemes are defined by the same formulas as semistable formal schemes with the ring of convergent power series $\so_L\{X_1,\cdots,X_l\}$ replaced by the ring of overconvergent power series $\so_L[X_1,\cdots, X_l]^{\dagger}$.}, respectively, 
      of such varieties  in a similar way as in the rigid analytic case above.    If we equip the weak formal schemes in $\sm^{\dagger,\sem}_L$  with the log-structure associated to the special fiber over the ring over which they split, every map in these categories is a map of log-schemes.
The functors
 $
\sm^{\dagger,\sem}_L\to \sm^{\dagger}_L ,\quad \sm^{\dagger,\sem,b}_C\to \sm^{\dagger,\sem}_C
 $
are  fully faithful embeddings. The $K$- and $C$-settings are connected by  the base change functors.
  \subsubsection{Semistable reduction}
   We say that an admissible weak formal $\so_L$-scheme $\sx$ is {\em algebraizable} if it is isomorphic to the weak completion of an admissible $\Spec(\so_L)$-scheme $X$. 
   The algebraization theorem, Theorem  \ref{algebraization}, combined with the fact that, up to an isomorphism, there is a unique dagger structure on every rigid analytic affinoid \cite[Cor. 7.5.10]{FvP},  yields the following 
   \begin{corollary}
   \label{referee15}
   Any affine $\eta$-smooth admissible weak formal $\so_L$-scheme $\sx$ is algebraizable. Moreover, we can find an affine $\eta$-smooth admissible $\so_L$-scheme $X$ such  that $\sx\simeq {X}^{\dagger}$. 
   \end{corollary}
  This corollary allows us to prove the following
       \begin{corollary}\label{Tem-dagger}
    \begin{enumerate}
    \item Let $\sx$ be a $\eta$-smooth admissible weak formal scheme over $\so_L$. Then there exists a finite field extension  $E/L$ and a $\eta$-\'etale covering $\sx^{\prime}\to\sx\otimes_{\so_L}\so_E$ such that $\sx^{\prime}$ is  semistable over $\so_E$.
    \item     
    Let  $X$ a smooth qcqs dagger space over $L$. Then there exists a finite extension $E/L$ and an \'etale covering $X^{\prime} \to X\otimes_LE$ such that $X^{\prime}$ is a dagger affinoid and has  a  semistable affine weak formal model. 
\end{enumerate}
    \end{corollary}
    \begin{proof}
    For (1), having Corollary \ref{referee15}, Temkin's proof of Theorem \ref{Tem1} goes through. For (2), we modify the proof of Corollary \ref{Tem11} using the algebraization result from Theorem \ref{algebraization}.
    \end{proof}
  \subsubsection{An equivalence of topoi.}
Let  $\sm^{\dagger}$ be any category from Section \ref{dagger-models} and let $F_{\eta}$ be the forgetful functor $\sx\mapsto\sx_L$. The main result of this section is the following
\begin{proposition}
\label{Bonn2}
If $\sm^{\dagger}$ is  the category $\sm^{\dagger}_K$ or $\sm^{\dagger,\sem}_K$ then $(\sm^{\dagger},F)$ is a base for ${\rm Sm}^{\dagger}_{K,\eet}$. If $\sm^{\dagger}$ is $\sm^{\dagger}_C$,  $\sm^{\dagger,\sem,b}_C$, or $\sm^{\dagger,\sem}_C$ then $(\sm^{\dagger},F_{\eta})$ is a base for ${\rm Sm}^{\dagger}_{C,\eet}$. 
\end{proposition}
 \begin{proof}
Consider first the $K$-setting.  Recall the following dagger version of Raynaud's theory of formal models of rigid analytic varieties:
\begin{theorem}{\rm (Langer-Muralidharan, \cite{LA})}
There is an equivalence of categories between
\begin{enumerate}
\item the category of quasi-paracompact admissible weak formal schemes over $\so_K$ localized by the class of weak formal blow-ups,
\item the category of quasi-separated quasi-paracompact $K$-dagger spaces.
\end{enumerate}
\end{theorem}
It is now easy to see that the proof of Proposition \ref{Bonn1} goes through in our case with Raynaud's theory replaced by this dagger analog. 

   For the  $C$-setting the argument is analogous to the one used in the proof of Proposition \ref{Bonn1}.
\end{proof}   
We call the topology induced by $F_{\eta}$ on the categories $\sm^{\dagger}$ the $\eta$-{\em \'etale} topology. The base-change functors  are continuous for the respective \'etale topologies. 
By Section \ref{dagger-topoi} and Proposition \ref{Bonn2}, $F_{\eta}$ identifies \'etale sheaves on ${\rm Sm}^{\dagger}_K$, resp. ${\rm Sm}^{\dagger}_C$, with $\eta$-\'etale sheaves on $\sm^{\dagger}_K$, $\sm^{\dagger,\sem}_K$, resp. $\sm^{\dagger}_C$, $\sm^{\dagger,\sem,b}_C$, $\sm^{\dagger,\sem}_C$. We obtain the {\em \'etale localization} functors
    $$
    {\rm Psh}(\sm^?_K)\to {\rm Sh}({\rm Sm}^{\dagger}_{K,\eet}),\quad  {\rm Psh}(\sm^?_C)\to {\rm Sh}({\rm Sm}^{\dagger}_{C,\eet}),
    $$
    which assign to any presheaf $\sff$ on weak formal models the corresponding \'etale sheaf $\sff^{\sim}$ viewed as an \'etale sheaf on dagger varieties. 
    Moreover, 
    for any presheaf on $\sm^{\dagger}_K$ or $\sm^{\dagger}_C$, its $\eta$-\'etale sheafification is the same as the $\eta$-\'etale sheafification of its restriction to resp. $\sm^{\dagger,\sem}_K$, $\sm^{\dagger,\sem,b}_C$, or $\sm^{\dagger,\sem}_C$. 
    
\section{Pro-\'etale cohomology of dagger varieties}
Let  the base field $L$  be $K$ or $C$. Fix a pseudo-uniformizer $\varpi\in L$, i.e., an invertible, topologically nilpotent element. All the rigid analytic varieties considered are over $L$; we assume that they are separated and taut\footnote{See \cite[Def. 5.6.6]{Hub} for the definition of "taut".}. 

 The purpose of this section is to define the pro-\'etale cohomology of dagger varieties. We will do it in the most naive way: for a dagger affinoid we will use its presentation of the dagger structure to define the pro-\'etale cohomology of the dagger affinoid as the homotopy colimit of pro-\'etale cohomologies of the (rigid) affinoids in the presentation; for a general dagger variety we will globalize the construction for dagger affinoids via \v{C}ech coverings.
 \subsection{Topology}
 Our cohomology groups will be equipped with a canonical topology. To talk about it in a systematic way, we will work  rationally in the category of locally convex $K$-vector spaces and integrally in the category of pro-discrete $\so_K$-modules. We review here briefly the relevant basic definitions and facts. For details and further reading and references the reader may consult \cite[Sec. 2.1, 2.3]{CDN3}. 
 \subsubsection{Derived category of locally convex $K$-vector spaces} 
  A topological $K$-vector space\footnote{For us, a {\em $K$-topological vector space} is a $K$-vector space with a linear topology.} is called {\em locally convex} ({\em convex} for short) if there exists a 
neighbourhood basis of the origin consisting of $\so_K$-modules.  
 We denote by $C_K$ the category of convex $K$-vector spaces. 
  It  is a quasi-abelian category.
Kernels, cokernels, images, and coimages  are taken in the  category of vector spaces and equipped with the induced topology. A morphism $f:E\to F$ is {\em strict} if and only if it is relatively open, i.e., for any neighbourhood $V$ of $0$ in $E$ there is a neighbourhood $V^{\prime}$ 
of $0$ in $F$ such that $f(V)\supset V^{\prime}\cap f(E)$. 

 The category $C_K$ has a  natural
 exact category structure: the admissible monomorphisms are  embeddings, the
admissible epimorphisms are open surjections. 
A complex $E\in C(C_K)$ is called   {\em strict} if its differentials are strict.
   There are  truncation functors on $C(C_K)$: 
\begin{align*}
\tau_{\leq n}E & :=\cdots \to E^{n-2}\to E^{n-1}\to 
\ker(d_n)\to 0\to\cdots\\
 \tau_{\geq n} E & :=\cdots \to 0\cdots \to
\coim(d_{n-1}) \to E^n\to E^{n+1}\to\cdots
\end{align*}
with cohomology objects $$\wt{H}^n(E):=
\tau_{\leq n}\tau_{\geq n}(E)=(\coim(d_{n-1})\to \ker(d_n)).
$$
We note that here $\coim(d_{n-1})$ and $\ker(d_n)$ are equipped naturally with the quotient and subspace topology, respectively. The  cohomology $H^*(E)$ taken in the category of $K$-vector spaces we will call {\em algebraic} and, if necessary, we will always equip it with the sub-quotient topology. 

  We will denote the left-bounded derived $\infty$-category of $C_K$ by $\sd(C_K)$. 
       A morphism of complexes that is a quasi-isomorphism in $\sd(C_K)$, i.e., its cone is strictly exact,  will be called a {\em strict quasi-isomorphism}. We will denote by $D(C_K)$ the homotopy category of $\sd(C_K)$.

 For $n\in \Z$, let $D_{\leq n}(C_K)$ (resp. $D_{\geq n}(C_K)$) denote the full subcategory of $D(C_K)$  of complexes that are strictly exact in degrees $k >n$ (resp. $k<n$). The above truncation functors extend to truncations functors
  $\tau_{\leq n}: D(C_K)\to D_{\leq n}(C_K)$ and $\tau_{\geq n}: D(C_K)\to D_{\geq n}(C_K)$. The pair $(D_{\leq n}(C_K),D_{\geq n}(C_K)$) defines a $t$-structure on $D(C_K)$. The (left) heart  $LH(C_K)$ is an abelian category: every object of $LH(C_K)$ is represented (up to  equivalence)  by a monomorphism $f:E\to F$, where $F$ is in degree $0$, i.e., it is isomorphic to a complex $0\to E\stackrel{f}{\to} F\to 0$; 
{\it if $f$ is strict} this object is also represented by the cokernel of $f$
(the whole point of this construction is to keep track of the two possibly different topologies
on $E$: the given one and the one inherited by the inclusion into $F$).
 
  We have an embedding
$I: C_K\hookrightarrow LH(C_K)$, $E\mapsto (0\to E)$,
that induces an equivalence $D(C_K)\stackrel{\sim}{\to} D(LH(C_K))$ that is compatible with t-structures. These t-structures pull back to  $t$-structures on the derived dg categories $\sd(C_K), \sd(LH(C_K))$ and so does the above equivalence. There is a functor (the {\em classical part}) $C: LH(C_K)\to C_K$ that sends the monomorphism $f: E\to F$ to $\coker f$. We have $CI\simeq \id_{C_K}$ and a natural epimorphism $e: \id_{LH(C_K)}\to IC$.

  We will denote by $\wt{H}^n: \sd(C_K)\to \sd(LH(C_K))$ the associated cohomological functors.  Note that $C\wt{H}^n=H^n$ and we have a natural epimorphism $\wt{H}^n\to IH^n$.  If, evaluated on $E$,  this epimorphism is an isomorphism  we will say that the  cohomology $\wt{H}^n(E)$ is {\em classical} (in most cases this is equivalent to $H^n(E)$ being separated). 
  
    \subsubsection{The category of pro-discrete $\so_K$-modules.}
 
Objects in the category $PD_K$ of pro-discrete $\so_K$-modules are topological $\so_K$-modules that are countable inverse limits, as topological $\so_K$-modules, of discrete $\so_K$-modules $M^i$, $i\in \N$. 
It is a quasi-abelian category.  It has countable filtered projective limits. Countable products are exact functors. 
    
   Inside  $PD_K$ we distinguish the category $PC_K$ of pseudocompact $\so_K$-modules, i.e., pro-discrete modules $M\simeq\invlim_iM_i$ such that each $M_i$ is of finite length (we note that if $K$ is a finite extension of $\Q_p$ this is equivalent to $M$ being profinite). It is an abelian category. It has countable exact products as well as exact countable filtered projective limits. 

  There is a functor from the category of pro-discrete $\so_K$-modules to convex $K$-vector spaces.
 Since $K \simeq \varinjlim(\so_K\lomapr{\varpi} \so_K\lomapr{\varpi}\so_K\lomapr{\varpi} \cdots)$, the 
algebraic tensor product $M\otimes_{\so_K}K$ is an inductive limit:
$$
M\otimes_{\so_K}K\simeq  \varinjlim(M\lomapr{\varpi} M\lomapr{\varpi}M\lomapr{\varpi} \cdots).
$$
We equip it with the induced inductive limit topology. 
This defines a tensor product functor 
  $$(-){\otimes}K: PD_K\to C_{K}, \quad M\mapsto M\otimes _{\so_K}K.
  $$
 Since $C_K$ admits filtered inductive limits, the functor $(-){\otimes}K$ extends 
to a functor $(-){\otimes}K: \Ind(PD_K)\to C_{K}$.

    The  functor $(-){\otimes}K$ is right exact but not, in general, left exact\footnote{We will call a functor $F$ right exact if  it transfers strict exact sequences $0\to A\to B\to C\to0$ to  costrict exact sequences 
    $F(A)\to F(B)\to F(C)\to 0$.}. For example,
  the short strict exact sequence
  $$
  0\to \prod_{i\geq 0}p^i\Z_p\lomapr{\can}\prod_{i\geq 0}\Z_p\to \prod_{i\geq 0}\Z_p/p^i\to 0
  $$
  after tensoring with $\Q_p$ is not costrict exact on the left (note that $(\prod_{i\geq 0}\Z_p/p^i){\otimes}\Q_p$ is not Hausdorff). We will consider its (compatible) left derived functors
  $$
  (-){\otimes}^LK: \sd^{-}(PD_K)\to \Pro(\sd^{-}(C_K)),\quad  (-){\otimes}^LK: \sd^{-}(\Ind(PD_K))\to \Pro(\sd^{-}(C_K)).
  $$
  
   The following fact will greatly simplify our computations.
  \begin{proposition}{\rm (\cite[Prop. 2.6]{CDN3})}\label{acyclic-integral}
  If $E$ is a complex of torsion free and $p$-adically complete (i.e., $E\simeq \varprojlim_n E/p^n$) modules from $PD_K$ then the natural map
  $$
   E{\otimes}^LK\to  E{\otimes}K
  $$ is a strict quasi-isomorphism. 
  \end{proposition}
\subsection{Pro-\'etale cohomology of dagger varieties}
In this section we will  define pro-\'etale cohomology of dagger varieties and  study its basic properties. 
\subsubsection{Dagger varieties  and pro-systems of rigid analytic  varieties} \label{pierre11} We will briefly review here the content of \cite[Appendix]{Vez}.
Recall the following definition \cite[Def. A.19]{Vez}: 
\begin{definition}
Let $X$ be a rigid analytic affinoid. A {\em presentation of a dagger structure on $X$} is a pro-affinoid rigid variety $\{ X_h\}$, $h\in \N$,
where $X$ and all $X_h$ are rational subvarieties of $X_1$, such that $X\Subset  X_{h+1}\Subset  X_h$ and the pro-system is 
coinitial among rational subvarieties of $X_1$ containing $X$ in their interiors\footnote{Recall that, for an open immersion $X\subset Y$ of adic spaces over $L$, we write $X\Subset Y$ if the inclusion factors over the adic compactification
 of $X$ over $L$ (see \cite[Th. 5.1.5]{Hub}).}. A {\em morphism} of presentations between 
$\{X_h\}$ and $\{ Y_k\}$ is a morphism of pro-objects, i.e., an element of $\invlim_k\dirlim_h \Hom(X_h, Y_k)$.
\end{definition}
\begin{example}
Let $X=X_1(f/g)$ be a rational subvariety of an affinoid variety $X_1$. The pro-system $\{X_h=X_1(\varpi f^h/g^h)\}$ of rational subvarieties  of $X_1$ is a presentation of a dagger structure on $X$.

 More generally, consider the rational inclusion $X=X_1(f_1/g,\cdots, f_m/g)\Subset X_1$ of affinoid rigid varieties. We can write
\begin{align*}
\so(X_1) & =L\{\varpi\tau_1,\cdots, \varpi\tau_n\}/I,\\
\so(X) & =L\{\tau_1,\cdots, \tau_n,v_1,\cdots,v_m\}/((v_ig-f_i)+I).
\end{align*}
 Let  $X_h$ be the rational subvariety of $X_1$ with
$$
\so(X_h)=L\{\varpi^{1/h}\tau_1,\cdots, \varpi^{1/h}\tau_n,\varpi^{1/h}v_1,\cdots, \varpi^{1/h}v_m\}/((v_ig-f_i)+I).
$$
The pro-system $\{X_h\}$ of rational subvarieties  of $X_1$ is a presentation of a dagger structure on $X$. We have
$$
\dirlim\so(X_h)\simeq L[\tau_1,\cdots, \tau_n,v_1,\cdots, v_m]^{\dagger}/((v_ig-f_i)+I),
$$
which is a dagger algebra.
\end{example}
 
 The following proposition  clarifies  the relationship between presentations of  dagger structures and dagger algebras.
\begin{proposition}{\rm (\cite[Prop. A.22]{Vez})}\label{equivalence1}
Let $\wh{X}=\Sp \wh{R}$ be a rigid affinoid and let $\{X_h\}$ be a presentation of a dagger structure on $\wh{X}$. We have 
\begin{enumerate}
\item $R=\dirlim \so(X_h)$ is a dagger algebra dense in $\wh{R}$;
\item the functor $\{X_h \}\mapsto \Sp^{\dagger} R$ induces an equivalence of categories between dagger 
affinoid varieties  and their presentations.
\end{enumerate}
\end{proposition}
In fact, it is not hard to see that we have a functor ${\rm pres}: X\mapsto \{X_h\}$ from dagger algebras to presentations of dagger structures (up to a unique isomorphism) that is the right inverse (on the nose) of the functor in the above proposition.

\subsubsection{\'Etale topology of dagger varieties} For basic properties of dagger algebras and varieties  and morphisms between them see \cite{GK0}. For basic properties of \'etale and smooth morphisms of dagger varieties see \cite{ET1}. 
We quote the following result.
\begin{proposition}{\rm (\cite[Th. 2.3]{ET1})}
Let $X$ be a dagger affinoid with completion $\wh{X}$. We have a  natural equivalence of \'etale topoi
$$
{\rm Sh}(\wh{X}_{\eet})\stackrel{\sim}{\to} {\rm Sh}(X_{\eet}).
$$ 
\end{proposition}

One can  promote the equivalence of categories between dagger spaces and their presentations in Proposition \ref{equivalence1} to an equivalence of topoi. 
  \begin{definition}{\rm (\cite[Def. A.24]{Vez})}
 (i)  Let ${\rm P}$ be a property of morphisms of rigid analytic varieties. We say that a morphisms of pro-rigid varieties  $\phi: X\to Y$ has the 
  property ${\rm P}$ if $X\simeq \{X_h\}, Y\simeq \{Y_k\}$ and $\phi=\{\phi_h\}$ with $\phi_h: X_h\to Y_h$ having property $P$.\\
 (ii)  We say that a collection of  morphisms of pro-rigid spaces 
  $\{\phi_i: \{U_{ih}\}\to \{X_h\}\}_{i\in I}$ is a {\em cover} if $X\Subset \bigcup_i\im(U_{ih})$ for all $h$. 
  \end{definition}
  In particular, one can  define open immersions, smooth, and \'etale morphisms  of presentations of dagger affinoids which agree with the corresponding notions for dagger affinoids. 
  Since  the morphisms $\wh{X}\subset X_h$ are open immersions (hence \'etale), we deduce that, if a morphism $X\to Y$ is an open immersion (resp. smooth, resp. \'etale), then so is the associated morphism $\wh{X}\to\wh{Y}$. 
  
  ($\bullet$) From now on we will use the following convention: if $X$ is a smooth dagger affinoid, the presentation $X\simeq \{X_h\}$ will be assumed to have all $X_h$ smooth as well. 
  \begin{corollary}{\rm (\cite[Cor. A.28]{Vez})}
  Let $X$ be a dagger affinoid  with a presentation $\{X_h\}$. We have a  natural equivalence of \'etale topoi
  $${\rm Sh}(X_{\eet}) \stackrel{\sim}{\to }{\rm Sh}(\{X_h\}_{\eet}).
  $$
  \end{corollary}
\subsubsection{Definition of pro-\'etale cohomology.} $\quad$

 (i) {\em Local definition.} If $\{X_h\}$ is a pro-rigid analytic  variety,  we set  $$\rg_{\proeet}(\{X_h\},\Z/p^n(r)):=
\hocolim_h\rg_{\proeet}(X_h,\Z/p^n(r))\stackrel{\sim}{\leftarrow}\hocolim_h\rg_{\eet}(X_h,\Z/p^n(r)),\quad r\in\Z.$$
Let $X$ be a dagger affinoid. We define its pro-\'etale cohomology as 
\begin{equation}
\label{def22}
\rg_{\proeet}(X,\Z/p^n(r)):=\rg_{\proeet}({\rm pres}(X),\Z/p^n(r)),\quad r\in\Z.
\end{equation} If the dagger affinoid $X$ has a dagger presentation $\{X_h\}$ then  
 $\rg_{\eet}(X,\Z/p^n(r))\stackrel{\sim}{\leftarrow}\hocolim_h\rg_{\eet}(X_h,\Z/p^n(r))$ and we have a natural quasi-isomorphism 
 \begin{equation}
 \label{pig1}
 \rg_{\eet}(X,\Z/p^n(r))\stackrel{\sim}{\to}\rg_{\proeet}(X,\Z/p^n(r)).
 \end{equation}
We make similar definitions for $\Z_p$ and $\Q_p$ coefficients. We have the natural maps (note the direction of the second map)
\begin{equation}
\label{dagmap}
\rg_{\proeet}(X,\Z_p(r))\to \rg_{\proeet}(X,\Q_p(r)),\quad \rg_{\proeet}(X,\Z_p(r))\to \rg_{\eet}(X,\Z_p(r)).
\end{equation}
The first map is a  rational quasi-isomorphism.
If the dagger affinoid $X$ has  dagger presentation $\{X_h\}$ then we define the second  map in the following way
\begin{align}
\label{dagmaps}
\rg_{\proeet}(X,\Z_p(r)) & =\hocolim_h\rg_{\proeet}(X_h,\Z_p(r))  \stackrel{\sim}{\leftarrow}\hocolim_h\rg_{\eet}(X_h,\Z_p(r))\\
& = \hocolim_h\holim_n\rg_{\eet}(X_h,\Z/p^n(r))
  \to \holim_n\hocolim_h\rg_{\eet}(X_h,\Z/p^n(r))\notag\\
  & =\holim_n\rg_{\eet}(X,\Z/p^n(r))\stackrel{\sim}{\leftarrow}
\rg_{\eet}(X,\Z_p(r)).\notag
\end{align}
Here the second quasi-isomorphism holds because $X_h$ is quasi-compact (cover $X_h$ with a finite number of affionoids and use the quasi-isomorphism (\ref{pig1})).

\smallskip
     (ii) {\em Topological issues.} We need to discuss topology.  Let, for a moment,  $X$ be a rigid analytic variety over $L$. We equip the pro-\'etale and \'etale cohomologies $\R\Gamma_{\proeet}(X,\Q_p(r))$, and  $\R\Gamma_{\eet}(X,\Q_p(r))
$ with a natural topology by proceeding as in \cite[Sec. 3.3.2]{CDN3} by using as local data compatible $\Z/p^n$-free complexes\footnote{Such complexes can be found, for example, by taking the system of  \'etale hypercovers.}. If  
 $X$ is quasi-compact, we obtain in this way complexes of Banach spaces over $\Q_p$. In that case the natural continuous map $\R\Gamma_{\eet}(X,\Q_p(r))\to \R\Gamma_{\proeet}(X,\Q_p(r))$ is a strict quasi-isomorphism. 
 
 More precisely, we have
 $$
 \rg_{\proeet}(X,\Q_p(r)):=\hocolim\rg_{\eet}(U_{\cdot},\Q_p(r)),
 $$
 where the homotopy colimit is over \'etale quasi-compact hypercoverings\footnote{Here and below, we use ``colimit over hypercoverings" as a shorthand for ``colimit over the filtered category of hypercoverings up to simplicial homotopy".} of $X$.   
Since all the complexes $\rg_{\eet}(U_{\cdot},\Q_p(r))$ are complexes of Fr\'echet spaces, all the arrows in the colimit are strict quasi-isomorphisms. Hence we can compute with any particular hypercovering. 
   
\begin{remark} 
\label{ducros}
We will often use the following  simple observation. If $X$ is a smooth  rigid analytic variety then we can find an increasing quasi-compact admissible covering $\{U_n\}_{n\in\N}$ of $X$ such that $U_i$ is contained in the relative interior of $U_{i+1}$.
If $X$ is moreover partially proper we can assume that $U_i\Subset U_{i+1}$. We have analogous statements for dagger varieties.

 It follows that, for a general smooth rigid analytic variety $X$ we have an increasing quasi-compact admissible covering $\{U_n\}_{n\in\N}$ of $X$, such that we have   (in $\sd(C_{\Q_p})$)
 $$
 \rg_{\proeet}(X,\Q_p(r))\simeq \holim_n\rg_{\eet}(U_{n},\Q_p(r)).
 $$
 Hence we have the short exact sequence
 $$
0\to H^1\holim_n\wt{H}^{i-1}_{\eet}(U_n,\Q_p(r))\to  \wt{H}^i \rg_{\proeet}(X,\Q_p(r))\to H^0\holim_n\wt{H}^i_{\eet}(U_{n},\Q_p(r))\to 0.
 $$
\end{remark}

 If $X$ is a dagger affionoid, its pro-\'etale cohomology acquires now natural topology by taking the homotopy colimit in (\ref{def22}) in $\sd(C_{\Q_p})$.  
 
\smallskip
 (iii) {\em Globalization.} For a general smooth dagger variety $X$, 
we have the  natural equivalence of analytic topoi
$$
{\rm Sh}(({\rm SmAff}^{\dagger}_{L}/X_L)_{\eet})\stackrel{\sim}{\to} {\rm Sh}(({\rm Sm}^{\dagger}_{L}/X_L)_{\eet}),
$$
where  ${\rm Sm}^{\dagger}_{L}/X_L$  is the category of smooth morphisms of dagger varieties to $X_L$ and ${\rm SmAff}^{\dagger}_{L}/X_L$ is its full subcategory of affinoid objects.
Using this equivalence, we define the sheaf $\sa_{\proeet}(r)$, $r\in\Z$,  on $X_{\eet}$ as the sheaf associated to the presheaf defined by:  $U\mapsto \rg_{\proeet}(U,\Q_p(r))$, $U\in {\rm SmAff}^{\dagger}_{L}$, $U\rightarrow X$ an \'etale map. We define the pro-\'etale cohomology of $X$ as 
$$
\rg_{\proeet}(X,\Q_p(r)):=\rg_{\eet}(X, \sa_{\proeet}(r)),\quad r\in\Z.
$$
 We equip it with topology by proceeding as in the case of pro-\'etale cohomology of rigid analytic varieties starting with the case of dagger affinoids that was described above.

\smallskip
  (iv) {\em Local-global compatibility.} This definition is consistent with the  previous definition:
\begin{lemma}
 Let  $X$ be a dagger affinoid with the presentation $\{X_h\}$.  Then the natural map
$$
\rg_{\proeet}(\{X_h\},\Q_p(r))\to \rg_{\eet}(X, \sa_{\proeet}(r)),\quad r\in\Z,
$$
is a strict quasi-isomorphism. 
\end{lemma}
\begin{proof}Set $\rg_{\proeet}^{\sharp}(X,\Q_p(r)):=\rg_{\proeet}(\{X_h\},\Q_p(r)).$ It suffices to show that, for any \'etale affinoid hypercovering $U_{\cdot}$ of $X$, the natural map 
$$
\rg_{\proeet}^{\sharp}(X,\Q_p(r))\to \rg_{\proeet}^{\sharp}(U_{\cdot},\Q_p(r))
$$
is a strict quasi-isomorphism (modulo taking a refinement of $U_{\cdot}$). For that, it suffices to show that, for any $k\in\N$, the map
\begin{equation}
\label{cieplo1}
\tau_{\leq k}\rg_{\proeet}^{\sharp}(X,\Q_p(r))\to \tau_{\leq k}\rg_{\proeet}^{\sharp}(T,\Q_p(r)),
\end{equation}
where $T=U_{\cdot}$, 
is a strict quasi-isomorphism. Since, for that, it is enough to work with the truncation $\tau_{\leq k+1} T$ we will assume that $T$ is a finite hypercovering and has a finite number of affinoids in every degree. 

Take  the dagger presentation $X\simeq \{X_h\}, h\in\N$. We can represent $T$ by a pro-system of hypercoverings
$\{T_h\to V_h\}, V_h \subset X_h,$ $h\in\N$, forming a dagger presentation of $T$ degree-wise\footnote{This uses the simple observation that if a collection of  morphisms of pro-rigid spaces 
  $\{\phi_i: \{V_{ih}\}\to \{X_h\}\}_{i\in I}$ is an \'etale  { cover} then we can choose a subsequence $\{X_{k_h}\}$ of $\{X_h\}$ such that the pro-rigid spaces $\{V_{i,k_h}:=V_{ih}\times_{X_h}X_{k_h}\}$ form an \'etale cover of $\{X_h\}$ and moreover all the maps 
  $\{\phi_i: \{V_{i,k_h}\}\to \{X_{k_h}\}\}_{i\in I}$  are \'etale covers (to see this use the "initial" part of the definition of presentations).}. We note that then $V_{h+1}\Subset V_h$. From the universal property of $\{X_h\}$ and the quasi-compactness of $V_h$, we get that the two pro-rigid varieties $\{X_h\}$ and $\{V_h\}$ are equivalent. 
It follows that we have a natural strict quasi-isomorphism
$$
\hocolim_h\rg_{\eet}(X_h,\Q_p(r))\stackrel{\sim}{\to}\hocolim_h\rg_{\eet}(V_h,\Q_p(r)).
$$
Hence the map (\ref{cieplo1}) is represented by a composition
\begin{align*}
& \tau_{\leq k}\rg_{\proeet}^{\sharp}(X,\Q_p(r))\stackrel{\sim}{\leftarrow}\tau_{\leq k}(\hocolim_h\rg_{\eet}(V_h,\Q_p(r)))\\
 & \quad \quad \stackrel{\sim}{\rightarrow}\tau_{\leq k}(\hocolim_h\rg_{\eet}(T_h,\Q_p(r))) \simeq \tau_{\leq k}\rg^{\sharp}_{\proeet}(T,\Q_p(r)),
\end{align*}
where the middle strict quasi-isomorphism follows from \'etale descent for rigid analytic varieties. This finishes our proof of the lemma.
\end{proof}
\begin{remark}For a smooth dagger variety $X$, we can define similarly the integral pro-\'etale cohomology $\rg_{\proeet}(X,\Z_p(r)), r\in\Z$. We have the natural maps
\begin{align*}
& \rg_{\proeet}(X,\Z_p(r))\to \rg_{\proeet}(X,\Q_p(r)),\\
&  \rg_{\proeet}(X,\Z_p(r))\to \rg_{\eet}(X,\Z_p(r))\stackrel{\sim}{\to}\rg_{\eet}(\wh{X},\Z_p(r))\stackrel{\sim}{\to}\rg_{\proeet}(\wh{X},\Z_p(r)).
\end{align*}
For $X$ quasi-compact, the first map becomes a strict  quasi-isomorphism after tensoring with $\Q_p$; this is not the case for general $X$. The second map is  a globalization of  maps for dagger affinoids defined in (\ref{dagmaps}).
\end{remark}
\subsubsection{Comparison isomorphisms.} Let $L=K,C$. For $X\in {\rm Sm}^{\dagger}_L$, we have a natural map 
\begin{equation}
\label{map-berkeley}
\iota_{\proeet}: \rg_{\proeet}(X,\Q_p(r))\to\rg_{\proeet}(\wh{X},\Q_p(r)). 
\end{equation}
It is obtained by the globalization of such maps for dagger affinoids: if the dagger affionoid $X$ has a dagger presentation $\{X_h\}$ then we set
$$
\iota_{\proeet}: \rg_{\proeet}(X,\Q_p(r))= \hocolim_h\rg_{\proeet}(X_h,\Q_p(r)) \lomapr{\can}\rg_{\proeet}(\wh{X},\Q_p(r)).
$$
\begin{proposition}
\label{kicia-kicia}
Let $X$ be partially proper. Then the map (\ref{map-berkeley})  is a strict quasi-isomorphism.
\end{proposition}
\begin{proof}Since a partially proper smooth dagger variety is locally Stein, we can assume $X$ to be Stein.  Choose an admissible covering of $X$ by  an increasing sequence of dagger affinoids $\{U_n\}$, $n\in\N$, strictly contained in each other. Then the map $\iota_{\proeet}$ from (\ref{map-berkeley}) can be written as the composition
$$\rg_{\proeet}(X,\Q_p(r))\stackrel{\sim}{\to} \holim_n\rg_{\proeet}(U_n,\Q_p(r))\to  \holim_n\rg_{\proeet}(\wh{U}_n,\Q_p(r))\stackrel{\sim}{\leftarrow} \rg_{\proeet}(\wh{X},\Q_p(r))
$$
and we need to show that the middle map is a strict quasi-isomorphism. 
But, for every $n > 1$, the map $\wh{U}_n\to \wh{U}_{n-1}$ factorizes canonically as $\wh{U}_n\to {\rm pres}(U_n)\to \wh{U}_{n-1}$ yielding the factorization 
$$
\rg_{\proeet}(\wh{U}_{n-1},\Q_p(r))\to \rg_{\proeet}({\rm pres}(U_n),\Q_p(r))\to \rg_{\proeet}(\wh{U}_n,\Q_p(r)).
$$
It follows that the prosystems
$$
\{\rg_{\proeet}(U_n,\Q_p(r))\},\quad \{\rg_{\proeet}({\rm pres}(U_n),\Q_p(r))\}
$$ are equivalent. Since,  $\rg_{\proeet}(U_n,\Q_p(r))\stackrel{\sim}{\leftarrow}\rg_{\proeet}({\rm pres}(U_n),\Q_p(r))$ we are done.
\end{proof}
\section{Rigid analytic  syntomic cohomology}
In this section we define syntomic cohomology for smooth rigid analytic varieties over $K$ or $C$ by $\eta$-\'etale descent of the classical definition due to Fontaine-Messing. We show that the computations of syntomic cohomology from \cite{CN}
done for rigid analytic varieties  with semistable reduction generalize to all smooth rigid varieties. We also introduce Hyodo-Kato cohomology for such varieties,  prove that it satisfies Galois descent,
 and  define the Hyodo-Kato morphism (that is a quasi-isomorphism over $C$). Finally, over $K$, we define Bloch-Kato  rigid analytic syntomic cohomology (built from Hyodo-Kato and de Rham cohomologies) and show that it is quasi-isomorphic to the  rigid analytic syntomic cohomology.

 \subsection{Definition of rigid analytic syntomic cohomology}We define  the syntomic cohomology of smooth rigid analytic varieties by \'etale descent of crystalline syntomic cohomology of semistable models.

    Let $\su\in \sm^{\sem}_K$. We consider it as a log-formal scheme with the log-structure associated to the special fiber.
     For $r\geq 0$, we have  the mod $p^n$, completed, and rational 
   absolute (i.e., over $\Z_p$) filtered crystalline cohomology 
   \begin{align*}
  & \R\Gamma_{\crr}(\su_n,\sj^{[r]}), \quad    \R\Gamma_{\crr}(\su,\sj^{[r]})  :=\holim_n \R\Gamma_{\crr}(\su_n,\sj^{[r]}),\\
   & \R\Gamma_{\crr}(\su,\sj^{[r]})_{\Q_p}  :=\R\Gamma_{\crr}(\su,\sj^{[r]})\otimes^L_{\Z_p}{\Q_p}.
   \end{align*}
     Here  $\sj^{[r]}$ denotes the $r$'th Hodge filtration sheaf.  The corresponding $\eta$-\'etale sheafifications on $\sm^{\sem}_K$ we will denote by $F^r\sa_{\crr,n}, F^r\sa_{\crr}, $ and $F^r\sa_{\crr,\Q_p}$. We make analogous definitions for crystalline cohomology of basic semistable models over $\so_C$ (see \cite{BE2} for details). 
     
    For $r\geq 0$, define the mod $p^n$, completed, and rational 
    crystalline syntomic cohomology
    \begin{align*}
    \R\Gamma_{\synt}(\su,\Z/p^n(r)) & :=[\R\Gamma_{\crr}(\su_n,\sj^{[r]})\lomapr{p^r-\phi}\R\Gamma_{\crr}(\su_n)]\simeq [[\R\Gamma_{\crr}(\su_n)]^{\phi=p^r}\lomapr{\can} \R\Gamma_{\crr}(\su_n)/\R\Gamma_{\crr}(\su_n,\sj^{[r]})],\\
     \R\Gamma_{\synt}(\su,\Z_p(r)) & :=\holim_n \R\Gamma_{\synt}(\su,\Z/p^n(r)),\\
       \R\Gamma_{\synt}(\su,\Z_p(r))_{\Q_p}  & :=     \R\Gamma_{\synt}(\su,\Z_p(r))\otimes^L_{\Z_p}\Q_p \simeq [\R\Gamma_{\crr}(\su,\sj^{[r]})_{\Q_p}\lomapr{p^r-\phi}\R\Gamma_{\crr}(\su)_{\Q_p}].
    \end{align*}
        The corresponding $\eta$-\'etale sheafifications on $\sm^{\sem}_K$ we will denote by $\sa_{\synt,n}(r), \sa_{\synt}(r), $ and $\sa_{\synt}(r)_{\Q_p}$. We make analogous definitions for crystalline  syntomic cohomology of basic semistable models over $\so_C$. 
        We have the distinguished triangles
 \begin{align*}
 &  \sa_{\synt,n}(r)\to F^r\sa_{\crr,n}\lomapr{p^r-\phi} \sa_{\crr,n},\\
 & \sa_{\synt,n}(r)\to \sa_{\crr,n}^{\phi=p^r}\to \sa_{\crr,n}/F^r,
 \end{align*}
 where we set $\sa_{\crr,n}^{\phi=p^r}:=[\sa_{\crr,n}\lomapr{p^r-\phi} \sa_{\crr,n}]$.  Similarly for the completed and rational cohomology.
    
     For $X\in {\rm Sm}_L$, $L=K,C$,  we define two rational (rigid analytic) syntomic cohomologies: 
     $$
 \R\Gamma_{\synt}(X,\Z_p(r))_{\Q_p} :=\rg_{\eet}(X,\sa_{\synt}(r))\otimes^L_{\Z_p}{\Q_p},\quad  \R\Gamma_{\synt}(X,\Q_p(r)) :=\rg_{\eet}(X,\sa_{\synt}(r)_{\Q_p}).
    $$
    From now on, to simplify the notation, we will write $(-)_{\Q_p}$ for $(-)\otimes_{\Z_p}^L{\Q_p}$; similarly for coefficients other than $\Q_p$.
    There is a canonical map 
    \begin{equation}
    \label{passage}
     \R\Gamma_{\synt}(X,\Z_p(r))_{\Q_p} \to \R\Gamma_{\synt}(X,\Q_p(r)).
     \end{equation} It follows immediately from the definitions that, for $X$ quasi-compact, this is a quasi-isomorphism (but it is not so in general).
     By proceeding just as in \cite[Sec. 3.3.1]{CDN3} (using crystalline embedding systems) we can equip both complexes in (\ref{passage}) 
  with a natural topology for which they   become  complexes of Banach  spaces over $\Q_p$  in the case $X$ is quasi-compact\footnote{We note that $\so_K$ being syntomic over $\so_F$, all the integral complexes in sight are
   in fact  $p$-torsion free.} (and in that case the quasi-isomorphism (\ref{passage}) is strict).  We do the same for the crystalline complexes involved in the definition of syntomic cohomology. 
  We have  distinguished triangles in $\sd(C_{\Q_p})$
\begin{align}
\label{seq11}
 & \R\Gamma_{\synt}(X,\Z_p(r))_{\Q_p} \to  \rg_{\eet}(X,\sa_{\crr}^{\phi=p^r})_{\Q_p}\to \rg_{\eet}(X,\sa_{\crr}/F^r)_{\Q_p},\\
 & \R\Gamma_{\synt}(X,\Q_p(r)) \to  \rg_{\eet}(X,\sa_{\crr,\Q_p}^{\phi=p^r})\to \rg_{\eet}(X,\sa_{\crr,\Q_p}/F^r).\notag
 \end{align}
 We will show later (see Corollary \ref{compt1}) that if $X=\sx_K$, for an admissible  semistable  formal scheme  $\sx$ over $\so_K$, then the canonical map
 $$
 \R\Gamma_{\synt}(\sx,\Q_p(r))\to \R\Gamma_{\synt}(X,\Q_p(r))
 $$ is a strict quasi-isomorphism. 
 \subsubsection{Rigid analytic de Rham  cohomology}\label{derham1}
  Let $L=K,C$. Consider the presheaf $X\mapsto \R\Gamma_{\dr}(X)$ of filtered dg $L$-algebras on ${\rm Sm}_L$. Let $\sa_{\dr}$ be its \'etale sheafification on ${\rm Sm}_L$. It is a sheaf of filtered $L$-algebras on ${\rm Sm}_{L,\eet}$. For $X\in {\rm Sm}_L$, we have the natural filtered quasi-isomorphism: $\R\Gamma_{\dr}(X)\stackrel{\sim}{\to}\R\Gamma_{\eet}(X,\sa_{\dr})$. We equip $\R\Gamma_{\dr}(X)$ with the topology induced by the canonical topology on affinoid algebras; we equip $\R\Gamma_{\eet}(X,\sa_{\dr})$ with topology using \'etale descent as we did before.
  Then the above quasi-isomorphism is strict:  sheaves of differential forms satisfy \'etale descent in the strict sense.

  Let $X\in {\rm Sm}_L$. We will need to understand the cohomology groups in degrees $r-1$ and $r$ of 
 $$
 \rg_{\dr}(X)/F^r\simeq\rg(X,\so_X\to\Omega^1_X\to\cdots\to\Omega^{r-1}_X).
 $$
 To do that consider the distinguished triangle (in $\sd(C_L)$)
 \begin{equation}\label{poznan1}
 0\to \ker d_r[-r]\to \tau_{\leq r}\Omega\kr_X\to \Omega^{\leq r-1}_X\to 0,
 \end{equation}
 where $d_r:\Omega^r_X\to\Omega^{r+1}_X$ is the de Rham differential.
 It yields the long exact sequence
 $$
 0\to \wt{H}^{r-1}_{\dr}(X)\to \wt{H}^{r-1}(\rg_{\dr}(X)/F^r)\to \wt{H}^r(X,\ker d_r[-r])\to \wt{H}^r_{\dr}(X).
 $$
 Or, since $\wt{H}^r(X,\ker d_r[-r])=\Omega^{r}(X)^{d=0}$, 
 the short exact sequence
 $$
 0\to \wt{H}^{r-1}_{\dr}(X)\to \wt{H}^{r-1}(\rg_{\dr}(X)/F^r)\to \ker \pi \to 0,
 $$
 where $\pi$ is the natural map $\Omega^r(X)^{d=0}\to \wt{H}^r_{\dr}(X).$
We have a monomorphism $\im d_{r-1}(X)\hookrightarrow \ker \pi$.

 The distinguished triangle (\ref{poznan1}) yields also the long exact sequence
 $$
0\to \coker \pi \to \wt{H}^r(\rg_{\dr}(X)/F^r)\to \wt{H}^1(X,\ker d_r)\to \wt{H}^{r+1}(X,\tau_{\leq r}\Omega\kr_X).
 $$
\begin{remark}
  (a)  If $X$ is proper, all the Hodge and de Rham cohomology groups are classical (finite dimensional vector spaces over $K$), the Hodge-de Rham spectral sequence degenerates at $E_1$ \cite[Cor. 1.8]{Sch},  and we get the  isomorphisms
  $$
  H^{r-1}_{\dr}(X)\stackrel{\sim}{\to} \wt{H}^{r-1}(\rg_{\dr}(X)/F^r),\quad  H^{r}_{\dr}(X)/\Omega^r(X)\stackrel{\sim}{\to} \wt{H}^{r}(\rg_{\dr}(X)/F^r).
  $$
 (b) If  $X$ is Stein, we have $H^i(X,\Omega^j_X)=0$, $i\neq 0$, and all the de Rham cohomology groups are classical (Fr\'echet spaces).
We have
$$
\rg_{\dr}(X)/F^r\simeq (\so(X)\to\Omega(X)\to \cdots\to \Omega^{r-1}(X))
$$
with strict differentials.
Hence we get the  isomorphisms 
  \begin{align*}
  \wt{H}^{r-1}(\rg_{\dr}(X)/F^r)\simeq  \Omega^{r-1}(X)/\im d_{r-1},\quad 
   \wt{H}^i(\rg_{\dr}(X)/F^r)\simeq 0,\, i\geq r.
  \end{align*}
  Hence the cohomology $\wt{H}^{r-1}(\rg_{\dr}(X)/F^r)$  is classical.
  \end{remark}

  \begin{proposition}Let $X\in {\rm Sm}_K$.
 \label{derham}Let $r
 \geq 0$. We have a canonical strict quasi-isomorphism
 $$\gamma_r: \rg_{\dr}(X)/F^r\stackrel{\sim}{\to}\rg_{\eet}(X,\sa_{\crr,\Q_p}/F^r).
  $$
 \end{proposition}
 \begin{proof}Let $\sx$ be a quasi-compact semistable    formal scheme over $\so_E$, $[E:K] < \infty$. 
Recall that \cite[Cor. 2.4]{NN}
there exists a functorial  and compatible with base-change  quasi-isomorphism
$$\gamma_r:\quad \R\Gamma_{\dr}(\sx_K)/F^r \stackrel{\sim}{\to} \R\Gamma_{\crr}(\sx,\so/\sj^{[r]})_{\Q_p}.
$$
This quasi-isomorphism is in fact strict: this is not completely evident because the integral version of the morphism is only a $p^N$-quasi-isomorphism for some constant $N$ 
but can be seen by an argument identical to the one used at the end of the proof of \cite[Prop. 6.1]{CDN3}. 
By $\eta$-\'etale descent  we get the  strict quasi-isomorphism in the proposition.
 \end{proof}
 \subsubsection{Some computations.} Recall that, in a stable range and up to some universal constants, crystalline syntomic cohomology has a simple relation to de Rham cohomology. Let $\sx$ be an affine semistable formal scheme over $\so_K$. 
Let $r\geq 0$. We note that 
$\tau_{\leq r-1}(\rg_{\crr}(\sx)/F^r)\stackrel{\sim}{\to}\rg_{\crr}(\sx)/F^r$ and that the natural map $\tau_{\leq r+1}([\rg_{\crr}(\sx)]^{\phi=p^r})\to [\rg_{\crr}(\sx)]^{\phi=p^r}$ is a $p^{2r}$-quasi-isomorphism (since $1-p^s\phi, s\geq 1, $ is invertible on differentials in degree $r+s$).
\begin{proposition}{\rm(Colmez-Nizio\l, \cite[Prop. 3.12]{CN})}
\label{added}
{\rm (i)} The natural map
$$\tau_{\leq r+1}\rg_{\synt}(\sx,\Z_p(r))\to\rg_{\synt}(\sx,\Z_p(r))$$
is a $p^{2r}$-quasi-isomorphism and $H^{r+1}\rg_{\synt}(\sx,\Z_p(r))\stackrel{\sim}{\to}H^{r+1}([\rg_{\crr}(\sx)]^{\phi=p^r})$.

{\rm (ii)}
The complex $\tau_{\leq r-1}([\rg_{\crr}(\sx)]^{\phi=p^r})$ is $p^{N}$-acyclic, for a constant $N=N(e,d,p,r)$, where $e=[K:F], d=\dim \sx/\so_K$. Hence the natural map 
$\rg_{\crr}(\sx)/F^r\to \tau_{\leq r-1}(\rg_{\synt}(\sx,\Z_p(r))[1])$ is a $p^{N}$-quasi-isomorphism.

{\rm (iii)}
 The above statements are valid also modulo $p^n$. Moreover, $H^{r+1}([\rg_{\crr}(\sx_n)]^{\phi=p^r})$ is, \'etale locally on $\sx_n$, $p^{N}$-trivial, for a constant $N=N(r)$.
\end{proposition}

    Let $X\in {\rm Sm}_K$, $r\geq 0$. The distinguished triangle (\ref{seq11}) and Lemma \ref{derham} yield a natural map
 $$\partial_r: (\rg_{\dr}(X)/F^r)[-1]\to \rg_{\synt}(X,\Q_p(r)).$$
    \begin{corollary}\label{first1}
    \begin{enumerate}
    \item 
    For $i\leq r-1$, the map
    $$
    \partial_r: \wt{H}^{i-1}_{\dr}(X)\to \wt{H}^{i}_{\synt}(X,\Q_p(r))
    $$
    is an isomorphism.
    \item We have the  exact sequence
    $$
    0\to \wt{H}^{r-1}(\rg_{\dr}(X)/F^r)\lomapr{\partial_r} \wt{H}^r_{\synt}(X,\Q_p(r))\to \wt{H}^{r}_{\eet}(X,\sa_{\crr,\Q_p}^{\phi=p^r})\to \wt{H}^{r}(\rg_{\dr}(X)/F^r)
    $$
    \end{enumerate}
    \end{corollary}
    \begin{proof} 
    To prove the first claim, note that we have the long exact sequence
    $$
    \wt{H}^{i-1}_{\eet}(X,\sa_{\crr,\Q_p}^{\phi=p^r})\to \wt{H}^{i-1}(\rg_{\dr}(X)/F^r)\to \wt{H}^i\rg_{\synt}(X,\Q_p(r))\to \wt{H}^{i}_{\eet}(X,\sa_{\crr,\Q_p}^{\phi=p^r})
    $$
    If $i\leq r-1$ then $\wt{H}^{i-1}\rg_{\dr}(X)\stackrel{\sim}{\to}\wt{H}^{i-1}(\rg_{\dr}(X)/F^r)$ and 
    (1) follows from Proposition \ref{added} (which implies that $\wt{H}^{i-1}_{\eet}(X,\sa_{\crr,\Q_p}^{\phi=p^r})=0$ and $\wt{H}^{i}_{\eet}(X,\sa_{\crr,\Q_p}^{\phi=p^r})=0$).
    
    By a similar argument we get that the map
    $
    \partial_r:  \wt{H}^{r-1}(\rg_{\dr}(X)/F^r)\to \wt{H}^r_{\synt}(X,\Q_p(r))
    $ is injective which yields    the second claim of the corollary.
    \end{proof}
\subsection{Arithmetic rigid analytic  Hyodo-Kato cohomology} \label{kwak-kwak}We define here  Hyodo-Kato  cohomology of smooth rigid analytic varieties over $K$ as well as a Hyodo-Kato morphism. We do it by $\eta$-\'etale descent of crystalline Hyodo-Kato cohomology and the Hyodo-Kato morphism for semistable models.
  
\subsubsection{Hyodo-Kato cohomology.}  Let $\sa_{\hk}$ be the $\eta$-\'etale sheafification of the presheaf $\sx\mapsto \R\Gamma_{\hk}(\sx_0):=\rg_{\crr}(\sx_0/\so^0_{F_L})_{\Q_p}$ on $\sm_{K}^{ss}$. Here $\sx$ is a semistable formal model over $\so_L$, $[L:K]< \infty,$ $L=K_{\sx}$,  and $F_L$ is the maximal absolutely unramified subfield of $L$. 
The sheaf $\sa_{\hk}$ 
     is a sheaf of dg $ F$-algebras on ${\rm Sm}_{K,\eet}$ equipped with a $\phi$-action and a derivation $N$ such that $N\phi=p\phi N$. For $X\in {\rm Sm}_{K}$, set $\R\Gamma_{\hk}(X):=\R\Gamma_{\eet}(X,\sa_{\hk})$. Equip it with a topology in the usual way, via $\eta$-\'etale descent, from the natural topology on $\R\Gamma_{\hk}(\sx_0)$.
     
\subsubsection{Convergent cohomology.}  Let $\sa_{\conv}$ be the $\eta$-\'etale sheafification of the presheaf
 \footnote{Here $\rg_{\conv}(\sx_1/\so_L^{\times})$ (and later $\rg_{\rig}(\sx_1/\so_L^{\times})$) are defined following the construction of Grosse-Kl\"onne \cite[1.1-1.4]{GKFr} 
 by taking rigid analytic tubes (resp. dagger tubes).} $\sx\mapsto \rg_{\conv}(\sx_1/\so_L^{\times})$, $L=K_{\sx}$, on $\sm_{K,\eet}^{\sem}$.  For $X\in {\rm Sm}_K$, we set $\rg_{\conv}(X):=\rg_{\eet}(X,\sa_{\conv})$. It is a dg $K$-algebra. We equip it with the topology induced  by $\eta$-\'etale descent from the topology of the $\R\Gamma_{\conv}(\sx_{1}/\so_L^{\times})$'s. We have natural (strict) quasi-isomorphisms
 $$
 \sa_{\conv}\simeq \sa_{\dr},\quad \rg_{\conv}(X)\simeq\rg_{\dr}(X)
 $$   
 induced by the quasi-isomorphisms $\rg_{\conv}(\sx_1/\so_L^{\times})\simeq \rg_{\dr}(\sx_L)$ that hold because $\sx$ is log-smooth over $\so_L^{\times}$. 
\subsubsection{Hyodo-Kato morphism.}  To define the Hyodo-Kato quasi-isomorphism we will use the  original Hyodo-Kato quasi-isomorphism defined for quasi-compact formal schemes in \cite{HK} (see also \cite{OM}). 
     We will  describe  it now  in some detail. Denote by $r^+_F$ the algebra $\so_F[[T]]$
with the log-structure associated to $T$.  Sending $T$ to $p$ induces
a surjective morphism $r^+_F\to \so_F^{\times}$.    We denote by $r^{\rm PD}_F$ the $p$-adic divided power envelope of $r^+_F$ with respect to the kernel of this morphism. Frobenius is defined by $T\mapsto T^p$, monodromy is a $\so_F$-linear derivation given  by $T\mapsto T$. We will skip the subscript $F$ if there is no danger of confusion.

\smallskip
   (i) {\em Local definition.}  Assume that we have an admissible  semistable formal scheme  $\sx$ over $\so_K$. We will work in the classical derived category. 
   Recall  that  the Frobenius
   $$
   r^{\rm PD}_{n,\phi}\otimes^L_{r^{\rm PD}_{n}}\R\Gamma_{\crr}(\sx_0/r^{\rm PD}_{n})\to \rg_{\crr}(\sx_0/r^{\rm PD}_{n}),\quad 
   \so_{F,n,\phi}\otimes^L_{\so_{F,n}}\rg_{\crr}(\sx_0/\so^0_{F,n})\to \rg_{\crr}(\sx_0/\so_{F,n}^0)
   $$
   has a $p^N$-inverse, for  $N=N(d)$, $d=\dim \sx_0$. This is proved in \cite[2.24]{HK}.
   Recall also that   the projection $p_0:  \rg_{\crr}(\sx_0/r^{\rm PD}_{n})\to \rg_{\crr}(\sx_0/\so_{F,n}^0)$, $T\mapsto 0$,
has a functorial (for maps between formal schemes and a change of $n$) and Frobenius-equivariant $p^{N_{\iota}}$-section, $N_{\iota}=N(d)$,
   $$\iota_n: 
   \rg_{\crr}(\sx_0/\so_{F,n}^0)\to  \rg_{\crr}(\sx_0/r^{\rm PD}_{n}),
    $$
 i.e., $p_0\iota_n=p^{N_{\iota}}$. 
 This  follows easily from the proof of Proposition 4.13 in \cite{HK}; 
 the key point being that the Frobenius on $\rg_{\crr}(\sx_0/\so_{F,n}^0)$ is close to a quasi-isomorphism and the Frobenius on the ${\rm PD}$-ideal of $r^{\rm PD}$ is close to zero. Moreover, the resulting map 
 \begin{equation}
 \label{section1}
 \iota_n: 
   \rg_{\crr}(\sx_0/\so_{F,n}^0)\otimes^L_{\so_{F,n}}r^{\rm PD}_{n} \to  \rg_{\crr}(\sx_0/r^{\rm PD}_{n})
    \end{equation}
 is a $p^N$-quasi-isomorphism, $N=N(d)$, \cite[Lemma 5.2]{HK} and so is the composite
 $$
 p_{p} \iota_n: \rg_{\crr}(\sx_0/\so_{F,n}^0)\to  \rg_{\crr}(\sx_0/\so^{\times}_{F,n}),
 $$
 where   the projection $p_p:  \rg_{\crr}(\sx_0/r^{\rm PD}_{n})\to \rg_{\crr}(\sx_0/\so_{F,n}^{\times})$ is defined by $T\mapsto p$. 
Taking $\holim_n$ of the last map we obtain a map 
 $$
  p_{p} \iota: \rg_{\crr}(\sx_0/\so_{F}^0) \to  \rg_{\crr}(\sx_0/\so^{\times}_{F})
 $$
 that  is a $p^N$-quasi-isomorphism, $N=N(d)$. 
 
 We define the Hyodo-Kato map as the composition
 \begin{align}
 \label{definition-HK}
\iota_{\hk}: \rg_{\crr}(\sx_0/\so_{F}^0)_{F} & \verylomapr{p^{-N_{\iota}}p_{p} \iota} \rg_{\crr}(\sx_0/\so^{\times}_{F})_{F}\to \rg_{\crr}(\sx_0/\so^{\times}_{K})_{K}\stackrel{\sim}{\leftarrow}\rg_{\conv}(\sx_0/\so^{\times}_{K})\\
 & \stackrel{\sim}{\leftarrow}
 \rg_{\conv}(\sx_1/\so^{\times}_{K})\stackrel{\sim}{\leftarrow}\rg_{\dr}({\sx}_K).\notag
 \end{align}
 The fourth map  is actually a natural isomorphism by the invariance under infinitesimal thickenings of convergent cohomology \cite[0.6.1]{Og}.
 The induced map 
 $\iota_{\hk}:\rg_{\crr}(\sx_0/\so_{F}^0)_{F}\otimes_{F}K\to \rg_{\dr}(\sx_K)$ 
 is a  strict quasi-isomorphism.

\smallskip
  (ii) {\em Globalization.}  Let now $X$ be a smooth rigid analytic variety over $K$. Since the computation, leading to the existence of the section $\iota$, in Proposition 4.13 in \cite{HK} can be done on the big topos as long as we can control the dimension of the schemes involved, the above Hyodo-Kato map can be lifted to a Hyodo-Kato map
   $$
   \iota_{\hk}:\sa_{\hk}\to \sa_{\dr}
   $$
   in the classical derived category of \'etale sheaves on $X$. It induces  the Hyodo-Kato map
   \begin{equation}
   \label{cieplo}
   \iota_{\hk}: \rg_{\hk}(X)\to\rg_{\dr}(X).
   \end{equation}
   \begin{proposition}{\rm ({\em Local-global compatibility})}
   \label{hypercov}
   For a semistable formal scheme  $\sx$ over $\so_K$, the canonical map
   \begin{equation}
   \label{hypercov0}
   \rg_{\hk}(\sx_0)\to\rg_{\hk}(\sx_K)
   \end{equation}
   is a strict quasi-isomorphism.
   \end{proposition}
   \begin{proof}
   The proof of Proposition 3.18 in \cite{NN} goes through practically verbatim. Key points: the de Rham
   analog of (\ref{hypercov0})   holds plus   we have Galois descent for both sides of (\ref{hypercov0}) that allows us to deal with the field extensions appearing in the construction of local semistable models.
   \end{proof}

   \begin{remark}The above definition of the Hyodo-Kato quasi-isomorphism was normalized (at $p$) so that it is functorial. 
   A more customary definition depends on the uniformizer $\varpi$ (one basically proceeds as above but using the ${\rm PD}$-envelope of the map $\so_F\{T\}\to \so_K, T\mapsto \varpi,$
   instead of $r^{\rm PD}_F$) and hence it is not functorial.
   \end{remark}

\subsubsection{Arithmetic $r^{\rm PD}$-cohomology.}\label{tea1}
We define the  $r^{\rm PD}$-cohomology of smooth rigid analytic varieties over $K$ by $\eta$-\'etale descent of 
 the $r^{\rm PD}$-cohomology of semistable models.
  
 Let $\sa_{\rm PD}$ be the $\eta$-\'etale sheafification of the presheaf $\sx\mapsto \R\Gamma_{\crr}(\sx_0/r^{\rm PD}_{L})_{\Q_p}$ on $\sm_{K}^{ss}$. Here $\sx$ is an admissible semistable formal scheme over $\so_L$, $L=K_{\sx}$. We wrote $r^{\rm PD}_{L}$ for the $r^{\rm PD}$-ring corresponding to $F_L$.  Let ${ \srr}^{\rm PD}$ be the $\eta$-\'etale sheafification  of  the presheaf $\sx\mapsto r^{\rm PD}_{L,\Q_p}$ on $\sm_{K}^{ss}$. 
The sheaf $\sa_{\rm PD}$ 
     is a sheaf of dg $ \srr^{\rm PD}_{\Q_p}$-algebras on ${\rm Sm}_{K,\eet}$ equipped with a $\phi$-action and a derivation $N$, compatible with the derivation on $\srr^{\rm PD}$,  such that $N\phi=p\phi N$. For $X\in {\rm Sm}_{K}$, set $\R\Gamma_{\rm PD}(X):=\R\Gamma_{\eet}(X,\sa_{\rm PD})$. Equip it with a topology in the usual way, via $\eta$-\'etale descent, from the natural topology on the $\R\Gamma_{\crr}(\sx_0/r^{\rm PD}_L)_{\Q_p}$'s.
\begin{proposition} {\rm ({\em Local-global compatibility})} For a semistable formal model  $\sx$ over $\so_K$, the canonical map
$$
\rg_{\crr}(\sx_0/r^{\rm PD}_K)_{\Q_p}\to\rg_{\rm PD}(\sx_K)
$$
is a strict quasi-isomorphism. 
\end{proposition}
\begin{proof} It suffices to show that, for any $\eta$-\'etale hypercovering $\su_{\cdot}$ of $\sx$  from $\sm^{\sem}_K$ (we may assume that in every degree of the hypercovering we have a  quasi-compact formal scheme), the natural map
\begin{equation*}
 \R\Gamma_{\crr}(\sx_0/r^{\rm PD})_{\Qp}\to  \R\Gamma_{\crr}(\su_{\cdot,0}/r^{\rm PD}_{{L_{\cdot}}})_{\Q_p}
 \end{equation*}
is a strict quasi-isomorphism (modulo taking a refinement of $\su_{\cdot}$). 
Recall that the $p^N$-quasi-isomorphism $\iota$ from (\ref{section1}) yields a strict quasi-isomorphism ($\wh{\otimes}^R$ denotes the right derived functor of the tensor product)
  \begin{equation}
  \label{twisted-section}
  s=p^{-N_{\iota}}\iota: \quad \rg_{\hk}(\sx_0)\wh{\otimes}^R_Fr^{\rm PD}_{K,\Q_p}\stackrel{\sim}{\to} \rg_{\crr}(\sx_0/r^{\rm PD}_K)_{\Q_p}.
  \end{equation}
Using it we get the following commutative diagram
$$
\xymatrix@C=.6cm@R=.6cm{ 
\R\Gamma_{\crr}(\sx_0/r_K^{\rm PD})_{\Qp}\ar[r] &  \R\Gamma_{\crr}(\su_{\cdot,0}/r^{\rm PD}_{{L_{\cdot}}})_{\Q_p}\\
\rg_{\hk}(\sx_0)\wh{\otimes}^R_Fr^{\rm PD}_{K,\Q_p}\ar[r] \ar[u]_{s}^{\wr} & \rg_{\hk}(\su_{\cdot,0})\wh{\otimes}^R_{F_{L_{\cdot}}}r^{\rm PD}_{{L_{\cdot}},\Q_p}\ar[u]_{s_{\cdot}}^{\wr}
}
$$
Since $\rg_{\hk}(\su_{\cdot,0})\wh{\otimes}^R_{F_L}r^{\rm PD}_{L,\Q_p}
\simeq \rg_{\hk}(\su_{\cdot,0})\wh{\otimes}^R_{F}r^{\rm PD}_{K,\Q_p}
$ and since, by Proposition \ref{hypercov}, the natural map $\rg_{\hk}(\sx_0)\to \rg_{\hk}(\su_{\cdot,0})$ is a strict quasi-isomorphism so is the bottom map in the above diagram. It follows that  the top map is also a strict-quasi-isomorphism, as wanted. 
\end{proof}
 \subsection{Geometric rigid analytic  Hyodo-Kato cohomology}  \label{berkeley11} 
We will now define the  Hyodo-Kato  cohomology of smooth rigid analytic varieties over $C$. We will do it by $\eta$-\'etale descent of crystalline Hyodo-Kato cohomology of basic semistable models.

  \subsubsection{Definition and basic properties}\label{part-i} Let  ${f}:\sx\to \Spf(\so_C)^{\times}$ be a semistable  formal model.
 Suppose  that ${f}$ is the base change of 
 a semistable  formal model   ${f}_L:\sx_{\so_L}\to \Spf(\so_{L})^{\times}$ by $\theta: \Spf(\so_C)^{\times}\to\Spf(\so_{L})^{\times}$, for a finite extension $L/K$.
 That is, we have a map $\theta_{L}: \sx\to \sx_{\so_L}$ such that the square $({f},{f}_L,\theta,\theta_{L})$ is Cartesian.
 In the algebraic setting (algebraic schemes and $\ovk$ in place of $C$) such data $(L,\sx_{\so_L},\theta_{L})$ clearly form a filtered set.
 In our analytic case this is also 
 the case for the system 
$$\Sigma=\big\{(L,\sx_{\so_{L,1}},\theta_{L})\big\}$$
corresponding to the reduction modulo $p$ of such data\footnote{This is because the schemes $\sx_{\so_L,1}$ from above are algebraic.}, i.e., a system in which objects are reductions $(L,\sx_{\so_{L,1}},\theta_{L})$ modulo $p$ of the tuples $(L,\sx_{\so_{L}},\theta_{L})$ as above but morphisms are morphisms between the reduced objects.

\smallskip
 (i) {\em Hyodo-Kato cohomology.} For a morphism of tuples $(L\pri,\sx_{\so_{L\pri},1}\pri,\theta\pri_{L^\prime})\to (L,\sx_{\so_L,1},\theta_{L})$ from $\Sigma$,
 we   have a canonical base change identification compatible with $\phi$-action (crystalline unramified base change)
$$\R\Gamma_{\hk}(\sx_{\so_L,0})\otimes_{F_L}F_{L\pri}\stackrel{\sim}{\to}
 \R\Gamma_{\hk}(\sx\pri_{\so_{L^\prime},0}).
$$
We set\footnote{Everything here and below is done in the derived $\infty$-category $\sd(C_{\Q_p})$.}
$$ \R\Gamma_{\hk}(\sx_1):= \hocolim_{\Sigma}\R\Gamma_{\hk}(\sx_{\so_L,0}).
$$
$ \R\Gamma_{\hk}(\sx_1)$ is a dg $ F^{\nr}$-algebra\footnote{The field $F^{\nr}$ is equipped with the inductive limit topology. Later on we will use the same type of topology for $\ovk$.}  equipped with a $\phi$-action and a derivation $N$ such that $N\phi=p\phi N$.
It is functorial with respect to  $\sx$: note that the restriction of a morphism  $\sx\to \sy$ to a morphism $\sx_1\to\sy_1$ is defined over a finite extension of $K$.  Let $\sa_{\hk}$ be the $\eta$-\'etale sheafification of the presheaf $\sx\mapsto \rg_{\hk}(\sx_1)$ on $\sm_C^{\sem,b}$. For $X\in {\rm Sm}_C$, we set $\rg_{\hk}(X):=\rg_{\eet}(X,\sa_{\hk})$. It is a dg $F^{\nr}$-algebra equipped with a Frobenius, monodromy action, and a continuous action of
 $\sg_K$ if $X$ is defined over $K$ (this action is smooth, i.e., the stabilizer of every element is an open subgroup of $\sg_K$, if $X$ is quasi-compact; in general, it is only ``pro-smooth"). We equip it with the topology induced by $\eta$-\'etale descent from the topology of the $\R\Gamma_{\hk}(\sx_{\so_L,0})$'s. 
 
\smallskip
 (ii) {\em Convergent cohomology.}  Let $\sa_{\conv}$ be the $\eta$-\'etale sheafification of the presheaf
  $\sx\mapsto \rg_{\conv}(\sx_1/\so_C^{\times})$ on $\sm_{C,\eet}^{\sem,b}$.  For $X\in {\rm Sm}_C$, we set $\rg_{\conv}(X):=\rg_{\eet}(X,\sa_{\conv})$. It is a dg $C$-algebra  equipped with a continuous action of
 $\sg_K$. We equip it with the topology induced  by $\eta$-\'etale descent from the topology of the $\R\Gamma_{\conv}(\sx_{1}/\so_C^{\times})$'s. We have natural (strict) quasi-isomorphisms
 $$
 \sa_{\conv}\simeq \sa_{\dr},\quad \rg_{\conv}(X)\simeq\rg_{\dr}(X).
 $$
 
  Let $\sa_{\conv,\ovk}$ be the \'etale sheafification of the presheaf $\sx\mapsto \rg_{\conv,\ovk}(\sx_1)$ on $\sm_{C,\eet}^{\sem,b}$, where we set
  $$ \R\Gamma_{\conv,\ovk}(\sx_1):= \hocolim_{\Sigma}\R\Gamma_{\conv}(\sx_{\so_L,1}/\so_L^{\times})
  $$
  in the notation from above. For $X\in {\rm Sm}_C$, we set $\rg_{\conv,\ovk}(X):=\rg_{\eet}(X,\sa_{\conv,\ovk})$. It is a dg $\ovk$-algebra  equipped with a continuous action of
 $\sg_K$ if $X$ is defined over $K$ (this action is smooth if $X$ is quasi-compact). We equip it with the topology induced  by $\eta$-\'etale descent from the topology of the $\R\Gamma_{\conv}(\sx_{\so_L,1}/\so_L^{\times})$'s. 
 There are natural continuous morphisms
$$
  \sa_{\conv,\ovk}  \to \sa_{\conv},\quad \rg_{\conv,\ovk}(X)\to \rg_{\conv}(X).
$$
 \begin{remark}
 Instead of $\R\Gamma_{\conv,\ovk}(\sx_1)$ above we could have used 
  $$ \R\Gamma_{\conv,F^{\nr}}(\sx_1):= \hocolim_{\Sigma}\R\Gamma_{\conv}(\sx_{\so_L,1}/\so_{F_L}^{\times}).
  $$
  This would give a natural $F^{\nr}$-structure on de Rham cohomology (see Proposition \ref{etale-descent} below).
 \end{remark}
 
\smallskip
  (iii) {\em $r^{\rm PD}$-cohomology.} Let $\sa_{\rm PD}$ be the $\eta$-\'etale sheafification of the presheaf $\sx\mapsto \rg_{\rm PD}(\sx_1)$ on $\sm_{C,\eet}^{\sem,b}$, where we set
  $$ \R\Gamma_{\rm PD}(\sx_1):= \hocolim_{\Sigma}\R\Gamma_{\crr}(\sx_{\so_L,0}/r^{\rm PD}_L)_{\Q_p}.
  $$
  in the notation from above. For $X\in {\rm Sm}_C$, we set $\rg_{\rm PD}(X):=\rg_{\eet}(X,\sa_{\rm PD})$. Set
  $ r^{\rm PD}_{\ovk}:=r^{\rm PD}_{F}\otimes_{\so_F}\so_{F^{\nr}}:=\dirlim_L(r^{\rm PD}_F\otimes_{\so_F}\so_{F_L})
  $, $[L:K]<\infty$. $\rg_{\rm PD}(X)$
   is a dg $r^{\rm PD}_{\ovk,\Q_p}$-algebra  equipped with a continuous action of
 $\sg_K$ if $X$ is defined over $K$ (this action is smooth if $X$ is quasi-compact). We equip it with the topology induced  by $\eta$-\'etale descent  from the topology of the $\R\Gamma_{\crr}(\sx_{\so_L,0}/r^{\rm PD}_L)_{\Q_p}$'s. \\

 \subsubsection{Hyodo-Kato quasi-isomorphisms} We keep the set-up from Section \ref{part-i}.
  The Hyodo-Kato morphism from  (\ref{definition-HK}):
  \begin{equation}
  \label{HK-jablko}
  \iota_{\hk}: \rg_{\hk}(\sx_{\so_L,0})\to \rg_{\conv}(\sx_{\so_L,1}/\so_L^{\times}),\quad  \iota_{\hk}: \rg_{\hk}(\sx_{\so_L,0})\otimes_{F_L}L\stackrel{\sim}{\to }\rg_{\conv}(\sx_{\so_L,1}/\so_L^{\times})
  \end{equation}
   is compatible with morphisms in $\Sigma$ and taking  its homotopy colimit yields the first of the following two natural strict quasi-isomorphisms (called again the {\em Hyodo-Kato quasi-isomorphisms})
\begin{align}
\label{HK-geometric}
 & \iota_{\hk}:\quad \R\Gamma_{\hk}(\sx_{1})\otimes_{F^{\nr}}{\ovk}\stackrel{\sim}{\to} \R\Gamma_{\conv,\ovk}(\sx_1)=\hocolim_{\Sigma}\R\Gamma_{\conv}(\sx_{\so_L,1}/\so_L^{\times}),\\
& \iota_{\hk}:\quad \R\Gamma_{\hk}(\sx_1)\wh{\otimes}^R_{F^{\nr}}{C}\stackrel{\sim}{\to} \R\Gamma_{\dr}(\sx_{C}).\notag
\end{align}
By definition,   $ \R\Gamma_{\hk}(\sx_{1})\otimes_{F^{\nr}}{\ovk}:=\hocolim_L(\R\Gamma_{\hk}(\sx_{1})\otimes_{F^{\nr}}L)$, the homotopy colimit taken over fields $L, [L:F^{\nr}]<\infty$.  We have
$ \R\Gamma_{\hk}(\sx_{1})\otimes_{F^{\nr}}{\ovk}\simeq \hocolim_{\Sigma}(\R\Gamma_{\hk}(\sx_{\so_L,0})\otimes_{F_L}L)$.
In the second Hyodo-Kato morphism in (\ref{HK-geometric}), by definition\footnote{See \cite[Sec. 2.1]{CDN3} for a quick review of basic facts concerning tensor products in the category $C_{\Q_p}$.}, 
$$
 \R\Gamma_{\hk}(\sx_{1})\wh{\otimes}^R_{F^{\nr}}C:=\hocolim_{\Sigma}(\R\Gamma_{\hk}(\sx_{\so_L, 0})\wh{\otimes}^R_{F_L}C).
$$
We note that all the maps in the homotopy colimits are strict quasi-isomorphisms. 
The Hyodo-Kato morphism itself is induced from the Hyodo-Kato strict quasi-isomorphism (\ref{HK-jablko}):
$$
\hocolim_{\Sigma}(\R\Gamma_{\hk}(\sx_{\so_L, 0})\wh{\otimes}^R_{F_L}C)\stackrel{\sim}{\to}\hocolim_{\Sigma}(\R\Gamma_{\conv}(\sx_{\so_L, 1}/\so_L^{\times})\wh{\otimes}^R_{L}C)
$$
and the strict  quasi-isomorphisms
\begin{align*}
\hocolim_{\Sigma}(\R\Gamma_{\conv}(\sx_{\so_L,1}/\so_L^{\times})\wh{\otimes}_{L}^RC)  \stackrel{\sim}{\to} \R\Gamma_{\conv}(\sx_{1}/\so_{C}^{\times})
\simeq \R\Gamma_{\dr}(\sx_{C}).
\end{align*}
The first quasi-isomorphism is given by base change. We note here that, since $\rg_{\conv}(\sx_{\so_L,1}/\so_L^{\times})$ is a complex of Banach spaces, the completed tensor product with $C$ is exact. 

  Similarly, for $\sx$ as at the beginning of Section \ref{part-i}, the strict quasi-isomorphism (\ref{twisted-section})  yields a strict quasi-isomorphism
  \begin{equation}
  \label{zimno-berkeley}
   s:\ \rg_{\hk}(\sx_1)\wh{\otimes}^R_{F^{\nr}}r^{\rm PD}_{\ovk,\Q_p}\stackrel{\sim}{\to} \rg_{\rm PD}(\sx_1),
  \end{equation}
  where we set $$\rg_{\hk}(\sx_1)\wh{\otimes}^R_{F^{\nr}}r^{\rm PD}_{\ovk,\Q_p}:=
  \hocolim_{\Sigma}(\R\Gamma_{\hk}(\sx_{\so_L,0})\wh{\otimes}^R_{F_L}r^{\rm PD}_{L,\Q_p}).
  $$
  We also get ($T\mapsto 0$)
  $$
  \rg_{\rm PD}(\sx_1)\otimes_{r^{\rm PD}_{\ovk,\Q_p}}F^{\nr}\simeq \rg_{\hk}(\sx_1), \quad 
  $$
where we set $$  \rg_{\rm PD}(\sx_1)\otimes_{r^{\rm PD}_{\ovk,\Q_p}}F^{\nr}:=
  \hocolim_{\Sigma}(\R\Gamma_{\crr}(\sx_{\so_L,0}/r^{\rm PD}_{L})_{\Q_p}\wh{\otimes}^R_{r^{\rm PD}_{L,\Q_p}}F_L).
$$

 Varying $\sx$ in the above constructions we obtain the (Hyodo-Kato) maps
 $$
\iota_{\hk}:  \sa_{\hk}\to \sa_{\conv,\ovk},\quad \iota_{\hk}: \sa_{\hk}\to \sa_{\dr}, \quad s: \sa_{\hk}\to \sa_{\rm PD}
 $$
 of sheaves on ${\rm Sm}_{C,\eet}$. We claim that,
for $X\in {\rm Sm}_C$, they induce the natural  (Hyodo-Kato) strict quasi-isomorphisms 
  \begin{equation}
  \label{HK-two}
   \iota_{\hk}:\R\Gamma_{\hk}(X)\wh{\otimes}_{F^{\nr}}{\ovk}\stackrel{\sim}{\to} \R\Gamma_{\conv,\ovk}(X),\quad
 \iota_{\hk}: \R\Gamma_{\hk}(X)\wh{\otimes}^R_{{F}^{\nr}}{C}\stackrel{\sim}{\to} \R\Gamma_{\dr}(X),\quad
 \R\Gamma_{\hk}(X)\wh{\otimes}^R_{{F}^{\nr}}{r^{\rm PD}_{\ovk,\Q_p}}\stackrel{\sim}{\to} \R\Gamma_{\rm PD}(X).
  \end{equation}
  Here we set\footnote{The notation is ad hoc and rather awful here but we hope that it is self-explanatory.} 
 \begin{align}
 \label{HK-referee}
  & \R\Gamma_{\hk}(X)\wh{\otimes}_{F^{\nr}}{\ovk}:=\hocolim((\R\Gamma_{\hk}{\otimes}_{F^{\nr}}{\ovk})(\su_{\cdot,1})),\\
 & \R\Gamma_{\hk}(X)\wh{\otimes}^R_{F^{\nr}}C:=\hocolim((\R\Gamma_{\hk}\wh{\otimes}^R_{F^{\nr}}{C})(\su_{\cdot,1})),\notag\\
 &  \R\Gamma_{\hk}(X)\wh{\otimes}^R_{{F}^{\nr}}r^{\rm PD}_{\ovk,\Q_p}:=\hocolim((\R\Gamma_{\hk}\wh{\otimes}^R_{F^{\nr}}r^{\rm PD}_{\ovk,\Q_p})(\su_{\cdot,1})),\notag
 \end{align}
 where the homotopy colimit is taken over $\eta$-\'etale hypercoverings $\su_{\cdot}$ from $\sm^{\sem, b}_C$.  We note that we have 
\begin{equation}
\label{etale-berkeley}\R\Gamma_{\conv,\ovk}(X)\simeq \hocolim\R\Gamma_{\conv,\ovk}(\su_{\cdot,1}),\quad
\R\Gamma_{\rm PD}(X)\simeq \hocolim\R\Gamma_{\rm PD}(\su_{\cdot,1}).
\end{equation}
Indeed, by Proposition \ref{etale-descent} below (there is no circular reasoning here) we have $$\hocolim\R\Gamma_{\conv,\ovk}(\su_{\cdot,1})\stackrel{\sim}{\to} \hocolim\R\Gamma_{\conv,\ovk}(\su_{\cdot,C}),\quad \hocolim\R\Gamma_{\rm PD}(\su_{\cdot,1})\stackrel{\sim}{\to} \hocolim\R\Gamma_{\rm PD}(\su_{\cdot,C}).
$$
 Hence (\ref{etale-berkeley}) follows from the fact that $\R\Gamma_{\conv,\ovk}(X)$ and $\R\Gamma_{\rm PD}(X)$ satisfy $\eta$-\'etale descent. Having (\ref{etale-berkeley}),
the first strict quasi-isomorphism in (\ref{HK-two}) follows from the first Hyodo-Kato strict quasi-isomorphism in (\ref{HK-geometric}). The second Hyodo-Kato strict quasi-isomorphism in (\ref{HK-geometric})  implies easily the second strict quasi-isomorphism we wanted. The third strict quasi-isomorphism follows from (\ref{zimno-berkeley}).

\subsubsection{Local-global compatibility and comparison results.} Having the  quasi-isomorphisms (\ref{HK-two}) we can prove the following 
comparison result (where the tensor products in (2) and (3) are defined as in (\ref{HK-referee}): 
\begin{proposition}\label{etale-descent}
\begin{enumerate}
\item 
  Let  $\sx\in \sm^{\sem,b}_C$.
The natural maps
  \begin{align*}
&  \R\Gamma_{\hk}(\sx_1)\to\rg_{\hk}(\sx_C),   \quad \rg_{\conv,\ovk}(\sx_1)\to \rg_{\conv,\ovk}(\sx_C), \\
& \rg_{\conv}(\sx_1)\to \rg_{\conv}(\sx_C), \quad 
  \rg_{\rm PD}(\sx_1)\to \rg_{\rm PD}(\sx_C)
  \end{align*}
  are strict quasi-isomorphisms. 
  \item For $X\in {\rm Sm}_C$, we have  natural strict quasi-isomorphisms
  $$
  \rg_{\conv,\ovk}(X)\wh{\otimes}^R_{\ovk}C\stackrel{\sim}{\to}\rg_{\conv}(X) \simeq \rg_{\dr}(X).
  $$
  \item For $X\in {\rm Sm}_K$, we have a natural strict quasi-isomorphism
  $$
  \rg_{\dr}(X)\wh{\otimes}_{K}\ovk\simeq \rg_{\conv,\ovk}(X_C).
  $$
\end{enumerate}
\end{proposition}
\begin{proof}For the first claim,  it suffices to show that, for any $\eta$-\'etale hypercovering $\su_{\cdot}$ of $\sx$  from $\sm^{\sem, b}_C$, the natural maps
\begin{equation}
\label{morning}
 \R\Gamma_{?}(\sx_1)\to  \R\Gamma_{?}(\su_{\cdot,1}),\quad ?=\hk, \{\conv,\ovk\}, \conv, {\rm PD}, 
\end{equation}
are strict quasi-isomorphisms (modulo taking a refinement of $\su_{\cdot}$). We may assume that in every degree of the hypercovering we have a finite number of formal models. For the Hyodo-Kato case, it suffices to show the strict quasi-isomorphism after we tensor both sides with $\ovk$ over $F^{\nr}$. But then we can use the Hyodo-Kato quasi-isomorphism
(\ref{HK-geometric}) to reduce   to the case of $\{\conv, \ovk\}$ in (\ref{morning}).

For that case, note that our map is  strictly quasi-isomorphic to a map
$$
 \R\Gamma_{\dr}(\sx_{L})\otimes_L\ovk\to  (\R\Gamma_{\dr}\otimes_{L_{\cdot}}\ovk)(\su_{\cdot,{L_{\cdot}}}).
$$
The rather ugly notation for the hypercovering just underscores the fact that the field over which the particular formal schemes split varies. Passing to cohomology ($\wt{H}(-)$-cohomology) and then to a truncated hypercovering we can assume that all the rigid spaces and maps involved are defined over a common field $K^{\prime}$, a finite extension of $L$. We get a strict  quasi-isomorphism by  \'etale descent for de Rham cohomology.  The cases of ${\rm PD}$- and $\conv$-cohomology, can be reduced to that of Hyodo-Kato and de Rham cohomologies via the strict quasi-isomorphisms $\rg_{\rm PD}(X)\simeq \rg_{\hk}(X)\wh{\otimes}_{F^{\nr}}^Rr^{\rm PD}_{\ovk,\Q_p}$ and $\rg_{\conv}(X)\simeq \rg_{\dr}(X)$, respectively. 

  For the second claim of the proposition, it suffices to show that for an $\eta$-\'etale  hypercovering $\su_{\cdot}$ of $X$ from $\sm^{\sem, b}_C$, we have a strict quasi-isomorphism
 $$
 ( \rg_{\conv,\ovk}\wh{\otimes}^R_{\ovk}C)(\su_{\cdot,1})\simeq \rg_{\dr}(\su_{\cdot,C}).
  $$
  It suffices to argue degree-wise. Hence it suffices to show that, for a semistable formal model $\su$ over $\so_E$, $[E:L]<\infty$, the first  top horizontal arrow in the following diagram is a strict quasi-isomorphism:
$$
\xymatrix@R=.6cm{
  \rg_{\conv,\ovk}(\su_{\so_C,1})\wh{\otimes}^R_{\ovk}C \ar[r] &  \rg_{\conv}(\su_{\so_C,1})\ar[r]^-{\sim}  & \rg_{\dr}(\su_{C})\\
    \rg_{\conv}(\su_{\so_E,1})\wh{\otimes}^R_{E}C\ar[ur]^-{\sim} \ar[u]^{\wr} \ar[r]^-{\sim} & \rg_{\dr}(\su_{\rm PD})\wh{\otimes}^R_EC\ar[ur]^{\sim}.
}
$$
Since this diagram clearly commutes and the other  arrows are strict quasi-isomorphisms, this is evident.

For the third  claim of the proposition,  it suffices to show that, for any $\eta$-\'etale hypercovering $\su_{\cdot}$ of $X_C$  from $\sm^{\sem, b}_C$, the natural map
\begin{equation}
\label{morningg}
 \R\Gamma_{\dr}(X)\wh{\otimes}_K\ovk\to  \R\Gamma_{\conv,\ovk}(\su_{\cdot,1})
 \end{equation}
is a  strict quasi-isomorphism (modulo taking a refinement of $\su_{\cdot}$). We can assume that $\su_{\cdot}$ has formal models  in every degree. Then both sides of (\ref{morningg}) can be computed by 
$
(\R\Gamma_{\dr}\otimes_{L_{\cdot}}\ovk)(\su_{\cdot,L_{\cdot}})
$ proving what we wanted.
\end{proof}

 \subsubsection{Galois descent.}
 The following proposition shows that Hyodo-Kato cohomology satisfies Galois descent.
\begin{proposition}
\label{descent}
Let $X\in {\rm Sm}_K$. The natural projection $\varepsilon: X_{C,\eet}\to X_{\eet}$ defines pullback strict quasi-isomorphisms
\begin{equation}
\label{qis11}
 \varepsilon^*: \rg_{\hk}(X)\stackrel{\sim}{\to}\rg_{\hk}(X_C)^{\sg_K},\quad  \varepsilon^*: \rg_{\conv}(X)\stackrel{\sim}{\to}\rg_{\conv,\ovk}(X_C)^{\sg_K},\quad \varepsilon^*: \rg_{\rm PD}(X)\stackrel{\sim}{\to}\rg_{\rm PD}(X_C)^{\sg_K}.
 \end{equation}
\end{proposition}
\begin{remark}
\label{Galois-sense}
Here,  we denoted by $\rg_{\hk}(X_C)^{\sg_K}$, etc.,  the complex obtained by taking the $\sg_K$-fixed points of a representative of $\rg_{\hk}(X_C)$. 
This definition makes sense, i.e., two strictly quasi-isomorphic complexes representing $\rg_{\hk}(X_C)$ give two strictly quasi-isomorphic complexes representing $\rg_{\hk}(X_C)^{\sg_K}$. Or, otherwise speaking, taking a cone of the given quasi-isomorphism, for a complex $T:=T^0\to T^1\to T^2\to\cdots$ such that each $T^i$ is a direct sum of products of LB-spaces with a smooth action of $\sg_K$, the complex $T^{\sg_K}$ is strictly exact.
 Indeed,  since the complex $T$ is strictly exact, for all $i$, we have the strictly exact sequence
  \begin{equation}
  \label{exact-11}
  0\to \ker d_i\to T^i\to \ker d_{i+1}\to 0,
  \end{equation}
and we need to show that the induced sequence
  \begin{equation}
  \label{exact-12}
  0\to (\ker d_i)^{\sg_K}\to (T^i)^{\sg_K}\to (\ker d_{i+1})^{\sg_K}\to 0
  \end{equation}
 is exact. We note that there exists a normalized trace function
  $$\tr: T^i\to (T^i)^{\sg_K},\quad x\mapsto \varinjlim_{L \subset \ovk}\frac{1}{[L:K]}\sum _{\sigma\in \Gal(L/K)} \sigma(x).
 $$
 This is well-defined because $T^i$ is a finite direct sum of products of smooth $\sg_K$-modules and on a smooth $\sg_K$-module the limit in the formula stabilizes. Let now $x\in (\ker d_{i+1})^{\sg_K}$. Since the sequence (\ref{exact-11}) is exact, there exists $y\in T^i$ mapping to $x$. But then $\tr(y)$ maps to $\tr(x)=x$. Since $\tr(y)\in (T^i)^{\sg_K}$ this means that the sequence (\ref{exact-12}) is exact, as wanted. 
 \end{remark}
\begin{proof}({\em of Proposition \ref{descent}}) By $\eta$-\'etale descent, 
 we may assume that $X=\sx_K$ for $\sx\in \sm^{\sem}_K$. 
 Recall that the action of $\sg_K$ on $\R\Gamma_{\hk}(X_{C})$,  $\R\Gamma_{\conv}(X_{C})$, and $  \R\Gamma_{\rm PD}(X_C)$ is then smooth.  We will prove only the first quasi-isomorphism - the proof of the others being analogous. 
 
 Passing to a finite  extension of the splitting field $L$ of $\sx$, if necessary, we may assume that $\sx$ is semistable  over a finite Galois extension $L$ of $K$. Consider the following commutative diagram (we added the base $K$ and $L$ in the definition of the arithmetic Hyodo-Kato cohomology to stress that we are working with the category $\sm^{\sem}_K$ and $\sm^{\sem}_L$, respectively):
 $$
 \xymatrix@R=.6cm{
 \R\Gamma_{\hk}(X/L)  \ar[r]^-{\varepsilon^*} &  \R\Gamma_{\hk}(X\otimes_LC)^{\sg_L}\\
 \R\Gamma_{\hk}(X/K)\ar[u]^{\wr}  \ar[r]^-{\varepsilon^*}  &  \R\Gamma_{\hk}(X\otimes_KC)^{\sg_K}\ar[u]
 }
 $$
 By Proposition \ref{hypercov} and  Proposition \ref{etale-descent}, the top horizontal map is quasi-isomorphic to the map
 \begin{equation}
 \label{kolo20}
  \varepsilon^*: \R\Gamma_{\hk}(\sx_0)  \to (\R\Gamma_{\hk}(\sx_0) \otimes_{F_L}F^{\nr})^{\sg_L},
\end{equation}
which clearly is a quasi-isomorphism. Since $X\otimes_KC\simeq (X\otimes_LC)\times H$ for $H=\Gal(L/K)$, we have
$$ \R\Gamma_{\hk}(X\otimes_KC)  \simeq \R\Gamma_{\hk}(X\otimes_LC)\times H.$$
Hence the right vertical map in the above diagram is a quasi-isomorphism as well. It follows that so is the bottom horizontal map, as wanted. 
\end{proof}
\subsection{Passage to Bloch-Kato arithmetic rigid analytic syntomic cohomology} \label{passage1}
 Let $X\in {\rm Sm}_K$. Let $r\geq 0$. In this section, we  define   the Bloch-Kato rigid analytic syntomic cohomology: 
 \begin{align*}
 \rg^{\rm BK}_{\synt}(X, \Q_p(r)):= [[\rg_{\hk}(X)]^{N=0,\phi=p^r}\lomapr{\iota^{\prime}_{\hk}}\rg_{\dr}(X)/F^r],
\end{align*}
where the map $\iota^{\prime}_{\hk}$ is defined below, 
and  we show that it is strictly quasi-isomorphic to the rigid analytic syntomic cohomology of $X$:
\begin{proposition}\label{passage-BK}There is a natural strict quasi-isomorphism
 \begin{align*}
\iota_2: \quad \rg^{\rm BK}_{\synt}(X, \Q_p(r))\simeq\rg_{\synt}(X, \Q_p(r)).
\end{align*}
\end{proposition}
\begin{proof}$\quad$ 

 (i) {\em Local definition. } Let $\sx$ be an admissible semistable formal scheme over $\so_K$.
We define  a functorial strict quasi-isomorphism
\begin{align}
\label{herbata}
 \iota_2: & \quad  \rg^{\rm BK}_{\synt}(\sx,\Q_p(r)):= [[\rg_{\crr}(\sx_0/\so_F^0)_{F}]^{N=0,\phi=p^r}\lomapr{\iota^{\prime}_{\hk}} \rg_{\dr}({\sx}_K)/F^r]  \\& \quad \quad \stackrel{\sim}{\to} [[\rg_{\crr}(\sx_1/\so_F)_F]^{\phi=p^r}
  \lomapr{\can}  \rg_{\crr}(\sx_1/\so_K^{\times})_{K}/F^r] \simeq 
 \rg_{\synt}(\sx,\Z_p(r))_{\Q_p}\notag
\end{align}
by the following diagram
\begin{equation}\label{cr=HK}
\xymatrix@C=-.1cm{ &  &  \rg_{\crr}(\sx_1/\so_K^{\times})_{K} & \rg_{\dr}({\sx}_K)\ar[l]^-{\sim}\ar[dl]^-{\sim}\ar[ddl]^-{\sim}\\
 [\rg_{\crr}(\sx_1/\so_F)_F]^{\phi=p^r}\ar[rru]^{\can} \ar[d]^{i^*}_{\wr}&  [\rg_{\conv}(\sx_1/\so_F)]^{\phi=p^r}\ar[l]^{\varepsilon_1}_{\sim}\ar[d]^{i^*}_{\wr} \ar[r]& \rg_{\conv}(\sx_1/\so_K^{\times})\ar[d]^{i^*}_{\wr}\ar[u]^{\wr}\\
 [\rg_{\crr}(\sx_0/\so_F)_F]^{\phi=p^r}\ar[dr] \ar[d]^{\wr}&  [\rg_{\conv}(\sx_0/\so_F)]^{\phi=p^r}\ar[rd]\ar[r]\ar[l]^{\varepsilon_0}_{\sim} & \rg_{\conv}(\sx_0/\so_K^{\times})\\
[\rg_{\crr}(\sx_0/r^{\rm PD}_F)_{\Q_p}]^{N=0,\phi=p^r}\ar[r]^-{p_p}\ar[drr]^{p_0}_{\sim}&  \rg_{\crr}(\sx_0/\so_F^{\times})_F & \rg_{\conv}(\sx_0/\so_F^{\times})\ar[l]^{\sim}\ar[u] \\
 & & [\rg_{\crr}(\sx_0/\so_F^{0})_F]^{N=0,\phi=p^r}\ar@/_30pt/[ruuuu]^{\iota^{\prime}_{\hk}}
}
\end{equation}
The vertical left bottom map is a quasi-isomorphism by \cite[Lemma 4.2]{K1}.
The map $\iota_{\hk}^{\prime}$ is defined by the zigzag in the diagram. The map $p_0$  is a quasi-isomorphism because Frobenius is highly nilpotent on $T$. The slanted map  from the convergent to crystalline cohomology is a  strict quasi-isomorphism because the log-scheme $\sx_1$ is log-smooth over $\so^{\times}_{K,1}$. The two right maps $i^*$ are strict quasi-isomorphisms (actually, natural isomorphisms) by the invariance of convergent cohomology under infinitesimal thickenings; the left map $i^*$ is a quasi-isomorphism by a standard Frobenius argument (see \cite[proof of Lemma 5.9]{CN}). We claim that the maps $\varepsilon_1, \varepsilon_0$ are strict quasi-isomorphisms. Indeed, it suffices to check this for the second of the two maps and then it follows from the commutative diagram
$$
\xymatrix@R=.6cm@C=.8cm{ [\rg_{\crr}(\sx_0/\so_F)_F]^{\phi=p^r}\ar[d]^{\wr}&  [\rg_{\conv}(\sx_0/\so_F)]^{\phi=p^r}\ar[l]^-{\varepsilon_0}\ar[d]^{\wr}\\
 [\rg_{\crr}(\sx_0/r^{\rm PD}_F)_{\Q_p}]^{N=0,\phi=p^r} \ar[d]_{p_0}^{\wr}&  [\rg_{\conv}(\sx_0/\hat{r}_F)]^{N=0,\phi=p^r}\ar[l]^-{\varepsilon}\ar[d]_{p_0}^{\wr}\\
  [\rg_{\crr}(\sx_0/\so_F^0)_F]^{N=0,\phi=p^r}&  [\rg_{\conv}(\sx_0/\so_F^0)]^{N=0,\phi=p^r}\ar[l]^-{\varepsilon^0}_{\sim}
}
$$
since the map $\varepsilon^0$ is a strict quasi-isomorphism by the log-smoothness of the log-scheme $\sx_0$ over $k^0$. Here $\hat{r}_F:=\so_F\{T\}$ and  the right vertical maps are strict quasi-isomorphisms by the same arguments as the left vertical maps.

\smallskip
 (ii) {\em Globalization.} Let $\sa_{\synt}^{\rm BK}$ be the $\eta$-\'etale sheafification of the presheaf $\sx\to \rg^{\rm BK}_{\synt}(\sx,\Q_p(r))$ on $\sm^{\sem}_{K,\eet}$. We have 
 \begin{align*}
 \rg_{\eet}(X,\sa^{\rm BK}_{\synt}) & \simeq [\rg_{\eet}(X,\sa_{\hk})^{N=0,\phi=p^r}\lomapr{\iota^{\prime}_{\hk}} \rg_{\eet}(X,\sa_{\dr})/F^r]\\
   & \simeq [\rg_{\hk}(X)^{N=0,\phi=p^r}\lomapr{\iota^{\prime}_{\hk}} \rg_{\dr}(X)/F^r]\simeq \rg_{\synt}^{\rm BK}(X,\Q_p(r)).
 \end{align*}
Since $\rg_{\synt}(X, \Q_p(r))=\rg_{\eet}(X,\sa_{\synt})$, by $\eta$-\'etale descent,  the strict quasi-isomorphisms $\iota_2$ from (\ref{herbata}) can be lifted to a  strict quasi-isomorphism
\begin{align*}
\iota_2: \quad  \rg_{\synt}(X, \Q_p(r))\simeq \rg^{\rm BK}_{\synt}(X, \Q_p(r)),
\end{align*}
as wanted.
\end{proof}
\begin{remark}
 Let us state the following corollary of the above computations.
  \begin{corollary}{\rm ({\em Local-global compatibility})} \label{compt1}
Let $r\geq 0$. For a semistable formal scheme $\sx$ over $\so_K$, the canonical map
$$
\rg_{\synt}(\sx,\Q_p(r))\to \rg_{\synt}(\sx_K,\Q_p(r))
$$
is a strict quasi-isomorphism. 
\end{corollary}
\end{remark}
\begin{proof}
By construction and Proposition \ref{passage-BK}, we have compatible  strict quasi-isomorphisms
\begin{align*}
\iota_2: \quad & \rg_{\synt}(\sx, \Q_p(r))\simeq [[\rg_{\hk}(\sx_0)]^{N=0,\phi=p^r}\lomapr{\iota^{\prime}_{\hk}}\rg_{\dr}(\sx_K)/F^r],\\
\iota_2: \quad & \rg_{\synt}(\sx_K, \Q_p(r))\simeq [[\rg_{\hk}(\sx_K)]^{N=0,\phi=p^r}\lomapr{\iota^{\prime}_{\hk}}\rg_{\dr}(\sx_K)/F^r].
\end{align*}
It suffice now to note that, by Proposition \ref{hypercov}, the natural map $\rg_{\hk}(\sx_0)\to \rg_{\hk}(\sx_K)$ is a strict quasi-isomorphism. 
\end{proof}

\section{Overconvergent syntomic cohomology} 
In this section we define syntomic cohomology for smooth dagger varieties over $K$ or $C$ in two ways (yielding strictly quasi-isomorphic theories).
 Recall that in \cite{CDN3}  syntomic cohomology of  semistable weak formal schemes is defined as a homotopy fiber of a map from Frobenius eigenspaces of Hyodo-Kato cohomology to a filtered quotients of de Rham cohomology.
  By $\eta$-\'etale descent this yields the first definition of syntomic cohomology for smooth dagger varieties.
 For the second definition we take, for smooth dagger affinoids,   the  homotopy colimits of syntomic cohomologies of the rigid analytic affinoids forming a presentation of the dagger structure, and then we globalize.
 The second  definition will allow us to define period maps to pro-\'etale cohomology. 
              
       To carry out the above, we introduce Hyodo-Kato cohomology for smooth dagger varieties, prove that it satisfies Galois descent, and define the Hyodo-Kato morphism (that is a strict quasi-isomorphism over~$C$).      
\subsection{Overconvergent de Rham cohomology}Let $L=K,C$. 
       Consider the presheaf $X\mapsto \R\Gamma_{\dr}(X)$ of filtered dg $L$-algebras on ${\rm Sm}^{\dagger}_L$. Let $\sa_{\dr}$ be its \'etale sheafification. It is a sheaf of filtered $L$-algebras on ${\rm Sm}^{\dagger}_{L,\eet}$. For $X\in {\rm Sm}^{\dagger}_L$, we have the filtered quasi-isomorphism: $\R\Gamma_{\dr}(X)\stackrel{\sim}{\to}\R\Gamma_{\eet}(X,\sa_{\dr})$.  We equip $\R\Gamma_{\dr}(X)$ with the topology induced by the canonical topology on dagger algebras; we equip $\R\Gamma_{\eet}(X,\sa_{\dr})$ with topology using \'etale descent as we did before.
  Then the above quasi-isomorphism is strict:  dagger differentials satisfy \'etale descent in the strict sense.   The de Rham cohomology $H^i_{\dr}(X)$ is classical: it is a finite dimensional $K$-vector space with its natural Hausdorff topology for $X$ quasi-compact and a Fr\'echet space (a surjective limit of finite dimensional $K$-vector spaces) for a general smooth $X$ (use Remark \ref{ducros}). See the proof of Proposition \ref{rain1} below for how this can be shown.

\subsubsection{Complex $ \rg_{\dr}(X)/F^r$. }      
       Let $X\in {\rm Sm}^{\dagger}_L$. The cohomology groups of 
 $
 \rg_{\dr}(X)/F^r
 $
 have the same description as their rigid analytic counterparts in Section \ref{derham1}. That is, 
 the distinguished triangle (in $\sd(C_L)$)
 \begin{equation}
 \label{triangle-kwak}
 0\to \ker d_r[-r]\to \tau_{\leq r}\Omega\kr_X\to \Omega^{\leq r-1}_X\to 0
 \end{equation}
 yields 
 the strict short exact sequence
 $$
 0\to H^{r-1}_{\dr}(X)\to \wt{H}^{r-1}(\rg_{\dr}(X)/F^r)\to \ker \pi \to 0,
 $$
 where $\pi$ is the natural map $\Omega^r(X)^{d=0}\to H^r_{\dr}(X).$
We have a strict monomorphism $\im d_{r-1}(X)\hookrightarrow \ker \pi$. We note that the cohomology $ \wt{H}^{r-1}(\rg_{\dr}(X)/F^r)$ is classical (as an extension of classical objects). 

 The  distinguished triangle  (\ref{triangle-kwak}) yields also the strict long exact sequence
 $$
0\to \coker \pi \to \wt{H}^r(\rg_{\dr}(X)/F^r)\to \wt{H}^1(X,\ker d_r)\to \wt{H}^{r+1}(X,\tau_{\leq r}\Omega\kr_X).
 $$

      \subsection{Arithmetic overconvergent Hyodo-Kato cohomology} We define  the Hyodo-Kato  cohomology of smooth dagger  varieties over $K$ by $\eta$-\'etale descent of overconvergent  Hyodo-Kato cohomology of semistable models.

   \subsubsection{Local definition.}   Let $X$ be a log-smooth scheme over $k^0$. The overconvergent Hyodo-Kato cohomology of $X$ is defined (by Grosse-Kl\"onne in \cite{GKFr}) as $\rg_{\hk}(X):=\rg_{\hk}(X/\so_F):=\rg_{\rig}(X/\so_F^0)$. It is a dg $F$-algebra, equipped with a $\phi$-action and a monodromy operator $N$ such that $N\phi=p\phi N$. We equip it with a topology as in \cite[Sec. 3.1]{CDN3}.
    
  Let $X$ be a semistable scheme over $k^0$.  Recall that we have the Hyodo-Kato morphism
    \begin{equation}
    \label{referee11}
    \iota_{\hk}: \rg_{\rig}(X/\so^0_F)\to \rg_{\rig}(X/\so_F^{\times})
    \end{equation}
    that is actually a strict quasi-isomorphism \cite[Section 3.1.3]{CDN3}.  We have chosen here the functorial version of this morphism as defined by Ertl-Yamada \cite[Prop. 2.5]{EY}: a  combinatorial modification of the original morphism of Grosse-Kl\"onne  yields easy functoriality on most of the data; full functoriality is obtained by a coherent zigzag construction \cite[Lemma 2.6]{EY}.
    \begin{remark} \label{referee12}
    For the convenience of the reader we will describe in more detail the constructions of Grosse-Kl\"onne (see for details \cite[Section 3.1.3]{CDN3}) and Ertl-Yamada.  
   Let $\{X_i\}_{i\in I}$ be the irreducible components of $X$ with the induced log-structure. Denote by $M_{\cdot}$ the nerve of the covering $\coprod_{i\in I}X_i\to X$. By \cite[Lemma 3.8]{CDN3}, the natural map
   $$
   \rg_{\rig}(X/\so)\to \rg_{\rig}(M_{\cdot}/\so),\quad \so=\so_F^0,\so_F^{\times},$$
    is a strict quasi-isomorphism. 
    
    Let $\overline{X}$ be the log-scheme with boundary attached to $X$ in \cite{GKFr}.  It comes equipped with a natural map $M_{\cdot}^{\prime}\hookrightarrow \overline{X}$, where $M_{\cdot}^{\prime}$ is a slight combinatorial modification\footnote{We take the definition of Ertl-Yamada, which allows multiplicities in the index set, rather than the original definition of Grosse-Kl\"onne, which does not allow them.} of $M_{\cdot}$:   there is a natural map $M_{\cdot}\to M_{\cdot}^{\prime}$ that induces a strict quasi-isomorphism
    $$
    \rg_{\rig}(M^{\prime}_{\cdot}/\so)\to \rg_{\rig}(M_{\cdot}/\so).
    $$
    We have the following commutative diagram, where $\so(0)=\so_F^0, \so(p)=\so_F^{\times}$, $a=0,p$, and $p_a$ is the map induced by $T\mapsto a$:
    $$
    \xymatrix@R=.6cm{
    \rg_{\rig}(X/\so(a))\ar[r]_{\sim}\ar[dr]^-{\sim} & \rg_{\rig}(M^{\prime}_{\cdot}/\so(a))\ar[d]^{\wr}    &\rg_{\rig}(\overline{X}/r^{\dagger}_F)\ar[l]^-{\sim}\ar[dddr]\\
    &  \rg_{\rig}(M_{\cdot}/\so(a))\\
     &  \rg_{\rig}(M_{\cdot}/r^{\dagger}_F)\ar[u]^{p_a}\\
         \rg_{\rig}(X/r^{\dagger}_F)\ar[rrr]^-{\sim} \ar[ur]^{\sim}\ar[uuu]^{p_a} && & \rg_{\rig}(M^{\prime}_{\cdot}/r^{\dagger}_F)\ar[ull]^{\sim}\ar[uuull]_{p_a}
}
    $$
    We wrote here $r^{\dagger}_F:=\so_F[T]^{\dagger}$ with the log-structure associated to $T$; Frobenius is defined by $T\mapsto T^p$, monodromy is the $\so_F$-linear derivation given by $T\mapsto T$.   
    The Hyodo-Kato morphism (\ref{referee11}) is now defined as the following composition
    $$
 \iota_{\hk}:\quad       \rg_{\rig}(X/\so_F^0)\stackrel{\sim}{\to}\rg_{\rig}(M^{\prime}_{\cdot}/\so_F^0)\stackrel{\sim}{\leftarrow}\rg_{\rig}(\overline{X}/r^{\dagger}_F)\stackrel{\sim}{\to} \rg_{\rig}(M^{\prime}_{\cdot}/\so_F^{\times}) \stackrel{\sim}{\leftarrow}\rg_{\rig}(X/\so_F^{\times}).
$$
For another semistable scheme $Y$ over $k^0$ and a map of log-schemes $g:Y\to X$, Ertl-Yamada define in \cite[Lemma 2.6]{EY} a pullback 
    morphism $g^*: \rg_{\rig}(\overline{X}/r^{\dagger}_F)\to \rg_{\rig}(\overline{Y}/r^{\dagger}_F)$ that makes $\iota_{\hk}$ functorial.
    
    In what follows, to simplify the notation, we will write
    \begin{align*}
    & p_a:\quad \rg_{\rig}(\overline{X}/r^{\dagger}_F)\stackrel{\sim}{\to} \rg_{\rig}(M^{\prime}_{\cdot}/\so(a)) \stackrel{\sim}{\leftarrow}\rg_{\rig}(X/\so(a)),\\
   &  f_1: \quad \rg_{\rig}(\overline{X}/r^{\dagger}_F) \to \rg_{\rig}(M^{\prime}_{\cdot}/r^{\dagger}_F) \stackrel{\sim}{\leftarrow}\rg_{\rig}(X/r^{\dagger}_F).
    \end{align*}
    The above commutative diagram yields the functorial commutative diagram
    $$
    \xymatrix{
       \rg_{\rig}(X/\so(a)) & \rg_{\rig}(\overline{X}/r^{\dagger}_F)\ar[l]^-{p_a}_-{\sim}\ar[dl]^{f_1}\\
         \rg_{\rig}(X/r^{\dagger}_F)\ar[u]^-{p_a}
    }
    $$
    \end{remark}
    If $\sx$ is a semistable weak formal scheme over $\so_K$,
  we define  the Hyodo-Kato map $$\iota_{\hk}: \R\Gamma_{\hk}(\sx_0)\to\R\Gamma_{\dr}(\sx_K)$$
as the following composition
  \begin{equation}
  \label{kolo21}\R\Gamma_{\hk}(\sx_0)=\R\Gamma_{\rig}(\sx_0/\so_F^0)\lomapr{\iota_{\hk}} \rg_{\rig}(\sx_0/\so_F^{\times})\to  \rg_{\rig}(\sx_0/\so_K^{\times})
  \simeq \R\Gamma_{\dr}(X_K).
\end{equation}
Note that this definition works also for base changes (with respect to $\so_K$) of semistable weak formal schemes over $\so_K$.
Since the natural morphism $ \rg_{\rig}(\sx_0/\so_F^{\times})\otimes_F K\to  \rg_{\rig}(\sx_0/\so_K^{\times})
    $ is a strict quasi-isomorphism so is the induced 
morphism
    $$\iota_{\hk}: \rg_{\hk}(\sx_0)\otimes_FK\stackrel{\sim}{\to} \rg_{\dr}(\sx_K).
    $$
 \subsubsection{Globalization.}    
     Let $\sa_{\hk}$ be the $\eta$-\'etale sheafification of the presheaf $\sx\mapsto \R\Gamma_{\hk}(\sx_0/\so_{F_L})$, $L=K_{\sx}$,  on $\sm_{K}^{\dagger,ss}$; this is an \'etale sheaf of dg $F$-algebras on ${\rm Sm}^{\dagger}_{K}$ equipped with a $\phi$-action and a derivation $N$ such that $N\phi=p\phi N$. For $X\in {\rm Sm}^{\dagger}_{K}$, set $\R\Gamma_{\hk}(X):=\R\Gamma_{\eet}(X,\sa_{\hk})$.  Equip it with a topology in the usual way, via $\eta$-\'etale descent, from the topology on the $\rg_{\hk}(\sx_0/\so_{F_L})$'s.
          \begin{proposition}{\rm ({\em Local-global compatibility})}
          \label{hypercov-dagger}Let $\sx$ be a semistable weak formal scheme over $\so_K$. Then the natural map
     $$
     \rg_{\hk}(\sx_0)\to\rg_{\hk}(\sx_K)
     $$
     is a strict quasi-isomorphism.
     \end{proposition}
     \begin{proof}
    Same as the proof of Proposition \ref{hypercov}.
     \end{proof}
     
      For $X\in {\rm Sm}^{\dagger}_K$, we define natural $F$-linear maps ({\em the overconvergent Hyodo-Kato morphisms})
$$\iota_{\hk}:\sa_{\hk}\to\sa_{\dr},\quad \iota_{\hk}:  \R\Gamma_{\hk}(X){\to}\R\Gamma_{\dr}(X)
$$
by the $\eta$-\'etale sheafification of the Hyodo-Kato map $\iota_{\hk}: \R\Gamma_{\hk}(\sx_0){\to}\R\Gamma_{\dr}(\sx_K)
$ and its globalization, respectively.

  \subsubsection{Topology.}    
     We will now discuss  topology in more detail. 
     \begin{proposition} 
\label{rain1}
Let $X$ be a smooth dagger variety over $K$. 
\begin{enumerate}
\item If $X$ is quasi-compact then $\wt{H}^*_{\hk}(X)$ is classical. It  is a finite dimensional $F$-vector space with its unique locally convex Hausdorff topology.
\item For a general $X$, the cohomology $\wt{H}^*_{\hk}(X)$ is classical. It is a  Fr\'echet space, a limit of finite dimensional $F$-vector spaces.
\item The endomorphism $\phi$ on $H^*_{\hk}(X)$ is  a  homeomorphism.
\item If $k$ is finite and $X$ is quasi-compact then $H^*_{\hk}(X)$ is a mixed $F$-isocrystal, i.e., the eigenvalues\footnote{We define the eigenvalues of $\phi$ in $\Q\otimes F^*$ to be the $s$'th roots of the eigenvalues of $\phi^s$, where $s$ is  any non-zero multiple of $f$ for $|k|=p^f$. We note that this definition is stable under base change from $F$ to $F^{\prime}$, $[F^{\prime}:F]<\infty$.}  of $\phi$  are Weil numbers (if $X$ is not quasi-compact then $H^*_{\hk}(X)$ is a product of mixed $F$-isocrystals). 
\end{enumerate}
\end{proposition}
\begin{proof}In the case $X=\sx_K$, for a semistable weak formal model  $\sx$ over $\so_K$,  and for $\wt{H}^*_{\hk}(\sx_0)$ this is \cite[Prop. 3.2]{CDN3}. 
All algebraic statements concerning   cohomology in the proposition follow from that by using $\eta$-\'etale descent and the quasi-isomorphism from Proposition \ref{hypercov-dagger}. 

 We treat now the topological claims. For  (1),  we first use the  $\eta$-\'etale descent and the fact that claim (1) holds in the case $X$ has a semistable model over $\so_K$ to  construct a filtration on the classical cohomology $H^i_{\hk}(X)$ with graded pieces finite rank vector spaces over $F$ with their canonical Hausdorff topology.  This implies that the natural topology on $H^i_{\hk}(X)$ is also Hausdorff. It remains to show that $\wt{H}^i_{\hk}(X)$ is classical. Take an $\eta$-\'etale hypercovering $\su_{\cdot}$ of $X$ built from objects of $\sm^{\dagger, \sem}_K$. Assume that in every degree we have a finite number of affine weak formal schemes (recall that $X$ is quasi-compact). Then the complex $\rg_{\hk}(\su_{\cdot,0})$ is built from inductive limits of Banach spaces with injective and compact transition maps. Using the fact that these are strong duals of reflexive Fr\'echet spaces we know that the kernels of the differentials and their coimages have the same property. In particular, they are $LB$-spaces. The cohomology $\wt{H}^i_{\hk}(X)$ is represented by the pair $\coim d_{i-1}\to \ker d_{i}$ and ${H}^i_{\hk}(X)=\ker d_i/\im d_{i-1}$ with the induced topology. Let $W$ be a subspace of $\ker d_i$ that maps onto
 ${H}^i_{\hk}(X)$ and has the same rank as the latter. Then the map $\coim d_{i-1}\oplus W\to \ker d_i$ is a continuous map of  $LB$-spaces that is an algebraic isomorphism hence, by the Open Mapping Theorem, it is a topological isomorphism. Hence the map $\coim d_{i-1}\to \ker d_{i}$ is strict and the cohomology $\wt{H}^i_{\hk}(X)$ is classical. 
 
 A similar argument, using strong duals of reflexive Fr\'echet spaces, implies that  a  map between two Hyodo-Kato  complexes associated to two (different) $\eta$-\'etale  affine hypercoverings  of $X$ as above is a strict quasi-isomorphism. This implies that, for $X$ quasi-compact, the cohomology of $\R\Gamma_{\hk}(X)$ is strictly quasi-isomorphic to the cohomology of $\R\Gamma_{\hk}(\su_{\cdot,0})$ for any $\eta$-\'etale affine hypercovering $\su_{\cdot}$ as above. 

 To see that $\phi$ is a homeomorphism in (3), note that this is clear for quasi-compact $X$ by the above remarks. For a general $X$,  as in the case of pro-\'etale cohomology,  cover it with an admissible increasing quasi-compact covering $\{U_n\}_{n\in\N}$.  We obtain the exact sequence
 $$
 0\to H^1\holim_n \wt{H}^{i-1}_{\hk}(U_n)\to \wt{H}^i_{\hk}(X)\to H^0\holim_n \wt{H}^{i}_{\hk}(U_n)\to 0
 $$
 But, by (1), the cohomologies $\wt{H}^{i}_{\hk}(U_n)$ are classical and finite dimensional over $F$. Hence, the cohomology $\wt{H}^i_{\hk}(X)$ is classical and we have
 $${H}^i_{\hk}(X)\stackrel{\sim}{\to} \invlim_n {H}^{i}_{\hk}(U_n).
 $$
 Hence it is Fr\'echet, as wanted. We have proved (2), and (4)  follows now trivially from (1). 
 \end{proof}
      
     \subsubsection{$(\phi,N)$-cohomology.}Let $X\in{\rm Sm}^{\dagger}_K$, $r\geq 0$. 
 We will need to understand the cohomology of 
 $[\rg_{\hk}(X)]^{N=0,\phi=p^r}$. We have
 $$
 [\rg_{\hk}(X)]^{N=0,\phi=p^r}=\left[\begin{aligned}\xymatrix@R=.6cm{\rg_{\hk}(X)\ar[d]^{N}\ar[r]^{p^r-\phi} & \rg_{\hk}(X)\ar[d]^{N}\\
 \rg_{\hk}(X)\ar[r]^{p^r-p\phi} & \rg_{\hk}(X)
 }\end{aligned}\right]
 $$
  This gives rise to a  spectral sequence
\begin{equation}\label{sseq22}
E^{ij}_2=\wt{H}^i([ H^j_{\hk}(X)]^{N=0,\phi=p^r})\Rightarrow \wt{H}^{i+j}(\rg_{\hk}(X)^{N=0,\phi=p^r}),
\end{equation}
where $\wt{H}^*([H^j_{\hk}(X)]^{N=0,\phi=p^r})$ is the cohomology of the complex 
$$
\left[\begin{aligned}\xymatrix@R=.6cm{H^j_{\hk}(X)\ar[r]^{p^r-\phi} \ar[d]^{N} & H^j_{\hk}(X)\ar[d]^{N}\\
H^j_{\hk}(X)\ar[r]^{p^r-p\phi} & H^j_{\hk}(X)
}\end{aligned}\right]
$$
That is, we can compute it by the sequence
$$
H^j_{\hk}(X)\veryverylomapr{(N,p^r-\phi)}  H^j_{\hk}(X)\oplus H^j_{\hk}(X)\veryverylomapr{(p^r-p\phi)-N}H^j_{\hk}(X).
$$
The  cohomology $\wt{H}^i([H^j_{\hk}(X)]^{N=0,\phi=p^r})$ is classical and a  Fr\'echet space. This is because we can write naturally $H^i_{\hk}(X)\simeq \invlim_nH^i_{\hk}(U_n)$, for an admissible increasing quasi-compact covering 
$\{U_n\}_{n\in\N}$ of $X$, and all the cohomologies $H^i_{\hk}(U_n)$ are finite dimensional over $F$. 

 Hence, in the spectral sequence (\ref{sseq22}),  the terms are classical and Fr\'echet spaces. Arguing by limits as above, we conclude that so is the abutment. 

\begin{remark}In the case when $H^j_{\hk}(X)$ is a finite $(\phi,N)$-module (for example when $X$ is quasi-compact), then ${H}^*([H^j_{\hk}(X)]^{N=0,\phi=p^r})\simeq \Ext^*_{\phi,N}(F,H^j_{\hk}(X)\{r\})$, the $\Ext$-groups in the category of finite $(\phi,N)$-modules  \cite{BE2}. 
\end{remark}
\begin{proposition}
\label{computHK}
Let $X\in {\rm Sm}^{\dagger}_K, r\geq 0$. 
\begin{enumerate}
\item We have ${H}^{i}([\rg_{\hk}(X)]^{N=0,\phi=p^r})=0$ for $i\leq r-1$.
\item 
    There is a  strict short exact sequence
\begin{equation}
\label{jesien}
0\to H^{r-1}_{\hk}(X)^{\phi=p^{r-1}}\to H^{r}([\rg_{\hk}(X)]^{N=0,\phi=p^r})\to H^{r}_{\hk}(X)^{N=0,\phi=p^r}\to 0
\end{equation}
\end{enumerate}
\end{proposition}
\begin{proof}
To see that, we note that
 the slopes of Frobenius on $H^{i}_{\hk}(X)$  are $\leq i$:  it is enough to show this for $X$ with a semistable reduction where we can use  the weight spectral sequence  to reduce to showing that, for a smooth  scheme $Y$ over $k$, the slopes of Frobenius
on the (classical) rigid cohomology $H^i_{\rig}(Y/F)$ are $\leq i$; but this is well-known \cite[Th. 3.1.2]{CS}. It follows that 
 the morphism  $\phi-p^j$ is an isomorphism on $H^{i}_{\hk}(X)$ for $i<j$.  Knowing that, we obtain  both claims of the proposition  from  the spectral sequence (\ref{sseq22}).
\end{proof}

  \subsection{Geometric overconvergent Hyodo-Kato cohomology}\label{lebras}We define the Hyodo-Kato  cohomology of smooth dagger  varieties over $C$ by $\eta$-\'etale descent of overconvergent  Hyodo-Kato cohomology of semistable models. 
\subsubsection{Definition and basic properties.} \label{limit11}
  Let  ${f}:\sx\to \Spwf(\so_C)^{\times}$ be a  semistable  weak formal model. Suppose  that ${f}$ is the base change of 
 a semistable weak  formal model ${f}_L:\sx_{\so_L}\to \Spwf(\so_{L})^{\times}$  over $\so_L$ by $\theta: \Spwf(\so_C)^{\times}\to\Spwf(\so_{L})^{\times}$, for a finite extension $L/K$. That is, we have a map $\theta_{L}: \sx\to \sx_{\so_L}$ such that the square $({f},{f}_L,\theta,\theta_{L})$ is Cartesian. Such data $\{(L,\sx,\theta_{L})\}$ reduced modulo $p$ form a filtered set $\Sigma$ (cf. Section \ref{part-i}). \\
 (i) {\em Hyodo-Kato cohomology.} For a morphism of tuples $(L\pri,\sx_{\so_{L\pri,1}}\pri,\theta\pri_{L^\prime})\to (L,\sx_{\so_{L,1}},\theta_{L})$ from $\Sigma$,
 we   have a canonical base change identification compatible with $\phi$-action (unramified base change)
\begin{equation}
\label{limit1}
\R\Gamma_{\hk}(\sx_{\so_L,0})\otimes_{F_L}F_{L\pri}\stackrel{\sim}{\to}
 \R\Gamma_{\hk}(\sx\pri_{\so_{L^\prime},0}).
\end{equation}
We set
$$ \R\Gamma_{\hk}(\sx_1):= \hocolim_{\Sigma}\R\Gamma_{\hk}(\sx_{\so_L,0}).
$$
It is a dg $ F^{\nr}$-algebra\footnote{The field $F^{\nr}$ is equipped here with the inductive limit topology in $C_F$. In particular, a sequence $(x_n)_{n\in\N}$, of elements of $F^{\nr}$ converges if and only if there exists a finite extension $L$ of $F$ such that all $x_n\in L$ and the sequence  $(x_n)_{n\in\N}$ converges inside $L$.}  equipped with a $\phi$-action and a derivation $N$ such that $N\phi=p\phi N$.
It is functorial with respect to  $\sx$: note that the restriction of a morphism  $\sx\to \sy$ to a morphism $\sx_1\to\sy_1$ is defined over a finite extension of $K$. 

 Let $\sa_{\hk}$ be the $\eta$-\'etale sheafification of the presheaf $\sx\mapsto \rg_{\hk}(\sx_1)$ on $\sm_C^{\dagger, \sem,b}$. For $X\in {\rm Sm}^{\dagger}_C$, we set $\rg_{\hk}(X):=\rg_{\eet}(X,\sa_{\hk})$. It is a dg $F^{\nr}$-algebra equipped with a Frobenius, monodromy action, and a continuous action of
 $\sg_K$ if $X$ is defined over $K$ (this action is smooth if $X$ is quasi-compact). We equip it with the topology induced, by $\eta$-\'etale descent,  from the topology on the $\R\Gamma_{\hk}(\sx_{\so_L,0})$'s.
 \begin{proposition}
\label{rain15}
Let $X$ be a smooth dagger variety over $C$. 
\begin{enumerate}
\item If $X$ is quasi-compact then $\wt{H}^*_{\hk}(X)$ is classical. It  is a finite dimensional $F^{nr}$-vector space with its natural topology.
\item The cohomology $\wt{H}^*_{\hk}(X)$ is classical. It is a limit (in $C_F$) of finite dimensional $F^{\nr}$-vector spaces. 
\item The endomorphism $\phi$ on $H^*_{\hk}(X)$ is  a  homeomorphism.
\item  If $k$ is finite and $X$ is quasi-compact then $H^*_{\hk}(X)$ is a mixed $F$-isocrystal, i.e., the eigenvalues\footnote{The cohomology $H^*_{\hk}(X)$ together with its Frobenius, a priori an $F^{\nr}$-vector space of finite rank, is obtained by a base change from a finite rank $F^{\prime}$-vector space $V$, where $[F^{\prime}:F]<\infty$, equipped with a semilinear Frobenius so we can use the definition of eigenvalues of Frobenius  from the footnote to Proposition \ref{rain1}.}
 of $\phi$  are Weil numbers (if $X$ is not quasi-compact then $H^*_{\hk}(X)$ is a product of mixed $F$-isocrystals). 
\end{enumerate}
\end{proposition}
\begin{proof} For claim (1), it suffices to show that, for every $\eta$-\'etale hypercovering $\su_{\cdot}$ of $X$ from $\sm^{\dagger,\sem,b}_C$, the cohomology
$\wt{H}^i_{\hk}(\su_{\cdot,C})$, $i\geq 0$,  is classical and of finite rank over $F^{\nr}$. Since we can assume that the weak formal schemes in every degree of the hypercovering are admissible, this follows immediately from 
Proposition \ref{rain1} and the quasi-isomorphism (\ref{limit1}).

Claim (2) follows easily  from claim (1).  Claim (3) and (4) follow by the same argument as claim~(1).
\end{proof}
 
(i) {\em Rigid cohomology.}  Let $\sa_{\rig}$ be the $\eta$-\'etale sheafification of the presheaf $\sx\mapsto \rg_{\rig}(\sx_1/\so_C^{\times})$ on $\sm_{C}^{\dagger,\sem,b}$.
   For $X\in {\rm Sm}^{\dagger}_C$, we set $\rg_{\rig}(X):=\rg_{\eet}(X,\sa_{\rig})$. It is a dg $C$-algebra equipped with  a continuous action of
 $\sg_K$ if $X$ is defined over $K$. We equip it with the topology induced, by $\eta$-\'etale descent,  from the topology on the $\R\Gamma_{\rig}(\sx_1/\so_C^{\times})$'s.
We have natural (strict) quasi-isomorphisms
 $$
 \sa_{\rig}\stackrel{\sim}{\to} \sa_{\dr},\quad \rg_{\rig}(X)\stackrel{\sim}{\to} \rg_{\dr}(X).
 $$

  Let $\sa_{\rig,\ovk}$ be the $\eta$-\'etale sheafification of the presheaf $\sx\mapsto \rg_{\rig,\ovk}(\sx_1)$ on $\sm_{C}^{\dagger,\sem,b}$,
  where we set
  $$
  \rg_{\rig,\ovk}(\sx_1):=\hocolim_{\Sigma}\rg_{\rig}(\sx_0/\so_{L}^{\times}).
  $$
   For $X\in {\rm Sm}^{\dagger}_C$, we set $\rg_{\rig,\ovk}(X):=\rg_{\eet}(X,\sa_{\rig,\ovk})$. It is a dg $\ovk$-algebra equipped with  a continuous action of
 $\sg_K$ if $X$ is defined over $K$ (this action is smooth if $X$ is quasi-compact). We equip it with the topology induced, by $\eta$-\'etale descent,  from the topology on the $\R\Gamma_{\rig}(\sx_{\so_L,0})$'s.
There are natural continuous morphisms
$$
  \sa_{\rig,\ovk}  \to \sa_{\rig},\quad \rg_{\rig,\ovk}(X)\to \rg_{\rig}(X).
$$

\subsubsection{Galois descent.}
    Again we have a Galois descent.
\begin{proposition}
\label{descent1}
Let $X\in {\rm Sm}^{\dagger}_K$. The natural projection $\varepsilon: X_{C,\eet}\to X_{\eet}$ defines pullback quasi-isomorphisms
\begin{equation}
\label{qis111}
 \varepsilon^*: \rg_{\hk}(X)\stackrel{\sim}{\to}\rg_{\hk}(X_C)^{\sg_K},\quad  \varepsilon^*: \rg_{\dr}(X)\stackrel{\sim}{\to}\rg_{\rig,\ovk}(X_C)^{\sg_K}.
 \end{equation}
\end{proposition}
 \begin{proof}
 We can  use the proof of Proposition \ref{descent} almost verbatim\footnote{Note that Remark \ref{Galois-sense} applies to this setting.}. 
 \end{proof}

  \subsubsection{Hyodo-Kato quasi-isomorphisms.}  $\quad$
  
   (i) {\em Local definition.} Let $\sx\to \Spwf(\so_C)^{\times}$ be as above. 
  The Hyodo-Kato morphism from (\ref{kolo21}):
  \begin{equation}
  \label{HK-jablko2}
  \iota_{\hk}: \rg_{\hk}(\sx_{\so_L,0})\to \rg_{\rig}(\sx_{\so_L,0}/\so_L^{\times}),\quad \iota_{\hk}: \rg_{\hk}(\sx_{\so_L,0})\otimes_{F_L}L\stackrel{\sim}{\to} \rg_{\rig}(\sx_{\so_L,0}/\so_L^{\times})
  \end{equation}
   is compatible with the morphisms in $\Sigma$ and taking  its homotopy colimit yields the first of the following two natural strict quasi-isomorphisms (called again the {\em Hyodo-Kato quasi-isomorphisms})
\begin{align}
\label{HK-geometric-rig}
 & \iota_{\hk}:\quad \R\Gamma_{\hk}(\sx_{1}){\otimes}_{F^{\nr}}{\ovk}\simeq\hocolim_{\Sigma}(\rg_{\hk}(\sx_{\so_L,0})\otimes_{F_L}L)
 \stackrel{\sim}{\to} \hocolim_{\Sigma}\R\Gamma_{\rig}(\sx_{\so_L,0}/\so_L^{\times})=:\R\Gamma_{\rig,\ovk}(\sx_1),\\
& \iota_{\hk}:\quad \R\Gamma_{\hk}(\sx_1)\wh{\otimes}^R_{F^{\nr}}{C}\stackrel{\sim}{\to} \rg_{\rig}(\sx_1/\so_C^{\times})\simeq \R\Gamma_{\dr}(\sx_{C}).\notag
\end{align}
In the second Hyodo-Kato morphism, we set
$$
\R\Gamma_{\hk}(\sx_1)\wh{\otimes}^R_{F^{\nr}}{C}:=\hocolim_{\Sigma}(\R\Gamma_{\hk}(\sx_{\so_L,0})\wh{\otimes}^R_{F_L}{C}),
$$
where  all the maps in the homotopy limit are strict quasi-isomorphisms. 
This morphism is then defined as the composition
\begin{align*}
& \hocolim_{\Sigma}(\R\Gamma_{\hk}(\sx_{\so_L,0})\wh{\otimes}^R_{F_L}{C})\lomapr{\iota_{\hk}}\hocolim_{\Sigma}(\R\Gamma_{\rig}(\sx_{\so_L,0}/\so_L^{\times})\wh{\otimes}^R_{L}{C})\\
&\quad\quad \stackrel{\sim}{\to} \R\Gamma_{\rig}(\sx_{1}/\so_C^{\times})
\stackrel{\sim}{\to}\R\Gamma_{\dr}(\sx_{C}),
\end{align*}
where we have used the Hyodo-Kato quasi-isomorphism from (\ref{HK-jablko2}), the second map is a strict quasi-isomorphism by base change.
So defined morphism is clearly a strict quasi-isomorphism. 

(ii) {\em Globalization.} Varying $\sx$ in the above constructions we obtain the Hyodo-Kato maps
 $$
\iota_{\hk}:  \sa_{\hk}\to \sa_{\rig},\quad \iota_{\hk}: \sa_{\hk}\to \sa_{\dr}
 $$
 of sheaves on ${\rm Sm}^{\dagger}_{C,\eet}$.
 For $X\in {\rm Sm}^{\dagger}_C$, they induce the natural  Hyodo-Kato strict quasi-isomorphisms 
  \begin{equation}
  \label{HK-rig}
   \iota_{\hk}:\R\Gamma_{\hk}(X)\wh{\otimes}_{F^{\nr}}{\ovk}\stackrel{\sim}{\to} \R\Gamma_{\rig,\ovk}(X),\quad
 \iota_{\hk}: \R\Gamma_{\hk}(X)\wh{\otimes}^R_{F^{\nr}}{C}\stackrel{\sim}{\to} \R\Gamma_{\dr}(X).
  \end{equation}
   Here we set 
 \begin{align}
 \label{HK-referee1}
  & \R\Gamma_{\hk}(X)\wh{\otimes}_{F^{\nr}}{\ovk}:=\hocolim((\R\Gamma_{\hk}{\otimes}_{F^{\nr}}{\ovk})(\su_{\cdot,0})),\\
 & \R\Gamma_{\hk}(X)\wh{\otimes}^R_{F^{\nr}}C:=\hocolim((\R\Gamma_{\hk}\wh{\otimes}^R_{F^{\nr}}{C})(\su_{\cdot,0})),\notag
 \end{align}
 where the homotopy colimit is taken over $\eta$-\'etale hypercoverings from $\sm^{\dagger,\sem, b}_C$.  We note that 
\begin{equation}
\label{pierre2}
 \R\Gamma_{\rig,\ovk}(X)\simeq \hocolim\R\Gamma_{\rig,\ovk}(\su_{\cdot,1}).
 \end{equation}
  This is because $ \hocolim\R\Gamma_{\rig,\ovk}(\su_{\cdot,1})\simeq  \hocolim\R\Gamma_{\rig,\ovk}(\su_{\cdot,C})
$ by Proposition \ref{pierre1} below (there is no circular reasoning here) and we have $\eta$-\'etale descent for $\R\Gamma_{\rig,\ovk}(X)$. 
Having (\ref{pierre2}),
the first strict quasi-isomorphism in (\ref{HK-rig}) follows from the strict Hyodo-Kato quasi-isomorphism in (\ref{HK-geometric-rig}). The latter also imply easily the second strict quasi-isomorphism we wanted.

  (iii) {\em Local-global compatibility and comparison results.} The Hyodo-Kato quasi-isomorphisms allow us now to prove the following comparison result (where the tensor products in (2) and (3) are defined as in (\ref{HK-referee1}).
  \begin{proposition} \label{pierre1}
  \begin{enumerate}
  \item 
Let  $\sx\in\sm^{\dag,\sem, b}$. 
 Then the natural maps
 $$
 \rg_{\hk}(\sx_1)\to\rg_{\hk}(\sx_C),\quad  \rg_{\rig}(\sx_1)\to\rg_{\rig}(\sx_C), \quad  \rg_{\rig,\ovk}(\sx_1)\to\rg_{\rig,\ovk}(\sx_C)
 $$ are strict quasi-isomorphisms. 
 \item For $X\in {\rm Sm}^{\dagger}_C$, we have a natural strict quasi-isomorphism
  $$
  \rg_{\rig,\ovk}(X)\wh{\otimes}^R_{\ovk}C\stackrel{\sim}{\to} \rg_{\rig}(X)\simeq \rg_{\dr}(X).
  $$
\item For $X\in {\rm Sm}^{\dagger}_K$, we have a natural strict quasi-isomorphism
  $$
  \rg_{\dr}(X)\wh{\otimes}_{K}\ovk\simeq \rg_{\rig,\ovk}(X_C).
  $$
 \end{enumerate}
 \end{proposition}
 \begin{proof} The proof is almost verbatim the same as  the proof of Proposition \ref{etale-descent} (which contains analogous claims in the case of rigid analytic varieties) we just need to replace $\rg_{\conv}$ used there with
 $\rg_{\rig}$. 
 \end{proof}
 \begin{remark}
Much of what we have described above in Section \ref{lebras} goes through, with minimal changes, for $X\in {\rm Sm}_C$. Hence, working with formal schemes instead of weak formal schemes, we have the geometric Hyodo-Kato cohomology
$\rg^{\dagger}_{\hk}(X)$. We wrote ${}^\dagger$ to distinguished this cohomology from the geometric Hyodo-Kato cohomology $\rg_{\hk}(X)$ defined in Section \ref{berkeley11}. It is a dg $F^{\nr}$-algebra equipped with a $\phi$-action, derivation $N$ such that $N\phi=p\phi N$, and a continuous action of $\sg_K$ (which is smooth when $X$ is quasi-compact). It has an arithmetic analogue that satisfies  Galois descent of the type described  in Proposition \ref{descent1}.
We also have the Hyodo-Kato quasi-isomorphism
$$
\iota_{\hk}:\rg^{\dagger}_{\hk}(X)\wh{\otimes}_{F^{\nr}}\ovk\stackrel{\sim}{\to} \rg_{\rig,\ovk}(X),
$$
where the rigid cohomology is defined like its analog  for  dagger varieties.

     If $X$ is quasi-compact, the underlying isocrystal of $H^i\rg^{\dagger}_{\hk}(X)$ should be the one defined by Le Bras in \cite{LeBras}.
  
\end{remark}

    \subsection{Arithmetic overconvergent  syntomic cohomology}
    \label{def33}We define now arithmetic overconvergent syntomic  cohomology of smooth dagger  varieties over $K$ by $\eta$-\'etale descent of overconvergent  syntomic cohomology of semistable weak formal models.

Let $\sx$ be an admissible  semistable weak formal scheme over $\so_L$, $[L:K]<\infty$. For $r\geq 0$, we define the overconvergent  syntomic cohomology as 
\begin{equation}
\label{over0}
\rg_{\synt}(\sx,\Q_p(r)):=[[\rg_{\hk}(\sx_0)]^{N=0,\phi=p^r}\lomapr{\iota_{\hk}} \rg_{\dr}(\sx_L)/F^r].
\end{equation}
For a smooth dagger space $X$ over $K$ we define the syntomic cohomology $\sa_{\synt}(r)$ as the $\eta$-\'etale sheafification of the above complexes on $\sm_{K}^{\sem, \dagger}$; and we define the syntomic cohomology of $X$ as 
$$
\rg_{\synt}(X,\Q_p(r)):=\rg_{\eet}(X,\sa_{\synt}(r)).
$$
We have the distinguished triangle
\begin{equation}
\label{over1}
\rg_{\synt}(X,\Q_p(r))\to [\rg_{\hk}(X)]^{N=0,\phi=p^r}\lomapr{\iota_{\hk}}\rg_{\dr}(X)/F^r
\end{equation}
\begin{proposition}(Local-global compatibility)
\label{berkeley-sunny}
Let $r\geq 0$. Let $\sx$ be a semistable weak formal scheme over $\so_K$. Then the natural map
$$
\rg_{\synt}(\sx,\Q_p(r))\to \rg_{\synt}(\sx_K,\Q_p(r))
$$
is a strict quasi-isomorphism.
\end{proposition}
\begin{proof}
Using the presentations of syntomic cohomology from (\ref{over0}) and (\ref{over1}) we reduce to proving that the natural map
$
\rg_{\hk}(\sx_0)\to \rg_{\hk}(\sx_K)
$
is a strict  quasi-isomorphism. But this we know to be true by Proposition \ref{hypercov-dagger}. 
\end{proof}
\subsubsection{Examples.} \label{synt-examples}We will discuss a couple of examples.

 (i) {\em The closed  ball.} Let $L=K, C$. Let $X_L:={\mathbb B}^d_L(\rho)$ be the overconvergent closed ball over $L$ of dimension $d\geq 0$ and radius $\rho\in\sqrt |L^{\times}|$.  Since 
$H^0_{\dr}(X_L)\simeq L$ and 
$H^i_{\dr}(X_L)=0$,  $i>0$, and we have the Hyodo-Kato isomorphism $H^i_{\hk}(X_C){\otimes}_{F^{\nr}}C\simeq H^i_{\dr}(X_C)$ and the Galois descent $H^i_{\hk}(X_K)\stackrel{\sim}{\to} H^i_{\hk}(X_C)^{\sg_K}$,
 we get
$$
H^i_{\hk}({\mathbb B}^d_L(\rho))\simeq \begin{cases}
F_L & \mbox{ if } i=0,\\
0 & \mbox{ if } i\geq 1,
\end{cases}
$$
where $F_C=F^{\nr}$ and $F_K=F$.

 From the exact sequence (\ref{jesien}), we get 
 \begin{align*} H^0([\rg_{\hk}(X_K)]^{N=0,\phi=1}) & \stackrel{\sim}{\to} 
  H^0_{\hk}(X_K)^{N=0,\phi=1}, \\
  H^0_{\hk}(X_K)^{N=0,\phi=1} & \stackrel{\sim}{\to} H^1([\rg_{\hk}(X_K)]^{N=0,\phi=p}).
 \end{align*}
  Hence, by the above, 
 $$
H^i([\rg_{\hk}({\mathbb B}^d_K(\rho))]^{N=0,\phi=p^i}))\simeq \begin{cases}
\Q_p& \mbox{ if } i=0,1,\\
0 & \mbox{ if } i\geq 2.
\end{cases}
$$

    Let $r\geq 0$. By the triviality, in nonzero degrees,   of the  cohomology of coherent sheaves on ${\mathbb B}^d_K(\rho)$, we have
 $$
 \rg_{\dr}(X_K)/F^r\simeq\so(X_K)\to\Omega(X_K)\to \cdots\to \Omega^{r-1}(X_K).
 $$
 Hence $H^{i}( \rg_{\dr}(X_K)/F^r)=0,$ for $i\geq r$, and $H^{r-1} (\rg_{\dr}(X_K)/F^r)\stackrel{\sim}{\leftarrow} \Omega^{r-1}(X_K)/\im d_{r-2}\simeq \Omega^r(X_K)^{d=0}$. 
 From the definition of syntomic cohomology and the above computations, 
 we get the long exact sequence
 $$
H^{r-1}([\rg_{\hk}(X_K)]^{N=0, \phi=p^r})\to  \Omega^{r-1}(X_K)/\im d_{r-2}  \to H^r_{\synt}(X_K,\Q_p(r))\to H^r([\rg_{\hk}(X_K)]^{N=0, \phi=p^r})\to 0
 $$
 Hence
 $$H^r_{\synt}({\mathbb B}^d_K(\rho),\Q_p(r))
 \simeq \begin{cases}
\Q_p& \mbox{ if } r=0,\\
 \Omega^{r-1}({\mathbb B}^d_K(\rho))/\im d_{r-2} & \mbox{ if } r\geq 2,
\end{cases}
$$
and, for $r=1$, we get an extension
$$
0\to \so({\mathbb B}^d_K(\rho))\to H^1_{\synt}(({\mathbb B}^d_K(\rho),\Q_p(1))\to \Q_p\to 0
$$

(ii) {\em The open  ball.} Let $L=K, C$. Let ${\mathbb B}^{{\rm o},d}_L(\rho)$ be the overconvergent open  ball over $L$ of dimension $d\geq 0$ and radius $\rho$. 
Cover ${\mathbb B}^{{\rm o},d}_L(\rho)$ with an increasing union of overconvergent closed balls $\{U_n\}_{n\in\N}$. By the above example, we have
$H^i_{\hk}({\mathbb B}^{{\rm o},d}_L(\rho))\simeq \invlim_n H^i_{\hk}(U_n)$. Hence
$$
H^i_{\hk}({\mathbb B}^{{\rm o},d}_L(\rho))\simeq \begin{cases}
F_L & \mbox{ if } i=0,\\
0 & \mbox{ if } i>0.
\end{cases}
$$
The rest of the computations is exactly the same as for the closed ball in the first example (note that ${\mathbb B}^{{\rm o},d}_K(\rho)$ is Stein) yielding the same final formulas for $H^r_{\synt}(({\mathbb B}^{{\rm o},d}_K(\rho),\Q_p(r))$ (with ${\mathbb B}^{{\rm o},d}_K(\rho)$ in the place of ${\mathbb B}^{d}_K(\rho)$). 

 \section{Comparison of overconvergent  and  rigid analytic arithmetic syntomic cohomology}
 We  define a map from syntomic cohomology of a smooth dagger variety to syntomic cohomology of its completion. We show that it is a strict quasi-isomorphism when the variety is partially proper. 
\subsection{Construction of the comparison morphism} 
    Let $X$ be a smooth dagger space over $K$. We will now construct a functorial map 
    \begin{equation*}
    \iota: \quad  \rg_{\synt}(X,\Q_p(r))\to \rg_{\synt}(\wh{X},\Q_p(r))
   \end{equation*}
    from the syntomic cohomology of $X$ to the syntomic cohomology of its completion $\wh{X}$.
    This will be done by first constructing a map $\iota_1$ to the Bloch-Kato syntomic cohomology from Section \ref{passage1}:
     \begin{equation*}
    \iota_1: \quad  \rg_{\synt}(X,\Q_p(r))\to \rg^{\rm BK}_{\synt}(\wh{X},\Q_p(r))
   \end{equation*}
   and then setting  $\iota:=\iota_2\iota_1$, for the map $\iota_2:  \rg^{\rm BK}_{\synt}(\wh{X},\Q_p(r))\simeq  \rg_{\synt}(\wh{X},\Q_p(r))$ that was defined in Proposition \ref{passage-BK}.

(i) {\em Local definition.}
   Let $\sx$ be a semistable weak formal scheme of finite type over $\so_K$.     
   First, we  define a functorial morphism
    \begin{align}
    \label{iota1}
  \iota_1:  \rg_{\synt}(\sx,\Q_p(r)) & =  [[\rg_{\rig}(\sx_0/\so_F^0)]^{\phi=p^r}\verylomapr{\iota_{\hk}}\rg_{\dr}(\sx_K)/F^r]\\
   & \to [[\rg_{\crr}(\sx_0/\so_F^0)_{F}]^{N=0,\phi=p^r}\lomapr{\iota^{\prime}_{\hk}} \rg_{\dr}(\wh{\sx}_K)/F^r].\notag
    \end{align}
    We use for that the following diagram (we note that that all the terms in the first two columns carry a monodromy operator and that all the maps between these terms are compatible with the monodromy action)
     \begin{equation}
     \label{sopot1}
\xymatrix@C=.6cm{
 [\rg_{\rig}(\sx_0/\so_F^0)]^{\phi=p^r}\ar@/_110pt/[dddd]\ar[dr]^{\iota_{\hk}}\ar@/^20pt/[rrrd]^{\iota_{\hk}}\\
[ \rg_{\rig}(\overline{\sx}_0/r^{\dagger}_F)]^{\phi=p^r}\ar[u]^{p_0}_{\wr}\ar[d]^{f_1}\ar[r]^-{p_p}_-{\sim}& \rg_{\rig}(\sx_0/\so_F^{\times})\ar[d]\ar[r]  &  \rg_{\rig}(\sx_0/\so_K^{\times})\ar[d] & \rg_{\dr}(\sx_K)\ar[d]\ar[l]^-{\sim}\\
 [\rg_{\rig}(\sx_0/r^{\dagger}_F)]^{\phi=p^r}\ar@/^60pt/[uu]^{p_0}\ar[ru]^{p_p}\ar[d] & \rg_{\conv}(\sx_0/\so_F^{\times})\ar[r]\ar[d]^{\wr}  &  \rg_{\conv}(\sx_0/\so_K^{\times})& \rg_{\dr}(\wh{\sx}_K)\ar[l]^-{\sim}\\
[ \rg_{\crr}(\sx_0/r^{\rm PD}_F)_{\Q_p}]^{\phi=p^r}\ar[d]_-{p_0}^-{\wr}\ar[r]^-{p_p} &  \rg_{\crr}(\sx_0/\so_F^{\times})_{F}\\
[\rg_{\crr}(\sx_0/\so_F^0)_{F}]^{\phi=p^r}\ar@/_20pt/[rrruu]_{\iota_{\hk}^{\prime}}  
}
\end{equation}
 The maps $p_0,p_p$ are defined by sending $T$ to $0,p$, respectively. 
The top triangle defines the overconvergent  Hyodo-Kato morphism $\iota_{\hk}$ as explained in Remark \ref{referee12}, where it is also shown that the maps $p_0,p_p$  from $\overline{\sx}_0$ commute with the ones from $\sx_0$.
The strict quasi-isomorphism between crystalline and convergent cohomology holds because $\sx_0$ is log-smooth over $k^0$. The morphism between de Rham cohomologies is  compatible with Hodge filtrations.

   (ii) {\em Globalization.} 
  We define the functorial map $\iota_1: \rg_{\synt}(X,\Q_p(r))\to\rg^{\rm BK}_{\synt}(\wh{X},\Q_p(r)) $ by   lifting the map (\ref{iota1}) via $\eta$-\'etale descent.

 \subsection{A comparison result}We are now ready to prove our main comparison theorem: 
  \begin{theorem} 
  \label{zamek}
  Let $X$ be a partially proper dagger space over $K$. 
    The map
    $$
    \iota: \rg_{\synt}(X,\Q_p(r))\to \rg_{\synt}(\wh{X},\Q_p(r))
    $$
    is a strict quasi-isomorphism.
    \end{theorem}
    \begin{proof}By the construction of the maps $\iota_1, \iota_2$,  it suffices to show that the following canonical  maps 
    \begin{equation}
    \label{kwaku2}
     \rg_{\dr}(X)\to  \rg_{\dr}(\widehat{X}),\quad  [\rg_{\hk}(X)]^{\phi=p^r}\to  [\rg_{\hk}(\widehat{X})]^{\phi=p^r}
    \end{equation}
    are (filtered) strict quasi-isomorphisms. 
    The first map is an isomorphism induced by the canonical identification of coherent cohomology of a partially proper dagger variety and its rigid analytic avatar \cite[Th. 2.26]{GK0}.
   For the second map, we will show that already the canonical map
    \begin{equation}
    \label{lyon22}
    \rg_{\hk}(X)\to  \rg_{\hk}(\widehat{X})
    \end{equation}
    is a strict quasi-isomorphism.  Our strategy is to pass  to the geometric situation, where we can use the Hyodo-Kato isomorphisms to reduce to the de Rham cohomology. The main difficulty in this approach lies in showing the compatibility of the overconvergent and rigid analytic Hyodo-Kato isomorphisms.\\
    
    (i) {\em Passage to de Rham cohomology.}
    
    We start with the passage to the geometric cohomologies. Since we have compatible strict quasi-isomorphisms (see Proposition \ref{descent} and Proposition \ref{descent1})
    $$
    \rg_{\hk}(X)\stackrel{\sim}{\to} \rg_{\hk}(X_C)^{\sg_K},\quad 
     \rg_{\hk}(\widehat{X})\stackrel{\sim}{\to} \rg_{\hk}(\widehat{X}_C)^{\sg_K},
    $$
  to show that the map (\ref{lyon22}) is a strict quasi-isomorphism,   it suffices to show that so is the canonical map 
    \begin{equation}
    \label{mainmap}
     \rg_{\hk}(X_C)\to \rg_{\hk}(\widehat{X}_C).
     \end{equation} 
    \begin{remark}
    Now, if we were to argue in analogy with the algebraic situation, we would use the following approach: 
    
    (1) we would prove the commutativity of the diagram:
        $$
    \xymatrix@C=.7cm@R=.6cm{
    \rg_{\hk}(X_C)\wh{\otimes}^R_{{F}^{\nr}}C\ar[r]\ar[d]^{\iota_{\hk}}_{\wr}& \rg_{\hk}(\wh{X}_C)\wh{\otimes}^R_{{F}^{\nr}}C\ar[d]^{\iota_{\hk}}_{\wr}\\
  \rg_{\dr}(X_C)\ar[r]^{\sim}& \rg_{\dr}(\wh{X}_C).
    }
    $$
    This is not an easy task since the constructions of the rigid and the crystalline Hyodo-Kato maps are very different.

    (2) 
    The vertical arrows are the Hyodo-Kato quasi-isomorphisms (\ref{HK-two}) and (\ref{HK-rig}) and the bottom arrow is a strict quasi-isomorphism because $X_C$ is partially proper. Hence the top arrow is a strict quasi-isomorphism. The problem is that we do not know how to show  that  this implies the same for the map (\ref{mainmap}). So, below, we use instead the $\ovk$-Hyodo-Kato quasi-isomorphisms. 
    \end{remark}
   Consider  the  diagram
  \begin{equation}
  \label{reduction}
  \xymatrix@R=.6cm{
  \rg_{\hk}(X_C)\ar[r]\ar[d]^{\can} & \rg_{\hk}(\wh{X}_C)\ar[d]^{\can}\\
  \rg_{\hk}(X_C)\wh{\otimes}_{F^{\nr}}\ovk\ar[r]^{\beta}\ar[d]^{\iota_{\hk}}_{\wr} \ar@{-->}@/^10pt/[u]^{\alpha}& \rg_{\hk}(\wh{X}_C)\wh{\otimes}_{F^{\nr}}\ovk\ar[d]^{\iota_{\hk}}_{\wr}\ar@{-->}@/^10pt/[u]^{\hat{\alpha}}\\
  \rg_{\rig,\ovk}(X_C)\ar[r]^{\beta^{\prime}}_{\sim} &\rg_{\conv,\ovk}(\wh{X}_C)\\
  \rg_{\dr}(X)\wh{\otimes}_{K}\ovk\ar[r]^{\sim}\ar[u]^{\wr} & \rg_{\dr}(\wh{X})\wh{\otimes}_{K}\ovk\ar[u]^{\wr}.
  }
  \end{equation}
  The maps $\alpha, \hat{\alpha}$ are the normalized trace maps, natural  left inverses of the canonical vertical maps. The top squares, the dotted and the non-dotted one, commute. The bottom square clearly commutes. Its vertical maps are strict quasi-isomorrphisms by Proposition \ref{etale-descent} and Proposition \ref{pierre1}. The bottom map is a strict quasi-isomorphism because $X$ is partially proper. 
It follows that the map $\beta^{\prime}$ is a strict quasi-isomorphism.   We will show below that the middle square  commutes on the level of ($\wt{H}$-)cohomology. This will imply that the map $\beta$ is a cohomological  isomorphism. This in turn 
 will imply immediately that the map (\ref{mainmap}) is injective on cohomology level; we get its cohomological surjectivity by using the maps $\alpha, \hat{\alpha}$. \\
 
 (ii) {\em Comparison of Hyodo-Kato quasi-isomorphisms.}
  
   Hence, it remains to show that the middle  square in the above diagram  commutes on cohomology level, or that the following diagram commutes
\begin{equation}
  \label{reduction1}
  \xymatrix@R=.6cm@C=.8cm{
 \wt{H}^i( \rg_{\hk}(X_C)\wh{\otimes}_{F^{\nr}}\ovk)\ar[r]\ar[d]^{\iota_{\hk}}_{\wr} & \wt{H}^i(\rg_{\hk}(\wh{X}_C)\wh{\otimes}_{F^{\nr}}\ovk)\ar[d]^{\iota_{\hk}}_{\wr}\\
  \wt{H}^i_{\rig,\ovk}(X_C)\ar[r]^{\sim}& \wt{H}^i_{\conv,\ovk}(\wh{X}_C).
  }
  \end{equation}
We claim that we can assume that $X$ is quasi-compact and argue just on the level of classical cohomology. Indeed,
 write $X$ as an increasing union of quasi-compact open sets $\{U_n\}$, $n\geq 0$. Then we have 
 $$
 \rg_{\hk}(X_C)\wh{\otimes}_{F^{\nr}}\ovk \simeq \holim_n(\rg_{\hk}(U_{n,C}){\otimes}_{F^{\nr}}\ovk).
 $$
 This yields the exact sequence
 $$0\to {H}^1\holim_n(\wt{H}^{i-1}_{\hk}(U_{n,C}){\otimes}_{F^{\nr}}\ovk) \to 
 \wt{H}^i(\rg_{\hk}(X_C)\wh{\otimes}_{F^{\nr}}\ovk)\to {H}^0\holim_n(\wt{H}^i_{\hk}(U_{n,C}){\otimes}_{F^{\nr}}\ovk)\to 0
 $$
 By Proposition \ref{rain15}, the cohomology $\wt{H}^i_{\hk}(U_{n,C})$ is classical and finite rank over $F^{\nr}$. This implies that the cohomology $ \wt{H}^i(\rg_{\hk}(X_C)\wh{\otimes}_{F^{\nr}}\ovk)$ is classical as well and 
 $$
  {H}^i(\rg_{\hk}(X_C)\wh{\otimes}_{F^{\nr}}\ovk)\stackrel{\sim}{\to} {H}^0\holim_n({H}^i_{\hk}(U_{n,C}){\otimes}_{F^{\nr}}\ovk).
 $$
Similarly, we can show that the cohomology $\wt{H}^i_{\conv,\ovk}(\wh{X}_C)$ is classical and we have 
$$
{H}^i_{\conv,\ovk}(\wh{X}_C)\stackrel{\sim}{\to}H^0\holim_n{H}^i_{\conv,\ovk}(\wh{U}_{n,C}).
$$Indeed, arguing as above we get the exact sequence
\begin{equation}
\label{norain1}
0\to {H}^1\holim_n\wt{H}^{i-1}_{\conv,\ovk}(\wh{U}_{n,C})\to 
 \wt{H}^i_{\conv,\ovk}(\wh{X}_C)\to {H}^0\holim_n\wt{H}^i_{\conv,\ovk}(\wh{U}_{n,C})\to 0
 \end{equation}
We note that the prosystems $\{\wt{H}^i_{\conv,\ovk}(\wh{U}_{n,C})\}_{n\in\N}$ and $\{\wt{H}^i_{\rig,\ovk}({U}_{n,C})\}_{n\in\N}$ are equivalent. This follows from the commutative diagram of prosystems
$$
\xymatrix@R=.6cm@C=.8cm{\{\wt{H}^i_{\conv,\ovk}(\wh{U}_{n,C})\}_{n\in\N}\ar[r]^{\sim} &\{\wt{H}^i_{\conv,\ovk}({U}^{\rm o}_{n,C})\}_{n\in\N}\\
\{\wt{H}^i_{\rig,\ovk}({U}_{n,C})\}_{n\in\N}\ar[u]\ar[r]^{\sim} & \{\wt{H}^i_{\rig,\ovk}({U}^{{\rm o},\dagger}_{n,C})\}_{n\in\N}\ar[u]^{\wr}
}
$$
Here $U^{{\rm o},\dagger}$ denotes the rigid analytic space $U^{\rm o}$, the interior of $U$,  equipped with its canonical overconvergent structure.  The horizontal equivalences are clear. The right vertical map is an isomorphism degree by degree because $U^{{\rm o},\dagger}$ is partially proper. This implies that the left vertical map is a an equivalence, as wanted. 

 Now,  the cohomology $\wt{H}^i_{\rig,\ovk}({U}_{n,C})$ is classical and finite rank over $\ovk$ (it is strictly quasi-isomorphic to $H^i_{\dr}(U_n){\otimes}_K\ovk$ by Proposition \ref{pierre1}). Hence the term $H^1\holim_n$ in the exact sequence (\ref{norain1}) vanishes and we get our claim.

 So, from now on, $X$ is quasi-compact and we will show that the diagram (\ref{reduction1}) commutes on the level of classical cohomology. 
We have
$$H^i( \rg_{\hk}(X_C)\wh{\otimes}_{F^{\nr}}\ovk)\simeq H^i_{\hk}(X_C){\otimes}_{F^{\nr}}\ovk,\quad 
H^i( \rg_{\hk}(\wh{X}_C)\wh{\otimes}_{F^{\nr}}\ovk)\simeq H^i_{\hk}(\wh{X}_C){\otimes}_{F^{\nr}}\ovk.
$$
Hence, 
  we are reduced to showing that, for a quasi-compact $X\in {\rm Sm}_K$,  the following diagram commutes
 \begin{equation}
 \label{reduction2}
 \xymatrix@R=.6cm@C=.8cm{
 H^i_{\hk}(X_C)\ar[r]\ar[d]^{\iota_{\hk}}& H^i_{\hk}(\wh{X}_C)\ar[d]^{\iota_{\hk}}\\
  H^i_{\rig,\ovk}(X_C)\ar[r]& H^i_{\conv,\ovk}(\wh{X}_C).
  }
\end{equation}
 Assume first that $X$ has an admissible  semistable weak formal model $\sx$ over $\so_L$, $[L:K]< \infty$,  and consider the diagram
 \begin{equation}
 \label{definition-diagram}
 \xymatrix{
 & \rg_{\rig}(\sx_0/\so_{F_L}^0)\ar@/_40pt/[dddl]\ar[dr]^{\iota_{\hk}}\\
& \rg_{\rig}(\overline{\sx}_0/r^{\dagger}_L)\ar[u]^{p_0}_{\wr}\ar[d]^{f_1}\ar[r]^{p_p}& \rg_{\rig}(\sx_0/\so_{F_L}^{\times})\ar[d]\\
& \rg_{\rig}(\sx_0/r^{\dagger}_L)\ar@/^50pt/[uu]^{p_0}\ar[ru]^{p_p}\ar[d] & \rg_{\conv}(\sx_0/\so_{F_L}^{\times})\ar[d]^{\wr}  \\
\rg_{\crr}(\sx_0/\so_{F_L}^0)_{\Q_p}\ar@/_20pt/[rr]^{\iota_{\hk}}\ar@/^24pt/[r]^{s} & \rg_{\crr}(\sx_0/r^{\rm PD}_L)_{\Q_p}\ar[l]^{p_0}\ar[r]^{p_p} &  \rg_{\crr}(\sx_0/\so_{F_L}^{\times})_{\Q_p}
 }
 \end{equation}
If we remove the section $s$ (and hence also the bottom map $\iota_{\hk}$) the above diagram  commutes.
 For a general  quasi-compact and smooth $X$, take first a homotopy colimit of the above diagram (over $L$) and then glue by $\eta$-\'etale descent. We obtain the following diagram
  \begin{equation}
  \label{reduction4}
 \xymatrix{
 & \rg_{\hk}(X_C)\ar@/_40pt/[dddl]\ar[dr]^{\iota_{\hk}}\\
& \rg_{\rig}(\overline{X}_C/r^{\dagger})\ar[u]^{p_0}_{\wr}\ar[d]^{f_1}\ar[r]^{p_p}_{\sim}& \rg_{\rig}(X_C/\so_F^{\times})\ar[d]\\
& \rg_{\rig}(X_C/r^{\dagger})\ar@/^50pt/[uu]^{p_0}\ar[ru]^{p_p}\ar[d] & \rg_{\conv}(\wh{X}_C/\so_F^{\times})\ar[d]^{\wr}  \\
\rg_{\hk}(\wh{X}_C)\ar@/_20pt/[rr]^{\iota_{\hk}}\ar@/^24pt/[r]^{s} & \rg_{\rm PD}(\wh{X}_C)\ar[l]^{p_0}\ar[r]^{p_p} &  \rg_{\crr}(\wh{X}_C/\so_F^{\times})
 }
 \end{equation}
The notation should be mostly  self-explanatory: the cohomology complexes are defined by the homotopy colimit and the \'etale descent from the corresponding complexes in the diagram (\ref{definition-diagram}) following the procedure used in Section \ref{limit11}.
The groups in the right column are $F^{\nr}$-modules. 

If we remove the section $s$  the above diagram commutes. 
To prove that the diagram (\ref{reduction2}) commutes, 
   by the diagram (\ref{sopot1}),  it suffices to show that so does, on the level of classical cohomology,  the large  round triangle\footnote{That is, the round  triangle with vertices $\rg_{\hk}(X_C)$, $\rg_{\hk}(\wh{X}_C)$, and $\rg_{\rm PD}(\wh{X}_C)$.}, in the diagram (\ref{reduction4}). For that we note that 
   we have the isomorphism 
 \begin{equation}
 \label{kolosept}
   s: H^i_{\hk}(\wh{X}_C)\wh{\otimes}_{F^{\nr}}r^{\rm PD}_{\ovk,\Q_p}\stackrel{\sim}{\to} H^i_{\rm PD}(\wh{X}_C).
 \end{equation}
  If $\wh{X}$ has a quasi-compact semistable formal model $\sx$ over $\so_L$, this arises from the $p^N$-quasi-isomorphism, $N=N(d)$,  (see (\ref{section1}))
   $$
   s: \rg_{\crr}(\sx_0/\so_{F_L}^0)\wh{\otimes}_{\so_{F_L}}r^{\rm PD}_{L}{\to} \rg_{\crr}(\sx_0/r^{\rm PD}_L)
   $$
  and  the fact that $\rg_{\crr}(\sx_0/\so_{F_L}^0)\wh{\otimes}_{\so_{F_L}}r^{\rm PD}_{L}$  is $p$-adically derived complete and $r^{\rm PD}_{L,n}$ is free over $\so_{F_L,n}$. 
   For a general quasi-compact and smooth $\wh{X}$ over $K$, the above argument goes through yielding the isomorphism (\ref{kolosept}), as wanted.

   Now, to show that the round triangle in the diagram (\ref{reduction4}) commutes, consider the ideal
   $$
   I_n:=\Big\{\sum_{i\geq p^n}\tfrac{a_i}{\lfloor i/e\rfloor !}T^i, \lim_{i\mapsto \infty}a_i=0\Big\}.
   $$
   We have the exact sequence
   $$
   0\to I_0\to r^{\rm PD}_{\ovk,\Q_p}\to F^{\nr}\to 0.
   $$
The  $F^{\nr}$-linear and Frobenius equivariant section 
   $s: H^i_{\hk}(\wh{X}_C)\to H^i_{\rm PD}(\wh{X}_C)$
   of the projection $p_0$ satisfies 
   $$
   s(a)=\phi^n\widetilde{\phi^{-n}(a)} \mod H^i_{\hk}(\wh{X}_C)\wh{\otimes}_{F^{\nr}} I_n, \quad a\in H^i_{\hk}(\wh{X}), n\geq 0,
   $$
   where $\widetilde{b}$, for $b\in H^i_{\hk}(\wh{X}_C)$,  is a lifting of $b$ via $p_0$. 
  This is because, for any $a\in H^i_{\hk}(\wh{X}_C)$, we have $ s(a)=\phi^ns(\phi^{-n}(a)) $ and $s(a)= \phi^n\widetilde{\phi^{-n}(a)} \mod H^i_{\hk}(\wh{X}_C)\wh{\otimes}_{F^{\nr}} I_0$. And we also have $\phi^n(I_0)\subset I_n$.
   
   Hence, to show that the large round triangle in the diagram (\ref{reduction4}) commutes, it suffices to show that the intersection of the submodules $H^i_{\hk}(\wh{X}_C)\wh{\otimes}_{F^{\nr}} I_n$, $n\geq 0$, is trivial. 
 But this is clear.
\end{proof}
   \subsection{Overconvergent syntomic cohomology via presentations of dagger structures} In this section we introduce a definition of overconvergent syntomic cohomology using presentations of dagger structures (see \cite[Appendix]{Vez}, Section \ref{pierre11}). We show that so defined syntomic cohomology, a priori different from the one defined in Section \ref{def33}, is strictly quasi-isomorphic to it.

      (i) {\em Local definition.} Let $X$ be a dagger  affinoid over $L=K,C$. Let ${\rm pres}(X)=\{X_h\}$. Define $$\rg_{\synt}^{\dagger}(X,\Q_p(r)):=\hocolim_h \rg_{\synt}(X_h,\Q_p(r)),\quad  r\in\N. $$
     Let $L=K$. We have a natural map 
    \begin{equation}
    \label{def1}
     \iota^{\dagger}_{\synt}\colon \rg_{\synt}^{\dagger}(X,\Q_p(r)) \to \rg_{\synt}(X,\Q_p(r))
    \end{equation}
defined as the composition
\begin{align}
\label{composition1}
 \rg_{\synt}^{\dagger}(X,\Q_p(r)) & =\hocolim_h \rg_{\synt}(X_h,\Q_p(r))\stackrel{\sim}{\to}\hocolim_h \rg_{\synt}(X^0_h,\Q_p(r))\\
  & \stackrel{\sim}{\leftarrow}
 \hocolim_h \rg_{\synt}(X^{0,\dagger}_h,\Q_p(r))\to \rg_{\synt}(X,\Q_p(r)).\notag
\end{align}
The third quasi-isomorphism holds by Theorem \ref{zamek} because $X^0_h$ is partially proper. 
       
     (ii) {\em Globalization.}  For a general smooth dagger variety $X$ over $L$, 
using the  natural equivalence of analytic topoi
$$
{\rm Sh}({\rm SmAff}^{\dagger}_{L,\eet})\stackrel{\sim}{\to} {\rm Sh}({\rm Sm}^{\dagger}_{L,\eet}),
$$
we define the sheaf $\sa_{\synt}^{\dagger}(r)$, $r\in\N$,  on $X_{\eet}$ as the sheaf associated to the presheaf defined by:  $U\mapsto \rg_{\synt}^{\dagger}(U,\Q_p(r))$, $U\in {\rm SmAff}^{\dagger}_{L}$, $U\to X$ an \'etale map.  We define\footnote{We will show below (see Remark \ref{consistence}) that this definition of $\rg_{\synt}^{\dagger}(X,\Q_p(r))$, for a smooth dagger affinoid $X$,  gives an object naturally strictly quasi-isomorphic to the one defined above.} 
$$
\rg_{\synt}^{\dagger}(X,\Q_p(r)):=\rg_{\eet}(X, \sa_{\synt}^{\dagger}(r)),\quad r\in\N.
$$
Globalizing  the  map $\iota^{\dagger}_{\synt}$  from (\ref{def1})
     we obtain a natural map 
    $$
    \iota^{\dagger}_{\synt}:  \rg_{\synt}^{\dagger}(X,\Q_p(r)) \to \rg_{\synt}(X,\Q_p(r)).
    $$
    
     (iii) {\em A comparison quasi-isomorphism.}

    \begin{proposition}    \label{zamek3}
    The  above map $\iota^{\dagger}_{\synt}$ is a strict quasi-isomorphism.
    \end{proposition}
    \begin{proof}By \'etale descent, we may assume that $X$ is a smooth dagger affinoid. 
Looking at the composition (\ref{composition1}) defining the map $\iota^{\dagger}_{\synt}$  we see that it suffices to show that the natural map
    $$
    \hocolim_h \rg_{\synt}(X^{{\rm o},\dagger}_h,\Q_p(r))\to \rg_{\synt}(X,\Q_p(r))
    $$
    is a strict quasi-isomorphism. Or, from the definitions of both sides, that
  we have strict quasi-isomorphisms
     $$
     \rg_{\hk}(X)\stackrel{\sim}{\leftarrow}\hocolim_h \rg_{\hk}(X^{{\rm o},\dagger}_h),\quad \rg_{\dr}(X)\stackrel{\sim}{\leftarrow}\hocolim_h \rg_{\dr}(X^{{\rm o},\dagger}_h).
     $$
     
    This is clear in the case of the second map since this map  factors as
    \begin{equation}
    \label{def2}
    \hocolim_h \rg_{\dr}(X^{{\rm o},\dagger}_h)\stackrel{\sim}{\to}\hocolim_h \rg_{\dr}(X_{h+1})\stackrel{\sim}{\to} \rg_{\dr}(X).
    \end{equation}
  For the first map consider the commutative diagram
    $$
    \xymatrix@R=.6cm@C=.6cm{
     \rg_{\hk}(X)\ar[d]^{\wr} & \hocolim_h \rg_{\hk}(X^{{\rm o},\dagger}_h)\ar[l]\ar[d]^{\wr}\\
      \rg_{\hk}(X_C)^{\sg_K} & \hocolim_h \rg_{\hk}(X^{{\rm o},\dagger}_{h,C})^{\sg_K}\ar[l] \ar[r]^-{\sim}& (\hocolim_h \rg_{\hk}(X^{{\rm o},\dagger}_{h,C}))^{\sg_K}.
    }
     $$
     Here the vertical maps are strict quasi-isomorphisms by Proposition \ref{descent1}.  The horizontal map is a strict quasi-isomorphism because the prosystems $\{\rg_{\hk}(X^{{\rm o},\dagger}_{h,C})\}$ and 
     $\{\rg_{\hk}(X_{h,C})\}$ are equivalent and the  action of $\sg_K$ on the terms of the last one is smooth. It suffices thus to show that the natural map
     $$
       \rg_{\hk}(X_C)\leftarrow \hocolim_h \rg_{\hk}(X^{{\rm o},\dagger}_{h,C})
     $$
     is a strict quasi-isomorphism. For that consider the following diagram 
     $$ \xymatrix@C=.5cm@R=.6cm{
       \rg_{\hk}(X_C)\ar[r]\ar[d] & \hocolim_h \rg_{\hk}(X^{{\rm o},\dagger}_{h,C})\ar[d] \\
     \rg_{\hk}(X_C)\wh{\otimes}_{F^{\nr}}\ovk\ar[d]^{\wr}_{\iota_{\hk}}\ar@/^10pt/[u]^{\alpha}& \hocolim_h \rg_{\hk}(X^{{\rm o},\dagger}_{h,C})\wh{\otimes}_{F^{\nr}}\ovk\ar[l]_-{f_1}\ar[d]^{\wr}_{\iota_{\hk}}\ar@/^10pt/[u]^{\hocolim_h\alpha_h}\\
      \rg_{\rig,\ovk}(X_C) & \hocolim_h \rg_{\rig,\ovk}(X^{{\rm o},\dagger}_{h,C})\ar[l]_-{f_2}\\
       \rg_{\dr}(X)\wh{\otimes}_{K}\ovk \ar[u]^{\wr}_{\beta}& \hocolim_h \rg_{\dr}(X^{{\rm o},\dagger}_{h})\wh{\otimes}_{K}\ovk\ar[l]_-{f_3}\ar[u]^{\wr}_{\hocolim_h\beta_h}
       & \hocolim_h \rg_{\dr}(X_{h}){\otimes}_{K}\ovk\ar[l]_-{\sim} \\
        & & (\hocolim_h \rg_{\dr}(X_{h})){\otimes}_{K}\ovk\ar[u]^{\wr}_{\gamma}\ar[ull]^{\sim}
    }
$$
The maps $\alpha, \alpha_h$ are left inverses of the canonical vertical maps (used already in the diagram (\ref{reduction})). The Hyodo-Kato morphisms are the ones from (\ref{HK-rig}); they are strict quasi-isomorphisms. The maps $\beta, \beta_h$ are those
 from  Proposition \ref{pierre1}; they are strict quasi-isomorphisms as well.  The diagram clearly commutes. The strict quasi-isomorphism $\gamma$ uses the fact that $X_h$ is quasi-compact. It follows that the map $f_3$ is a quasi-isomorphism and then that so is the map $f_1$ and, finally, that so is the top horizontal map, as wanted. 
\end{proof}
\begin{remark}
\label{consistence}
The above proof shows that, for a smooth dagger affinoid $X$ over $K$ with a dagger presentation $\{X_h\}$,  the natural map 
$$
\hocolim_h\rg_{\synt}(X_h,\Q_p(r))\to \rg_{\eet}(X,\sa_{\synt}^{\dagger}(r))
$$
is a strict quasi-isomorphism. Hence the two definitions of $\rg^{\dagger}_{\synt}(X,\Q_p(r))$ that we gave above coincide.
\end{remark}
\section{Arithmetic $p$-adic pro-\'etale cohomology}We pass now to the computation of arithmetic $p$-adic pro-\'etale cohomology of smooth dagger and rigid analytic varieties. 
  \subsection{Syntomic period isomorphisms}First, we will use the comparison theorem between syntomic complexes and $p$-adic nearby cycles from \cite{CN} to define period maps for smooth rigid analytic and dagger varieties.

  Let $\sx$ be a semistable formal model over $\so_K$. 
    Recall that 
     Fontaine-Messing \cite{FM} and  Kato \cite{K1} have constructed  period morphisms  ($i: \sx_0\hookrightarrow \sx, j: \sx_{K}\hookrightarrow \sx$)
$$\alpha^{\rm FM}_{r,n}:  \quad \sss_n(r)_{\sx}  \rightarrow i^*Rj_*{\mathbf Z}/p^n(r)^{\prime}_{\sx_{K}},\quad r\geq 0,
$$
from syntomic cohomology  to  $p$-adic nearby cycles taken as complexes of sheaves on the \'etale site of $\sx_0$. Here  we set  $\Z_p(r)^{\prime}:=\tfrac{1}{p^{a(r)}}\Z_p(r)$,  for $r=(p-1)a(r)+b(r),$ $0\leq b(r)\leq p-1$. The syntomic sheaf $\sss_n(r)$ is associated to the presheaf $\su\mapsto \rg_{\synt}(\su,\Z/p^n(r))$, for formally \'etale $\su\to\sx$. 

 Recall the following comparison result.
    \begin{theorem}{\rm (Colmez-Nizio\l, \cite[Th. 1.1]{CN})}
 \label{main0}
For   $0\leq i\leq r$,  consider the period map 
\begin{equation}
\label{maineq1}
\alpha^{\rm FM}_{r,n}:\quad  \sh^i(\sss_n(r)_{\sx}) \rightarrow i^*R^ij_*{\mathbf Z}/p^n(r)'_{\sx_{K}}.
\end{equation}

{\rm (i)} If $K$ has enough roots of unity\footnote{See \cite[Sec.  2.2.1]{CN} for what it means for a field to contain enough roots of unity. 
For any $K$, the field $K(\zeta_{p^n})$, for $n\geq c(K)+3$, where $c(K)$ is  the conductor of $K$, contains enough roots of unity.} then the kernel  and cokernel of this map are annihilated by $p^{Nr+c_p}$ for a universal constant $N$ {\rm (not depending on $p$, $\sx$, $K$, $n$ or $r$)} and a constant $c_p$ depending only on $p$ (and $d$ if $p=2$).

{\rm(ii)} In general, the kernel  and cokernel of this map are annihilated by $p^N$ for an integer $N=N(e,p,r)$, which depends on $e$, $r$, but not on $\sx$ or $n$.
\end{theorem}
\subsubsection{Rigid analytic varieties.}
The above comparison quasi-isomorphism globalizes easily to smooth rigid analytic varieties:  
    
    \begin{corollary}\label{period1}For $X\in {\rm Sm}_L$,  $L=K,C$,
    the period maps
    $$
    \alpha_r: \R\Gamma_{\synt}(X,\Z_p(r))_{\Q_p}\to \R\Gamma_{\eet}(X,\Q_p(r)),\quad  \alpha_r: \R\Gamma_{\synt}(X,\Q_p(r))\to \R\Gamma_{\proeet}(X,\Q_p(r))
    $$
    are strict quasi-isomorphisms after truncation $\tau_{\leq r}$.
    \end{corollary}
    \begin{proof}Since both the domain and the target of the period maps satisfy $\eta$-\'etale descent we may assume that $X$ has a semistable model over $\so_K$. But in that case this  follows from 
    Theorem \ref{main0} as in analogous claims in the geometric setting in \cite[Prop. 6.1, Cor. 3.46]{CDN3}. 
    \end{proof}
   \subsubsection{Dagger varieties.} 
     The comparison quasi-isomorphism (\ref{maineq1}) can also be extended to smooth dagger varieties.
   Let $X\in {\rm Sm}^{\dagger}_K$, $r\geq 0$.  Define the period map
   \begin{equation}
   \label{period-dagger}
   \alpha_r: \rg_{\synt}(X,\Q_p(r))\to \rg_{\proeet}(X,\Q_p(r))
   \end{equation}
   as the composition
   $$
    \rg_{\synt}(X,\Q_p(r))\stackrel{\sim}{\leftarrow}  \rg_{\synt}^{\dagger}(X,\Q_p(r))\lomapr{\alpha_r^{\dagger}} \rg_{\proeet}(X,\Q_p(r)),
   $$
   where the first map is the map $\iota^{\dagger}_{\synt}$ from Proposition \ref{zamek3} and the second map is defined by globalizing the following map defined for $X$ a dagger affinoid with  presentation $\{X_h\}$:
   $$
    \rg_{\synt}^{\dagger}(X,\Q_p(r))=\hocolim_h\rg_{\synt}(X_h,\Q_p(r))\lomapr{\alpha_r}\hocolim_h\rg_{\proeet}(X_h,\Q_p(r))\simeq 
     \rg_{\proeet}(X,\Q_p(r)).
   $$
    Corollary \ref{period1} implies immediately the following result: 
   \begin{corollary}
   \label{period15}For $X\in {\rm Sm}^{\dagger}_K$, 
    the period map
    $$
    \alpha_r: \R\Gamma_{\synt}(X,\Q_p(r))\to \R\Gamma_{\proeet}(X,\Q_p(r))
    $$
    is a  strict quasi-isomorphism after truncation $\tau_{\leq r}$.
\end{corollary}
   \begin{remark}
   Let $X$ be a smooth partially proper dagger variety over $K$. 
   We claim that the following diagram commutes:
   $$
   \xymatrix@R=.6cm{
  \rg_{\synt}(X,\Q_p(r))\ar[r]^{\alpha_r} \ar[d]^-{\wr}_-{\iota} & \rg_{\proeet}(X,\Q_p(r))\ar[d]^-{\wr}_{\iota_{\proeet}}\\
    \rg_{\synt}(\wh{X},\Q_p(r))\ar[r]^{\wh{\alpha}_r} & \rg_{\proeet}(\wh{X},\Q_p(r))
    }
   $$
   The map $\iota$ is the strict quasi-isomorphism from Theorem \ref{zamek}; the map $\iota_{\proeet}$ is the strict quasi-isomorphism from Proposition \ref{kicia-kicia}. The period maps $\wh{\alpha}_r$, ${\alpha}_r$ are the ones defined above (we put hat above the rigid analytic period map to distinguish it from the dagger period map). 
   
    It suffices to show that this diagram naturally commutes \'etale locally. So we may assume that $X$ is a smooth dagger affinoid. Then checking commutativity is straightforward from the definitions (if tedious). 
   \end{remark}

    \subsection{Applications and Examples} We are now ready to list some applications of our computations and to discuss some examples of computations of $p$-adic pro-\'etale cohomology.
    \subsubsection{Rigid analytic varieties.}
  We start with the rigid analytic case. 
  Let $X\in {\rm Sm}_K$, $r\geq 0$. The distinguished triangle (\ref{seq11}), Lemma \ref{derham}, and the period map $\alpha_r$ above yield a natural map
 $$\partial_r: (\rg_{\dr}(X)/F^r)[-1]\to \rg_{\proeet}(X,\Q_p(r)).$$
    \begin{theorem}\label{rigid1}
    Let $X\in {\rm Sm}_K$, $r\geq 1$. 
    \begin{enumerate}
    \item 
    For $1\leq i\leq r-1$, the map
    $$
    \partial_r: \wt{H}^{i-1}_{\dr}(X)\to \wt{H}^{i}_{\proeet}(X,\Q_p(r))
    $$
    is an isomorphism. In particular, the cohomology $ \wt{H}^{i}_{\proeet}(X,\Q_p(r))$ is not, in  general, classical. 
    \item We have the short exact sequence
    \begin{equation}
    \label{Aseq}
    0\to \wt{H}^{r-1}(\rg_{\dr}(X)/F^r)\lomapr{\partial_r} \wt{H}^r_{\proeet}(X,\Q_p(r))\to \wt{H}^{r}([\rg_{\hk}(X)]^{N=0,\phi=p^r})\to \wt{H}^{r}(\rg_{\dr}(X)/F^r)
    \end{equation}
    \end{enumerate}
    \end{theorem}
    \begin{proof}  Corollary \ref{period1} allows us to pass (by the period map) to syntomic cohomology for which, by 
     Corollary \ref{first1}, we have an analogous claim with $\wt{H}^{r}_{\eet}(X,\sa_{\crr,\Q_p}^{\phi=p^r})$ in place of $\wt{H}^{r}([\rg_{\hk}(X)]^{N=0,\phi=p^r})$. That the latter two are isomorphic follows from diagram (\ref{cr=HK}).  
\end{proof}

    \subsubsection{Dagger varieties.}
 Now we pass to the overconvergent  case.     Let $X\in {\rm Sm}^{\dagger}_K$, $r\geq 0$. 
   The distinguished triangle (\ref{over1}) and the period map $\alpha_r$ from (\ref{period-dagger}) yield a natural map
 $$\partial_r: (\rg_{\dr}(X)/F^r)[-1]\to \rg_{\proeet}(X,\Q_p(r)).$$

  \begin{theorem}\label{first3}
  Let $X\in {\rm Sm}^{\dagger}_K$, $r\geq 1$. 
    \begin{enumerate}
    \item 
    For $1\leq i\leq r-1$, the map
    $$
    \partial_{r}: \wt{H}^{i-1}_{\dr}(X)\to \wt{H}^{i}_{\proeet}(X,\Q_p(r))
    $$
    is an isomorphism. In particular, the cohomology $ \wt{H}^{i}_{\proeet}(X,\Q_p(r))$ is classical. 
    \item We have the long exact sequence
    $$
    0\to \wt{H}^{r-1}(\rg_{\dr}(X)/F^r)\lomapr{\partial_r} \wt{H}^r_{\proeet}(X,\Q_p(r))\to \wt{H}^{r}([\rg_{\hk}(X)]^{N=0,\phi=p^r})\lomapr{\iota_{\hk}}\wt{H}^r(\rg_{\dr}(X)/F^r)
    $$
    \end{enumerate}
    \end{theorem}
    \begin{proof} For $i\leq r$, from the definition of syntomic cohomology and Corollary \ref{period15} 
   we get the long exact sequence
   $$
 \cdots  \to  \wt{H}^{i-1}(\rg_{\dr}(X)/F^r)\to \wt{H}^{i}_{\proeet}(X,\Q_p(r))\to \wt{H}^{i}([\rg_{\hk}(X)]^{N=0,\phi=p^r})\to \wt{H}^{i}(\rg_{\dr}(X)/F^r)\to \cdots
   $$
  For the first claim of the theorem, it suffices to show that, for $i\leq r-1$, $\wt{H}^{i}([\rg_{\hk}(X)]^{N=0,\phi=p^r})=0$ and $\wt{H}^{i-1}_{\dr}(X)\stackrel{\sim}{\to}\wt{H}^{i-1}(\rg_{\dr}(X)/F^r)$. The second isomorphism is clear and the first one follows from Proposition \ref{computHK}.
  
  For the second claim of the theorem, we note that the injectivity on the left is implied by the fact that  $\wt{H}^{r-1}([\rg_{\hk}(X)]^{N=0,\phi=p^r})=0$ (see Proposition \ref{computHK}). 
  
\end{proof}
\subsubsection{Overconvergent balls.} Let $X$ be the overconvergent open  or closed ball over $K$ of dimension $d\geq 0$ and radius $\rho\in\surd |K^{\times}|$.
Using Corollary \ref{period15} and Example \ref{synt-examples} we get
 $$H^r_{\proeet}(X,\Q_p(r))
 \simeq \begin{cases}
\Q_p& \mbox{ if } r=0,\\
 \Omega^{r-1}(X)/\ker d_{r-1}\simeq \Omega^{r}(X)^{d=0} & \mbox{ if } r\geq 2,
\end{cases}
$$
and, for $r=1$, we get a strict exact sequence
$$
0\to \so(X)\to H^1_{\proeet}(X,\Q_p(1))\to \Q_p\to 0.
$$
For comparison, recall that, for the geometric pro-\'etale cohomology, 
we have a topological isomorphism~\cite{CN2}
$$
\Omega^{r-1}(X_C)/\ker d_{r-1}\stackrel{\sim}{\to} H^r_{\proeet}(X_C,\Q_p(r)),\quad r\geq 1.
$$
\subsubsection{Proper smooth rigid analytic varieties.}
Let $X$ be a proper smooth dagger variety over $K$ (recall that every smooth proper rigid analytic variety over $K$ has a canonical dagger structure). For $r\geq 1$, Theorem \ref{first3} and Section \ref{derham1} imply that the cohomology $H^r_{\proeet}(X, \Q_p(r))$ is classical,
we have $$
H^{i-1}_{\dr}(X)\simeq H^i_{\eet}(X,\Q_p(r)),\quad 1\leq i\leq r-1,
$$
 and we have a strict exact sequence (we note that $H^r_{\proeet}(X, \Q_p(r))\simeq H^r_{\eet}(X, \Q_p(r))$)
$$
0\to H^{r-1}_{\dr}(X)\to H^r_{\eet}(X, \Q_p(r))\to E(r)\to 0,
$$
where $E(r)$ is an extension
$$
0\to H^{r-1}_{\hk}(X)^{\phi=p^{r-1}}\to E(r) \to H^r_{\hk}(X)^{N=0,\phi=p^r}\cap \Omega^r(X)\to 0
$$
\subsubsection{The Drinfeld half-space.}
 Let $d\geq 1$ and let ${\mathbb H}^d_K$  be the Drinfeld half-space of dimension $d$, 
i.e., 
$${\mathbb H}^d_K:={\mathbb P}^d_K\setminus \bigcup_{H\in \sh}H,$$ 
where $\sh$ denotes the set of $K$-rational hyperplanes. We set $G:={\rm GL}_{d+1}(K)$. 
For  $1\leq r\leq d$, 
denote by ${\rm Sp}_r(\Q_p)$  
the generalized 
locally constant  Steinberg $\Q_p$-representation of $G$ 
equipped with a trivial action of $\sg_K$ (for a definition see \cite[Sect. 5.2.1]{CDN3}).

\begin{corollary}
\label{drinfeld1}
    \begin{enumerate}
    \item For $0\leq i\leq r$, the cohomology $\wt{H}^i_{\proeet}({\mathbb H}^d_K,\Q_p(r))$ is classical. 
    \item  For $i\leq r-1$, there is a natural $G$-equivariant topological isomorphism
    $$H^i_{\proeet}({\mathbb H}^d_K,\Q_p(r))\simeq {\rm Sp}_{i-1}(K)^*.
    $$
    \item We have a $G$-equivariant diagram of strict exact sequences
    $$
    \xymatrix@R=.4cm@C=.4cm{
    & & & 0\ar[d]\\
     & & & {\rm Sp}_{r-1}(\Q_p)^*\ar[d]\\
0\ar[r] & \Omega^{r-1}({\mathbb H}^d_K)/\im d_{r-2} \ar[r] & H^r_{\proeet}({\mathbb H}^d_K,\Q_p(r))\ar[r] & {\rm E}(\Q_p) \ar[r] \ar[d]&  0\\
 & & &  {\rm Sp}_r(\Q_p)^* \ar[d]\\
 & & & 0
}
$$
 \end{enumerate}
    \end{corollary}
\begin{proof}Point (2) follows from Theorem \ref{first3} and the computations of Schneider-Stuhler \cite{SS} of the de Rham cohomology of the Drinfeld 
half-space: $\wt{H}^i_{\dr}({\mathbb H}^d_K)\simeq {\rm Sp}_{i}(K)^*$. 

For point (3), since ${\mathbb H}^d_K$ is Stein, by Section \ref{derham1}, we have 
$$
\wt{H}^{r-1}(\rg_{\dr}({\mathbb H}^d_K)/F^r)\simeq \Omega^{r-1}({\mathbb H}^d_K)/\im d_{r-2},\quad \wt{H}^r(\rg_{\dr}({\mathbb H}^d_K)/F^r)\simeq 0.
$$
On the other hand, from (\ref{jesien}) we get an exact sequence
\begin{equation}
\label{long-day}
0\to \wt{H}^{r-1}_{\hk}({\mathbb H}^d_K)^{\phi=p^{r-1}}\to \wt{H}^r([\rg_{\hk}({\mathbb H}^d_K)]^{N=0,\phi=p^r})\to \wt{H}^{r}_{\hk}({\mathbb H}^d_K)^{N=0,\phi=p^{r}}\to 0,
\end{equation}
where all the cohomologies are classical. But, by \cite[Lemma 5.11]{CDN3}, we have a $G$-equivariant isomorphism $ \wt{H}^{i}_{\hk}({\mathbb H}^d_K)^{\phi=p^{i}}\simeq 
{\rm Sp}_i(\Q_p)^*$. Since the monodromy is trivial (see \cite[Sect. 5.5]{CDN3}),  (\ref{long-day}) then yields  an exact sequence
$$
0\to {\rm Sp}_{r-1}(\Q_p)^*\to \wt{H}^r([\rg_{\hk}({\mathbb H}^d_K)]^{N=0,\phi=p^r}))\to {\rm Sp}_r(\Q_p)^*\to 0
$$
Plugging the above computations into Theorem \ref{first3} and setting $E(\Q_p):={H}^r([\rg_{\hk}({\mathbb H}^d_K)]^{N=0,\phi=p^r}))$ we get point (2).

 Point (1) follows now trivially from points (2) and (3). 
\end{proof}
\begin{remark}\begin{enumerate}
\item 
We note that we have the strict exact sequence
$$
0\to H^{r-1}_{\dr}({\mathbb H}^d_K)\to \Omega^{r-1}({\mathbb H}^d_K)/\im d_{r-2}\lomapr{d_{r-1}}\Omega^r({\mathbb H}^d_K)^{d=0}\to H^{r}_{\dr}({\mathbb H}^d_K)\to 0
$$
and that the two de Rham cohomology terms are topologically isomorphic to ${\rm Sp}_{r-1}(K)^*$ and ${\rm Sp}_r(K)^*$, respectively.
\item
It would be interesting to understand the computations in this example better. In particular, to describe the extensions of Steinberg representations that appear. 
\end{enumerate}
\end{remark}
\begin{remark}
It is interesting to link the computation of the arithmetic cohomology $H^i_{\proeet}({\mathbb H}^d_K,\Q_p(r))$ presented here to the computation of the geometric cohomology $H^i_{\proeet}({\mathbb H}^d_C,\Q_p(r))$ done in \cite[Th. 5.15]{CDN3}.  The following argument would need to be made more precise but it shows that the two computations, the arithmetic and the geometric one, are compatible.

  We have the Hochschild-Serre spectral sequence 
  \begin{equation}
\label{jussieu1}
H^n(\sg_K,H^{i-n}_{\proeet}({\mathbb H}^d_C,\Q_p(r)))\Longrightarrow H^i_{\proeet}({\mathbb H}^d_K,\Q_p(r)).
\end{equation}
(Only $n=0,1,2$ can possibly give a nonzero contribution.)
Now, the exact sequence from \cite[Th. 5.15]{CDN3} twisted by $(j-k)$, yields an exact sequence of $\sg_K\times G$-modules
$$0\to C(j-k)\widehat{\otimes}_K(\Omega^{k-1}({\mathbb H}^d_K)/\ker d_{k-1})\to H^{k}_{\proeet}({\mathbb H}^d_C,\Q_p(j))\to
{\rm Sp}_k(\Q_p)^*(j-k)\to 0.$$
   Hence the computation of $H^n(\sg_K,H^{i-n}_{\proeet}({\mathbb H}_C,\Q_p(r)))$ will
involve the groups $H^n(\sg_K,\Q_p(r-i+n))$ and $H^n(\sg_K,C(r-i+n))$.

Recall the following results of Tate and Bloch-Kato:
\begin{align}
\label{jussieu2}
&H^0(\sg_K,\Q_p(j))\simeq 
\begin{cases} \Q_p&{\text{if $j=0$,}}\\ 0 &{\text{if $j\geq 1$}},\end{cases}
\quad H^1(\sg_K,\Q_p(j))\simeq
\begin{cases} K\oplus\Q_p&{\text{if $j=1$,}}\\ K &{\text{if $j\geq 2$}},\end{cases}\\\notag
&H^2(\sg_K,\Q_p(j))=0,\ {\text{if $j\geq 2$}},\\\notag
&H^0(\sg_K,C(j))\simeq
\begin{cases} K&{\text{if $j=0$}},\\ 0 &{\text{if $j\geq 1$}},\end{cases}
\quad H^1(\sg_K,C(j))\simeq
\begin{cases} K&{\text{if $j=0$}},\\ 0 &{\text{if $j\geq 1$}},\end{cases}\\\notag
&H^2(\sg_K,C(j))=0\ {\text{if $j\geq 0$}}.
\end{align}
Using them, we see that the nonzero terms of the spectral sequence (\ref{jussieu1}) contributing to $ H^i_{\proeet}({\mathbb H}^d_K,\Q_p(r))$, $i\leq r$, are the following:
\begin{align*}
0\to \Omega^{i-1}({\mathbb H}^d_K)/\ker d_{i-1}\to &H^0(\sg_K,H^{i}_{\proeet}({\mathbb H}^d_C,\Q_p(r)))\to
{\rm Sp}_i(\Q_p)^*\to 0,\quad{\text{if $i=r$}};\\
&H^1(\sg_K,H^{i-1}_{\proeet}({\mathbb H}_C,\Q_p(r)))\simeq (K\oplus\Q_p)\otimes_{\Q_p}{\rm Sp}_{i-1}(\Q_p)^*,
\quad{\text{if $i=r$}};\\
&H^1(\sg_K,H^{i-1}_{\proeet}({\mathbb H}_C,\Q_p(r)))\simeq K\otimes_{\Q_p}{\rm Sp}_{i-1}(\Q_p)^*\simeq {\rm Sp}_{i-1}(K)^*,
\quad{\text{if $i\leq r-1$}}.
\end{align*}
Here the top sequence is exact though  (\ref{jussieu2}) is not enough to ensure the surjectivity of the map $H^0(\sg_K,H^{i}_{\proeet}({\mathbb H}^d_C,\Q_p(r)))\to
{\rm Sp}_i(\Q_p)^*$. It yields however the exact sequence
$$
H^0(\sg_K,H^{i}_{\proeet}({\mathbb H}^d_C,\Q_p(r)))\to
{\rm Sp}_i(\Q_p)^*\lomapr{\partial} \,\Omega^{i-1}({\mathbb H}^d_K)/\ker d_{i-1}
$$
Now the boundary map $\partial$ is trivial by a representation theory argument: the map $\partial$ is continuous  and $G$-equivariant,
the $G$-smooth vectors are dense in ${\rm Sp}_i(\Q_p)^*$, but
$\Omega^{i-1}({\mathbb H}^d_K)/\ker d_{i-1}$ does not have any nonzero $G$-smooth elements since it injects into $\Omega^{i}({\mathbb H}^d_K)$.

   Hence, for $0\leq i\leq r-1$, we get $ H^i_{\proeet}({\mathbb H}^d_K,\Q_p(r))\simeq {\rm Sp}_{i-1}(K)^*$ as in Corollary \ref{drinfeld1}. For $i=r$, we get the diagram of exact sequences
   $$
   \xymatrix@R=.4cm@C=.4cm{
    & & & 0\ar[d]\\
  & & & \Omega^{r-1}({\mathbb H}^d_K)/\ker d_{r-1}\ar[d]  \\
 0\ar[r] &  {\rm Sp}_{r-1}(K)^*\oplus{\rm Sp}_{r-1}(\Q_p)^* \ar[r] &  H^r_{\proeet}({\mathbb H}^d_K,\Q_p(r))\ar[r] &  H^0(\sg_K,H^{r}_{\proeet}({\mathbb H}^d_C,\Q_p(r)))\ar[r]\ar[d]  &  0\\
  & & & {\rm Sp}_r(\Q_p)^*\ar[d] \\
    & & & 0
    }
   $$
    To compare this with Corollary \ref{drinfeld1}, note that we have an exact sequence
$$0\to H^{i-1}_{\dr}({\mathbb H}^d_K)\to\Omega^{i-1}({\mathbb H}^d_K)/{\rm im}\,d_{i-2}\to
\Omega^{i-1}({\mathbb H}^d_K)/\ker d_{i-1}\to 0$$
and the Schneider-Stuhler  isomorphism
$$H^{i-1}_{\dr}({\mathbb H}^d_K)\cong {\rm Sp}_{i-1}(K)^*.$$ Hence Corollary \ref{drinfeld1} and the above computation via Galois descent give us the same Jordan-H\"older components of $ H^r_{\proeet}({\mathbb H}^d_K,\Q_p(r))$ but they are put together in two different ways.
\end{remark}


\begin{thebibliography}{leBras}
   \bibitem{Be1} {A.~Beilinson}, {\em  $p$-adic periods and derived de Rham cohomology}. J. AMS~{ 25} (2012), 715--738. 
   \bibitem{BE2} {A.~Beilinson}, {\em On the crystalline period map}. Cambridge J. Math.~{1} (2013), 1--51. References to the post-publication arXiv:1111.3316v4. 
   \bibitem{BhS} B.~Bhatt, P.~Scholze, {\em The pro-\'etale topology for schemes}, Ast\'erisque~{369} (2015), 99-201.
   \bibitem{BS} B.~Bhatt, A.~Snowden, {\em Refined alterations}. Preprint, 2017.
   \bibitem{Ray} S.~Bosch,  W.~L\"utkebohmert, {\em  Formal and rigid geometry. I. Rigid spaces}. Math. Ann. 295 (1993), 291--317.
   \bibitem{FRG} S.~Bosch,  W.~L\"utkebohmert,  {\em Formal and rigid geometry. II. Flattening techniques}. 
Math. Ann. 296 (1993), 403--429.
    \bibitem{CS} B.~Chiarellotto, B.~ Le Stum, {\em  Pentes en cohomologie rigide et F-isocristaux unipotents}. Manuscripta Math. 100 (1999), 455--468. 
    \bibitem{CDN1} P.~Colmez, G.~Dospinescu, W.~Nizio\l, 
{\em Cohomologie $p$-adique de la tour de Drinfeld: le cas de la dimension 1.} arXiv:1704.08928 [math.NT], to appear in J. Amer. Math. Soc.
   \bibitem{CDN3}  P.~Colmez, G.~Dospinescu, W.~Nizio\l, {\em Cohomology of $p$-adic Stein spaces}, arXiv:1801.06686 [math.NT],  to appear in Invent. Math.
   \bibitem{CN} P.~Colmez, W.~Nizio\l, {\em Syntomic complexes and $p$-adic nearby cycles}. Invent. Math. 208 (2017), 1-108.
   \bibitem{CN2} P.~Colmez, W.~ Nizio\l,  
{\em  On the cohomology of the affine space.} arXiv:1707.01133 [math.AG], to appear in $p$-adic Hodge theory (2017), Simons Symposia,
Springer-Verlag.
    \bibitem{CN4} P.~Colmez, W.~Nizio\l, {\em On $p$-adic comparison theorems for rigid analytic varieties, II}. In preparation. 
\bibitem{DN} F.~D\'eglise, W.~Nizio{\l}, {\em On $p$-adic absolute Hodge cohomology and syntomic coefficients, I}, Comment. Math. Helv. 93 (2018), 71--131.
\bibitem{Elk} R.~Elkik, {\em Solutions d'\'equations \`a coefficients dans un anneau hens\'elien}. Ann. Sci. \'Ecole Norm. Sup. 6 (1973), 553--603. 
   \bibitem{EY} V.~Ertl, K.~Yamada, {\em Comparison between rigid syntomic and crystalline syntomic cohomology for strictly semistable log schemes with boundary}. 	
   arXiv:1805.04974 [math.NT].
   \bibitem{ET1}  J.-Y.~Etesse, {\em Rel\`evements de sch\'emas et alg\`ebres de Monsky-Washnitzer: th\'eor\`emes d'\'equivalences et de pleine fid\'elit\'e. II.} Rend. Semin. Mat. Univ. Padova 122 (2009), 205--234. 
    \bibitem{Fal} G.~Faltings, {\em $p$-adic Hodge theory.} J. Amer. Math. Soc. 1 (1988), 255--299. 
    \bibitem{FvP}  J.~Fresnel, M.~van der Put, {\em  Rigid analytic geometry and its applications}. Progr. Math. 218, Birk\"auser, 2004. 
     \bibitem{FM} {J.-M.~Fontaine},  W.~Messing,  {\em  $p$-adic periods and $p$-adic \'{e}tale
cohomology}. Current Trends in Arithmetical Algebraic Geometry
(K.~Ribet, ed.), Contemporary
Math., vol.~{67}, AMS, Providence, 1987, 179--207.
\bibitem{GK0} E.~Grosse-Kl\"onne, {\em Rigid analytic spaces with overconvergent structure sheaf.} J. Reine Angew. Math. 519 (2000), 73--95. 
\bibitem{GKFr}  E.~Grosse-Kl\"onne, {\em Frobenius and monodromy operators in rigid analysis, and Drinfeld's symmetric space}. J. Algebraic Geom. 14 (2005), 391--437. 
    \bibitem{Urs} U.~Hartl, {\em  Semi-stable models for rigid-analytic spaces}. Manuscripta Math. 110 (2003), 365--380.
    \bibitem{Hub} R.~Huber, {\em \'Etale cohomology of rigid analytic varieties and adic spaces.} Aspects of Mathematics, E30. Friedr. Vieweg \& Sohn, Braunschweig, 1996. 
    \bibitem{HK} O.~Hyodo, K.~Kato, {\em Semi-stable reduction and crystalline cohomology with logarithmic poles}. Ast\'erisque 223 (1994), 221--268. 
    \bibitem{IL} L.~Illusie, {\em On the category of sheaves of objects of $\sd(R)$ (after Beilinson and Lurie)}, notes, 2013.
    \bibitem{K1} {K.~Kato}, {\em Semistable reduction and $p$-adic \'etale
cohomology}. Ast\'erisque~{223} (1994), 269--293.
 \bibitem{LA} A.~Langer, A.~Muralidharan, {\em An analogue of Raynaud's theorem: weak formal schemes and dagger spaces.} M\"unster J. Math. 6 (2013), 271--294.
 \bibitem{LeBras} A.-C.~Le Bras, {\em Overconvergent de Rham cohomology over the Fargues-Fontaine curve}. arXiv:1801.00429v2 [math.NT].
 \bibitem{Lu1} J.~Lurie, {\em Higher topos theory}, vol. 170, Annals of Mathematics Studies, Princeton University Press, 2009.
\bibitem{Lu2} J.~Lurie, {\em Higher Algebra}, preprint. 
\bibitem{NN} {J.~Nekov\'a\v{r}}, {W.~Nizio\l}, {\em Syntomic cohomology and $p$-adic regulators for varieties over $p$-adic fields}.  Algebra Number Theory 10 (2016), 1695--1790.
\bibitem{Ni} W.~Nizio\l, {\em On the image of $p$-adic regulators}. Invent. Math. 127 (1997), 375--400.
\bibitem{Og} A.~Ogus, {\em $F$-crystals on schemes with constant log structure.}  Compositio Math. 97 (1995), no. 1-2, 187--225.  
  \bibitem{OM} M.~Olsson, {\em Crystalline cohomology of algebraic stacks and Hyodo-Kato cohomology.}  Ast\'erisque  316 (2007).
\bibitem{Sai} T.~Saito, {\em Weight spectral sequences and independence of $\ell$}. 
J. Inst. Math. Jussieu 2 (2003), 583--634. 
 \bibitem{SS} P.~Schneider, U.~Stuhler, {\em The cohomology of $p$-adic symmetric spaces}. Invent. Math. 105 (1991), 47--122. 
 \bibitem{Sch} P.~Scholze, {\em $p$-adic Hodge theory for rigid-analytic varieties}. Forum of Mathematics, Pi, 1, e1 2013.
 \bibitem{ScE}  P.~Scholze, {\em $p$-adic Hodge theory for rigid-analytic varieties: corrigendum}. Forum Math. Pi 4 (2016), e6, 4 pp.
\bibitem{Tem} M.~Temkin, {\em Altered local uniformization of Berkovich spaces}. Israel J. Math. 221 (2017), 585--603.
\bibitem{Ts} T.~Tsuji, {$p$-adic \'etale cohomology and crystalline cohomology in the semi-stable reduction case},
 Invent. Math. 137 (1999), 233-411.
   \bibitem{V2} J.-L.~Verdier, {\em Fonctorialit\'e de cat\'egories de faisceaux, Th\'eorie des topos et cohomologie \'etale de sch\'emas (SGA~4)}, Tome 1, Lect. Notes in Math. 269, Springer-Verlag, 1972, pp. 265--298.
   \bibitem{Vez} A.~Vezzani, {\em The Monsky-Washnitzer and the overconvergent realizations}.  Int. Math. Res. Not. IMRN 2018, 3443--3489.
   \bibitem{Yam} K.~Yamada, {\em Log rigid syntomic cohomology for semistable schemes.},  arXiv:1505.03878v4 [math.AG]. 
     \bibitem{Zhe} W.~Zheng, {\em Note on derived $\infty$-categories and monoidal structures}, notes, 2013.
\end{thebibliography}
 \end{document}